\documentclass[final,10pt,reqno]{amsart}

\usepackage{comment} 

\usepackage[left=0.8in, right=0.8in, top=1.5in, bottom=1.5in]{geometry}

\usepackage[foot]{amsaddr} 
\usepackage{stmaryrd}

\usepackage[utf8]{inputenc}
\usepackage[T1]{fontenc}
\usepackage{calligra}
\usepackage{enumerate}
\usepackage{dsfont}
\usepackage{mathrsfs}
\usepackage[all]{xy}
\usepackage{mathrsfs}
\usepackage{bm}
\usepackage{empheq}

\usepackage{verbatim}
\usepackage{amssymb}
\usepackage{mathrsfs}
\usepackage{amsthm}
\usepackage{amsmath}
\usepackage{amsfonts}
\usepackage{latexsym}
\usepackage{mathtools} 
\usepackage{enumitem} 
\usepackage[english]{babel} 
\usepackage{mathabx}
\usepackage{scalerel,stackengine}

\usepackage{float}

\usepackage[pdftex,usenames,dvipsnames]{xcolor}  

\definecolor{dkblue}{RGB}{30,90,140} 
\usepackage[colorlinks=true, pdfstartview=FitV, linkcolor=dkblue, citecolor=dkblue, urlcolor=dkblue]{hyperref}    
\definecolor{mydarkbluett}{RGB}{12,111,174}

\usepackage[colorinlistoftodos,prependcaption,textsize=tiny]{todonotes}

\stackMath
\newcommand\reallywidehat[1]{%
\savestack{\tmpbox}{\stretchto{%
  \scaleto{%
    \scalerel*[\widthof{\ensuremath{#1}}]{\kern-.6pt\bigwedge\kern-.6pt}%
    {\rule[-\textheight/2]{1ex}{\textheight}}
  }{\textheight}%
}{0.5ex}}%
\stackon[1pt]{#1}{\tmpbox}%
}

\emergencystretch=\maxdimen
\theoremstyle{plain}
\newtheorem{teor}{Theorem}
\newtheorem{obs}[teor]{Remark}
\newtheorem{prop}[teor]{Proposition}
\newtheorem{coro}[teor]{Corollary}
\newtheorem{lemma}[teor]{Lemma}
\newtheorem{defi}[teor]{Definition}

\numberwithin{equation}{section}
\numberwithin{teor}{section}

\newcommand{\norma}[1]{{\left\vert\kern-0.25ex\left\vert\kern-0.25ex\left\vert #1 
    \right\vert\kern-0.25ex\right\vert\kern-0.25ex\right\vert}}
\def\R{\mathbb{R}}

\def\N{\mathbb{N}_0}
\def\T{\mathbb{T}}

\def\Z{\mathbb{Z}}

\addto\extrasenglish{
  
}

\newcommand{\mc}{\mathcal}
\newcommand{\what}{\widehat}

\newcommand{\dx}{ \, {\rm d} x}
\newcommand{\dt}{ \, {\rm d} t}
\newcommand{\ds}{ \, {\rm d} s}
\newcommand{\dy}{ \, {\rm d} y}
\newcommand{\dz}{ \, {\rm d} z}

\renewcommand{\P}{\mathbb{P}}

\DeclareMathOperator*{\argmin}{{\rm argmin}}

\newcommand{\dimitri}[1]{\textcolor{green}{[***Di: #1 ***]}}
\newcommand{\dani}[1]{\textcolor{blue}{[***Da: #1 ***]}}

\begin{document}

\title[Effective Dynamics and Blow up in a Model of Magnetic Relaxation]{Effective Dynamics and Blow up in a Model of Magnetic Relaxation}

\author[D. Cobb]{Dimitri Cobb$^1$}
\address{$^{1}$University of Bonn, Mathematisches Institut. Endenicher Allee 60, D-53115, Bonn, Germany.}
\email{\href{mailto:cobb@math.uni-bonn.de}{cobb@math.uni-bonn.de}}

\author[D.S\'anchez-Sim\'on del Pino]{Daniel S\'anchez-Sim\'on del Pino$^2$}
\email{\href{mailto:sanchez@iam.uni-bonn.de}{sanchez@iam.uni-bonn.de}}

\author[J. J. L. Vel\'azquez]{Juan J. L. Vel\'azquez$^2 $}
\address{$^{2}$University of Bonn, Institut für Angewandte Mathematik. Endenicher Allee 60, D-53115, Bonn, Germany.}
\email{\href{mailto:velazquez@iam.uni-bonn.de}{velazquez@iam.uni-bonn.de}}

\vspace{.2cm}
\date\today
\begin{abstract}
    In this article we study a one dimensional model for Magnetic Relaxation. This model was introduced by Moffatt \cite{Moffatt2015} and describes a low resistivity viscous plasma, in which the pressure and the inercia are much smaller than the magnetic pressure. In the limit of resistivity $\varepsilon\rightarrow 0$, we prove the existence of two time scales for the evolution of the magnetic field: a fast one for times of order $\log(\varepsilon^{-1})$ in which the resistivity plays no role and the energy is dissipated only via viscosity; and a slow one for times of order $\varepsilon^{-1}$ characterized by the influence of the resistivity. We show that in this second time scale, as $\varepsilon\rightarrow 0$, the modulus of magnetic field approaches a function that depends only on time. We also prove that, in this regime, the magnetic field $b_\varepsilon(t,x)$ can be approximated as $\varepsilon \rightarrow 0$ by the solution  of a PDE whose solutions exhibit blow up for some choices of initial data. 
\end{abstract}
\maketitle

{\footnotesize{\textbf{Key words:} Magnetic Relaxation, Magnetohydrodynamics, Effective Dynamics, Singular Perturbation of PDEs, Multi-Scale Analysis.}}

\smallskip

{\footnotesize{\textbf{Mathematics Subject Classification (MSC2020):  35Q35 (main), 76W05, 35B25, 35B40, 35B44 (secondary).}}
\normalsize
\tableofcontents

\section{Introduction}

In this paper we study the following system of PDEs
\begin{equation}\label{eq:magneticrelaxation}
    \begin{cases}
        \partial_t b+\partial_x (ub)=\varepsilon \partial_x^2 b, \\
        \partial_x^2u = \partial_x\left( \frac{1}{2} |b|^2 \right), 
    \end{cases}
\end{equation}
where $b:\R\times\R\rightarrow \R^2$. This model describes the behaviour of the magnetic field generated by a compressible conducting fluid in the regime of low electrical resistivity, which is especially relevant for low density plasmas driven by the magnetic field, such as the solar corona \cite{Moffatt2016}. 

The system \eqref{eq:magneticrelaxation} can be derived from the 3D compressible Magnetohydrodynamic equations (MHD) in a regime dominated by viscosity and after approximating the pressure and the inertia as negligible contributions (\textit{c.f.} Subsection \ref{ss:derivation}).
For the magnetofluids we consider, the velocities \textbf{u} and magnetic fields \textbf{B} of the form
\begin{equation}\label{ieq:Usymmetry}
    \textbf{u} = (u, 0, 0) \qquad \text{and} \qquad \textbf{B} = (0, b_1, b_2).
\end{equation}
These fields depend only on time and the first space variable, $x_1$, that we  will denote by $x$ from now on, which we take as an element in the 1D torus $x \in \R / \Z$ with unit measure (\textit{i.e.} we assume periodic boundary conditions). We make this choice of the domain for the sake of simplicity, to avoid difficulties arising from the boundary conditions.  The non-dimensional parameter $\varepsilon > 0$ is proportional to the electrical resistivity of the fluid. 

\medskip

While this is a simplification of the MHD equations, we will see that this model has a non-trivial behaviour, especially in the regime $\varepsilon \rightarrow 0^+$. Originally, equations \eqref{eq:magneticrelaxation} were introduced by H. K. Moffatt \cite{Moffatt2015} as a simplified model in order to obtain some insight about the phenomenon of magnetic relaxation behaviour in a magnetofluid. The concept of magnetic relaxation originates in a 1958 paper \cite{Woltjer} of L. Woltjer. This physical conjecture, which was reformulated by J. B. Taylor \cite{Taylor1974, Taylor1986}, is based on the observation that in ideal MHD, both the total energy and the magnetic helicity are conserved,
\begin{equation*}
    E = \int \Big( |\textbf{u}|^2 + |\textbf{B}|^2 + \mc H(\rho) \Big) \dx \qquad \text{and} \qquad H_m = \int \textbf{B} \cdot \textbf{A} \dx.
\end{equation*}
In this equation, $\textbf{A}$ is the vector magnetic potential and $\mc H(\rho) = \rho \int_1^\rho \frac{P(z)}{z^2}{\rm d}z$ is the contribution of the pressure to the energy (the state equation $P = P(\rho)$ is assumed). The energy and the magnetic helicity scale differently with regards to the magnetic field, \textit{i.e.} one scales like an $L^2$ norm, whereas the second as an $\dot{H}^{-1/2}$ norm. Furthermore in the situation where the viscosity $\kappa$ of the fluid is much larger than the resistivity $\kappa$ (in other words the magnetic Prandtl number is large $Pm = \kappa/\eta \gg 1$), the time scale of decay of the energy is much shorter than the time scale of evolution of the magnetic helicity. Accordingly, Taylor predicted that in a large magnetic Prandtl number regime, the magnetic field should be driven to a state of minimal energy while the magnetic helicity is roughly conserved, that is
\begin{equation}\label{ieq:Minimoze}
    \textbf{B} \approx \argmin_{H_m = H_m(t=0)} \int \frac{1}{2} |\textbf{B}|^2 \dx.
\end{equation}
These Taylor states correspond to force-free fields (or Beltrami states), and satisfy the eigenvalue equation
\begin{equation}\label{eq:Beltrami}
    \textbf{B} = \lambda \nabla \times \textbf{B} \qquad \text{for some } \lambda \in \R.
\end{equation}
The number $\lambda$ should be understood as a Lagrange multiplier associated to the constant helicity constraint in the minimization problem \eqref{ieq:Minimoze}. On the longer time scale related to the evolution of the magnetic helicity, the magnetic field is therefore expected to evolve while remaining close to a Beltrami state, until all available energy finally dissipates due to resistivity. In \cite{Moffatt1969}, H. K. Moffatt pointed out that in the minimization problem \eqref{ieq:Minimoze}, one also has to account for the conservation of all partial magnetic helicities, which are the integrals of the form 
\begin{equation}\label{ieq:HmConstraintSSS}
H_m(V) := \int_{V}\textbf{A}\cdot \textbf{B}\, \dx,
\end{equation}
where $V$ is any subdomain so that $\textbf{B}$ is normal to $\partial V$. In other words, the minimization problem \eqref{ieq:Minimoze} should take into account an infinite number of constraints, so that the resulting force-free state solves the equation
\begin{equation*}\label{ieq:ForceFree}
    \textbf{B} = \lambda(t, x) \nabla \times \textbf{B}, \qquad \text{for some function } \lambda(t, x),
\end{equation*}
and the function $\lambda(t, x)$ should be understood as a family of Lagrange multipliers associated to the partial helicity constraints \eqref{ieq:HmConstraintSSS}. 

\medskip

In general, the phenomenon of magnetic relaxation is far from being understood, and it is to this point not clear whether it occurs in the terms proposed by Taylor. In rigorous mathematical terms, it is known that when we take the resistivity tending to zero, the total magnetic helicity is conserved for solutions of sufficient regularity \cite{FL2020, BT}. On the other hand, it has also been shown, by means of techniques from convex integration, that some finite energy solutions of the MHD equations do not conserve magnetic helicity \cite{BBV}. In the physical literature, there exist different numerical attempts to prove or disprove Taylor's conjecture, \cite{MG, SCTSDB, YK, YRH}, as well as theoretical arguments suggesting that magnetic relaxation might occur (see in particular \cite{QLLS} for a formal argument based on the $L^2$ and $\dot{H}^{-1/2}$ scaling of $E$ and $H_m$ respectively). For a general introduction about the theory of Magnetic Relaxation, we refer to the lecture notes \cite{Yeates} and Chapter 11 in the book \cite{Bellan}.

\medskip

The system \eqref{eq:magneticrelaxation} related to a larger class of system of PDEs envisioned by Moffatt \cite{Moffatt1969, Moffatt2021} to describe the process of Magnetic Relaxation, namely Magneto-Stokes equations. These equations read, for an incompressible fluid,
\begin{equation}\label{ieq:MagnetoStokes}
    \begin{cases}
        \partial_t \textbf{B} + (\textbf{u} \cdot \nabla) \textbf{B} = (\textbf{B} \cdot \nabla) \textbf{u} \\
        (- \Delta)^\alpha \textbf{u} + \nabla \pi = (\textbf{B} \cdot \nabla)\textbf{B}\\
        {\rm div}(\textbf{u}) = 0.
    \end{cases}
\end{equation}
Here, $\pi = P + \frac{1}{2}|\textbf{B}|^2$ is the MHD pressure. Our system \eqref{eq:magneticrelaxation} is a particular case of magneto-Stokes system, albeit for a compressible fluid. The case $\alpha = 1$ is that of a usual Newtonian fluid. The advantage of equations \eqref{ieq:MagnetoStokes} is that they always feature decay of the magnetic energy through viscous dissipation, while preserving the magnetic helicity as a constant:
\begin{equation*}
    \frac{\rm d}{\dt} \int |\textbf{B}|^2 \dx + \int |(- \Delta)^{\alpha/2} \textbf{u}|^{2} \dx = 0, \qquad \text{and} \qquad H_m = H_m(t = 0).
\end{equation*}
Therefore, it is expected that solutions of \eqref{ieq:MagnetoStokes} converge to a force-free state \eqref{ieq:ForceFree} with magnetic helicity $H_m(t = 0)$ \cite{Moffatt1985}. Proving this convergence has not been rigorously done, and for $\alpha=1$, even in 2D existence of global (weak) solutions is still unknown. We refer to \cite{BKS,BFV, KK, Tan} for the mathematical analysis of \eqref{ieq:MagnetoStokes} with different ranges of $\alpha$.

\medskip

From a wider perspective, the study of the long time behaviour and the asymptotics of the solutions to the MHD equations is an active field of research. Some effort has been put in studying the qualitative behaviour of particular solutions with symmetries, as well as the different time regimes that appear when the viscosity and the resistivity vary. In \cite{Knobel1,Knobel2,Knobel3}, the authors study the nonlinear stability of solutions of the incompressible MHD equations around certain particular states and show how the solutions behave depending on whether the viscosity or the resistivity dominate. In \cite{Knobel3} it is shown that inflation of the norms of the magnetic field occurs in different times scales depending on the relative value of the viscosity and the resistivity. 

\medskip

Moffatt \cite{Moffatt2015} performed numerical simulations in the model \eqref{eq:magneticrelaxation} as well as some analytical  computations  in order to understand the relaxation process undergone by the solutions. In this paper, we perform an analytical study of equations \eqref{eq:magneticrelaxation} in order to obtain a general picture of the evolution of solutions. We will prove the existence of two time scales:
\begin{enumerate}
    \item An initial time scale of order  $\log (\varepsilon^{-1})$ influenced in a significant manner by viscosity, where the fluid relaxes to a state of constant modulus $|b| \equiv R(t)$. Therefore, if we consider $b=b_1+ib_2$ as a complex number, we can approximate it by a function of the form  $b = R(t) e^{i \theta(t,x)}$.

    \item A second stage on a time scale of order $1/\varepsilon$, where the evolution of the angular variable is influenced in a significant manner by the electrical resistivity. We will derive a PDE associated to the dynamics of $\theta$ and prove a precise convergence theorem of solutions of \eqref{eq:magneticrelaxation} to solutions of this PDE in the asymptotic regime $\varepsilon \rightarrow 0^+$.

\end{enumerate}

One main feature of this model is that the resulting PDE that describes the limit as $\varepsilon\rightarrow 0$ is not globally well posed. We prove in Section \ref{sec:limiteqandblowup} that it has a blow up in finite time. This opens the door to further study, since the original system \eqref{eq:magneticrelaxation} is globally well posed in time. Therefore, it would be interesting to study how the blow up is regularized if one keeps the original model \eqref{eq:magneticrelaxation}.

It is worth noticing that, even though it is already pointed out in \cite{Moffatt2015} that two time scales exist, we show here that they take place at longer times than the ones predicted by Moffatt. The choice of the constant $\varepsilon$ made in said article  is $\varepsilon=10^{-3}$ (we make a change of notation with respect to \cite{Moffatt2015}, as the quantity that we denote by $\varepsilon$ is denoted there by $\kappa$). Therefore, the relaxation to a state of constant modulus happens in times of order $1$, whereas the second time scales happens in times of order $10^3$. In other words, the numerical simulations described in \cite{Moffatt2015} do not run for long enough to be able to see some of the behaviour we describe in this article.

\subsection{Derivation of the Model}\label{ss:derivation}

    This model can be derived as a particular regime of the compressible MHD Equations. Recall that this system, written for a  compressible magnetofluid of viscosity $\kappa$ and resistivity $\eta$, is

\begin{subequations}\label{eq:MHDsystem}
    \begin{align}
        &\frac{\partial\rho}{\partial t}+\nabla\cdot\left(\rho \textbf{u}\right)=0,\\
        \label{2ndlawnewton}\rho\bigg(&\frac{\partial\textbf{u}}{\partial t}+\textbf{u}\cdot\nabla \textbf{u}\bigg)+\nabla P(\rho) -\left(\nabla\times \textbf{B}\right)\times \textbf{B}=\kappa \Delta \textbf{u},\\
        &\label{eq:Faraday}\frac{\partial \textbf{B}}{\partial t}-\nabla\times \left(\textbf{u}\times \textbf{B}\right)=\eta \nabla\times (\nabla\times  \textbf{B}),\\
        &\label{eq:divfree}\nabla\cdot\textbf{ B}=0.
    \end{align}
\end{subequations}

In agreement with the notation defined above, the vector fields $\textbf{u}$ and $\textbf{B}$ are, respectively, the velocity and the magnetic field of the fluid, while the scalar $\rho$ is the density. The quantity $P(\rho)$ is the pressure, and is given as a function of the density through a state equation, for example $P(\rho) = \rho^\gamma$, although this will be irrelevant in the sequel.

\medskip

We consider the special case where the solution of \eqref{eq:MHDsystem} have a particular form. More precisely, we assume that the magnetic field and the velocity only depend on the first space variable $x_1 = \textbf{x}\cdot (1, 0, 0)$, which we denote simply by $x$, and have the form
\begin{equation*}
    \textbf{B}(t, x) = (0, b_1(t, x), b_2(t, x)) \qquad \text{and} \qquad \textbf{u}(t,x) = (u(t, x), 0, 0).
\end{equation*}

Notice that such a magnetic field is automatically divergence-free. It will be convenient to define the 2D vector $b = (b_1, b_2)$, which we will sometimes identify with the complex number $b_1 + ib_2$ (we have removed the bold font for simplicity of notation). In this highly symmetric setting, the Lorentz force $(\nabla\times \textbf{B})\times \textbf{B}$ in \eqref{2ndlawnewton} takes the simpler form

$$(\nabla\times \textbf{B})\times \textbf{B}=\frac{1}{2}\partial_x(|\textbf{B}|^2)e_x.$$

On the other hand, Faraday's law \eqref{eq:Faraday} also simplifies under these assumptions, since

\begin{equation*}
    - \nabla \times (\textbf{u} \times \textbf{B}) - \eta \nabla \times (\nabla \times \textbf{B}) = - \partial_x (u \textbf{B}) - \eta \partial_x^2 \textbf{B}.
\end{equation*}

With this considerations in place, we can write the MHD system \eqref{eq:MHDsystem} as

\begin{subequations}\label{eq:MHDsystem2}
    \begin{align}
        &\frac{\partial\rho}{\partial t}+\partial_x\left(\rho u \right)=0,\\
        \rho\bigg(&\frac{\partial u}{\partial t}+u\partial_x u\bigg) +\frac{1}{2}\partial_x\left( P(\rho) + |b|^2\right)=\partial_x^2u,\\
        &\frac{\partial b}{\partial t}+\partial_x(ub)=\eta \partial_x ^2 b
    \end{align}
\end{subequations}

We now make the assumption that the density of the plasma is low enough that the fluid pressure $p$ is negligible compared to the magnetic pressure $\frac{1}{2} |\textbf{B}|^2$, this assumption is known as the low $\beta$ approximation, and is especially relevant for low-density plasmas driven by a large magnetic field. This leads to omitting the pressure terms from the system altogether. We then use reformulate the problem using non-dimensional variables. To this end, define the quantities
\begin{equation*}
    x\mapsto x/L\quad t\mapsto t/(\kappa / |\textbf{B}_0|^2) \quad \rho\mapsto \rho/\rho_0 \quad \textbf{B}\mapsto \textbf{B}/|\textbf{B}_0|\quad u\mapsto u/(|\textbf{B}_0|^2L/\kappa)
\end{equation*}
so that we obtain

\begin{subequations}
\begin{align}
        \label{eq:MHDsystem2rho}&\frac{\partial\rho}{\partial t}+\partial_x\left(\rho u \right)=0,\\
        \label{eq:MHDsystem2b}\alpha \rho\bigg(&\frac{\partial u}{\partial t}+u\partial_x u\bigg)+\frac{1}{2}\partial_x\left(|b|^2\right)=\partial_x^2u,\\
        \label{eq:MHDsystem2u}&\frac{\partial b}{\partial t}+\partial_x(ub)=\varepsilon \partial_x^2b.
\end{align}
\end{subequations} 

This is the system that was obtained and studied by Moffatt in \cite{Moffatt2015}. The parameters $\alpha=\rho_0B_0^2/\kappa_0^2 $ and $\varepsilon=\eta \kappa_0/ B_0^2L^2$ are assumed to be small, in virtue of the smallness of $\rho$ (diffused gas approximation) and the smallness of $\eta$ (small conductivity assumption). 

Finally, we make the further assumption that $\alpha=0$, which is consistent with the low density assumption (the kinematic viscosity scales as $1/\rho_0$ comparatively to the density, see the remarks in Section 8 of \cite{Moffatt2016}), so that \eqref{eq:MHDsystem2b} reduces to a Stokes-like equation. Then, \eqref{eq:MHDsystem2b} and \eqref{eq:MHDsystem2u} decouple from the density equation \eqref{eq:MHDsystem2rho}, and the density has no bearing on the problem. This leads to our final system \eqref{eq:magneticrelaxation}. Therefore, the system \eqref{eq:magneticrelaxation} corresponds to a situation where the gas is very diluted and has a very low pressure compared to the magnetic pressure.

\medskip

Note that the equations \eqref{eq:magneticrelaxation} have an additional degree of freedom, as the velocity $u$ is only determined by the equation $\partial_x u = \frac{1}{2}|b^2| - \int \frac{1}{2}|b^2|$ up to the addition of a function of time only. Therefore, to get well-posed equations, some additional constraint must be imposed on the velocity, for example one could require that it has zero mean $\int_\T u \dx = 0$, or it fulfils some cancellation condition $u(t,0) = 0$. However, the precise condition we impose on $u$ does not matter, thanks to the invariance of equations \eqref{eq:magneticrelaxation} by change of reference frames. If $g(t)$ is a smooth function of time determining a (non-inertial) change of reference frames $\bar{x} = x + g(t)$, and if $(b, u)$ is a solution of \eqref{eq:magneticrelaxation}, then the functions
\begin{equation}\label{ieq:Galileo}
    \bar{b}(t,\bar{x}) = b\big(t, \bar{x} - g(t)\big) \qquad \text{and} \qquad \bar{u}(t,\bar{x}) = u\big( t, \bar{x} - g(t) \big) + g'(t)
\end{equation}
also define a solution $(\bar{b}, \bar{u})$ of \eqref{eq:magneticrelaxation}. Therefore, by setting $g$ to be a solution of the ODE $u(t,g) + g' = 0$, we find a solution $v$ such that $\bar{u}(t, 0) = 0$. Unless otherwise specified, we will always impose this cancellation condition on the velocity.

The reader should observe that if $g(t) = Vt$, then \eqref{ieq:Galileo} is an inertial change of reference frames, so that equations \eqref{ieq:Galileo} are Galileo invariant. This property is inherited from the Galileo invariance of MHD. We refer to \cite{CobbThese} pp. 28--31 and \cite{Cobb2024} for more on that topic.

\subsection{Main Ideas and Statement of the Results}

The core of the paper is the study of the asymptotic behaviour of the solutions $b_\varepsilon(t,x)$ of \eqref{eq:magneticrelaxation} as $\varepsilon\rightarrow 0$. We now give a heuristic account of how the solution evolves when $\varepsilon \ll 1$. As we have explained above, two time scales are relevant for the analysis, depending on whether the resistive diffusion $\varepsilon \partial_x^2 b_\varepsilon$ has had time to act or not.

\medskip

\textbf{Viscous-driven time scale:} First of all, when the time is small compared to the scale on which viscous diffusion acts, the solutions of \eqref{eq:magneticrelaxation} behave roughly as solutions of the infinitely conducting system 
\begin{equation}\label{ieq:Hyperbolic}
    \begin{cases}
        \partial_t \widetilde{b} + \partial_x (\widetilde{u} \widetilde{b}) = 0,\\
        \partial_x \widetilde{u} = - (\partial_x \widetilde{b})^2 + \int (\partial_x \widetilde{b})^2,
    \end{cases}
\end{equation}
\textit{i.e.}, solutions of \eqref{eq:magneticrelaxation} with $\varepsilon=0$. The approximation $b_\varepsilon \sim \widetilde{b}$ holds roughly for time scales of order $-\log(\varepsilon)$. An important feature of system \eqref{ieq:Hyperbolic} is that it has a 
damping effect, \textit{i.e.} the mean-free part of the magnetic energy, $\widetilde{\psi} = \frac{1}{2}|\widetilde{b}|^2 - \int \frac{1}{2}|\widetilde{b}|^2$, which is in fact equal to $\partial_x \widetilde{u}$, solves the equation
\begin{equation}\label{ieq:HypEnergy}
    \partial_t \widetilde{\psi} + \widetilde{u} \partial_x \widetilde{\psi} + 2 |\widetilde{b}|^2 \widetilde{\psi} = \int_\T \widetilde{\psi}^2 \dx.
\end{equation}
In the above, the term $2|\widetilde{b}|^2 \widetilde{\psi}$ is the damping term. As long as the magnetic field remains bounded away from zero $|\widetilde{b}| \geq c_0$ it forces the mean-free energy $\widetilde{\psi}(t)$ to decay exponentially fast $O(e^{-c_0t})$. In other words, on a time scale of order $t\sim-\log(\varepsilon)$, the norm of the magnetic field $|b_\varepsilon|$ is expected to be equal to a constant up to a term of order $O(\varepsilon)$,
\begin{equation}\label{ieq:BConstantModulus}
    |b_\varepsilon(t,x)| = R_\varepsilon(t) + O(\varepsilon).
\end{equation}
We illustrate this viscosity-driven relaxation process in Figure \ref{fig:Relaxation}.

\begin{figure}[h!]
    \centering
    \includegraphics[width=0.5\linewidth]{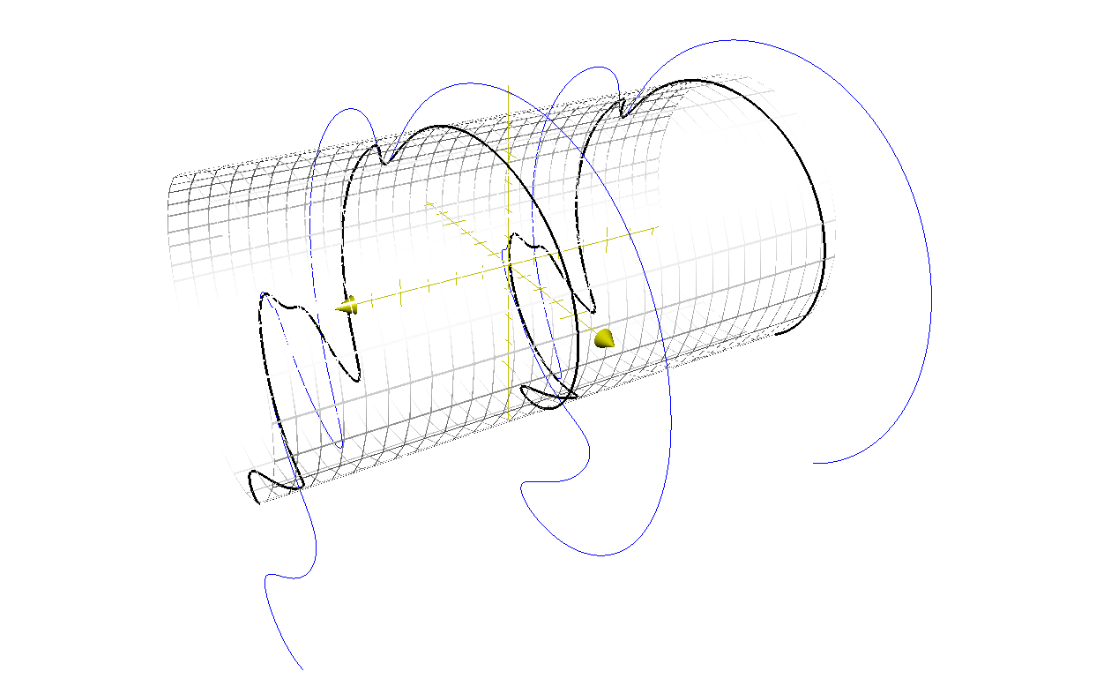}
    \caption{Illustration of the viscosity-driven time scale. The magnetic field is represented as a parametrized curve $x \mapsto (x, b_1(x), b_2(x))$. The initial magnetic field (the thin blue line) relaxes to a state where $x \mapsto |b|$ is constant (the thick black line lies at the surface of a cylinder).}
    \label{fig:Relaxation}
\end{figure}

\medskip

\textbf{Diffusive time scale:} For sufficiently long times, the electrical resistivity starts to matter. On a time scale of order $t\sim \varepsilon^{-1}$, the norm of the magnetic field is already constant up to a $O(\varepsilon)$ correction \eqref{ieq:BConstantModulus}, so that the quantity $U_\varepsilon := \frac{1}{\varepsilon}u_\varepsilon$ is of order $O(1)$. We rescale time accordingly $\tau = \varepsilon t$, and consider the singular perturbation problem
\begin{equation}\label{ieq:SingPert}
    \begin{cases}
        \partial_\tau b_\varepsilon + \frac{1}{\varepsilon} \partial_x (u_\varepsilon b_\varepsilon) = \partial_x^2 b_\varepsilon \\
        \partial_x u_\varepsilon = \frac{1}{2}|b_\varepsilon|^2 - \int \frac{1}{2}|b_\varepsilon|^2,
    \end{cases}
    \qquad \text{as } \varepsilon \rightarrow 0^+.
\end{equation}
In order to understand the behaviour of the magnetic field on the time scale $t \sim 1 / \varepsilon$, we want to show that the quantity $b_\varepsilon(\tau, x)$ has a limit $b(\tau, x)$ as $\varepsilon\rightarrow 0$, and to identify the limit dynamics, a PDE system solved by $b(\tau, x)$. 

\medskip
\medskip

Let us focus on the quantity $U_\varepsilon = \frac{1}{\varepsilon}u_\varepsilon$ (which, again, is of order $O(1)$ due to \eqref{ieq:BConstantModulus}). By computations similar to those that led to \eqref{ieq:HypEnergy}, we can show that $\psi_\varepsilon$ solves the equation
\begin{equation*}
    \partial_\tau \psi_\varepsilon+u_\varepsilon\partial_x \psi_\varepsilon+\frac{|b_\varepsilon|^2}{\varepsilon}\psi_\varepsilon-\int_{\T}\psi_\varepsilon^2= \partial_x^2\psi_\varepsilon+ \left(\int_{\T}|\partial_x b_\varepsilon|^2 \dx -|\partial_x b_\varepsilon|^2\right).
\end{equation*}
By discarding all the terms of order $O(\varepsilon)$ or higher, we find 
formula  for $\partial_x U_\varepsilon$ up to higher order terms:
\begin{equation}\label{ieq:MultiScalePsi}
    \frac{1}{\varepsilon} |b_\varepsilon|^2 \psi_\varepsilon = |b_\varepsilon|^2 \partial_x U_\varepsilon = - |\partial_x b_\varepsilon|^2 + \int_\T |\partial_x b_\varepsilon|^2 + O(\varepsilon).
\end{equation}
By putting together the above \eqref{ieq:MultiScalePsi} and \eqref{ieq:BConstantModulus}, we deduce that $b_\varepsilon(t, x) = b(t/\varepsilon, x) + O(\varepsilon)$, where $b(\tau, x)$ solves the PDE system
\begin{equation}\label{ieq:LimitDynamics}
    \begin{cases}
        \partial_\tau b + \partial_x (U b) = \partial_x^2 b \\
        \partial_x U = \frac{1}{|b|^2} \left( - |\partial_x b|^2 + \int_\T | \partial_x b|^2 \right),
    \end{cases}
\end{equation}
and this describes the dynamics of the magnetofluid on the diffusive time scale.

\medskip

\textbf{Study of the limit dynamics.} At this point, we must make an important remark. The heuristic arguments we have given above only hold provided that \eqref{ieq:BConstantModulus} is true with $R_\varepsilon(\tau) \gg \varepsilon$. Otherwise, the leading order terms in \eqref{ieq:MultiScalePsi} do not provide an equation for $U_\varepsilon$. Furthermore, if $|b_\varepsilon| \approx R_\varepsilon(\tau)$ is not positive, then the second equation in \eqref{ieq:LimitDynamics} does not even make sense. Therefore, for our analysis to be justified, it is important to check that the norm $|b_\varepsilon|$ of the magnetic field does not collapse to zero, or equivalently, we can check that the norm $|b|$ of the solution of \eqref{ieq:LimitDynamics} remains non-zero, as $|b| = |b_\varepsilon| + O(\varepsilon)$.

\medskip

By \eqref{ieq:BConstantModulus}, we know that the quantity $R(\tau) = |b_(\tau, x)|$ only depends on $\tau$, so we can write the solution of \eqref{ieq:LimitDynamics} as $b = R(\tau)e^{i \theta (\tau, x)}$, where the angle $\theta$ is defined on the covering $\R$ of the torus $\T = \R / \Z$. The couple $(R, \theta)$ solve
\begin{equation}\label{ieq:AngleVariable}
    \begin{cases}
        \partial_\tau \theta + U \partial_x \theta = \partial_x^2 \theta \\
        \partial_x U = - (\partial_x \theta)^2 + \int (\partial_x \theta)^2,
    \end{cases}
    \qquad \text{and} \qquad R'(\tau) = -R(\tau) \int_0^1 (\partial_x \theta)^2.
\end{equation}
Therefore, as long as the solution of this system is regular enough $\theta \in L^2_t(H^1_x)$, the norm $|b|$ cannot vanish, and our asymptotic analysis is justified. It is therefore crucial to see if solutions of \eqref{ieq:AngleVariable} are regular, and if they can develop singularities in finite time. On the one hand, regular initial data of \eqref{ieq:AngleVariable} give rise to smooth solutions on some time interval $[0, T)$. On the other hand, solutions undergo blow-up for some initial data.

\medskip

\textbf{Statement of the results.} In this article we make the heuristics above rigorous. We now state the different results we have obtained, first concerning the limit system \eqref{ieq:AngleVariable}, and then concerning the asymptotic analysis which depends on it.

\begin{teor}[see Theorems \ref{t:existencelimit}, \ref{t:smallOscillations} and \ref{prop:blowup}]\label{it:Limit}
    The following assertions hold:
    \begin{enumerate}
        \item Consider initial data $R_0 > 0$ and $\theta_0 \in \dot{H}^{1/2}(\R)$. Then there exists a time $T^\star > 0$ such that system \eqref{ieq:AngleVariable} has a unique solution $\theta \in C^\infty((0, T^\star) \times \R)$ associated to that initial datum. In particular, $R(\tau) > 0$ for all $0 \leq \tau < T^\star$.

        \item Under the condition that $\int_0^1 \theta_0 = 0$ and $\max_x \theta_0 - \min_x \theta_0 \leq \frac{1}{\sqrt{3}}$, the lifespan of the solution is infinite $T^\star = \infty$.

        \item There exists initial data (fulfilling a growth condition), such that the lifespan of the solution is finite, since the angle becomes non-smooth in finite time $T^\star< +\infty$ and
        \begin{equation}\label{ieq:BlowUpL4}
            \int_0^{T^\star} \| \partial_x \theta \|_{L^2}^4 \dt = + \infty.
        \end{equation}
    \end{enumerate}
\end{teor}

The meaning of the condition $\int_0^1 \theta_0$ is that the magnetic field, considered as a plane curve $x \in \T \mapsto (b_1(x), b_2(x)) \in \R^2$ does not make a full turn around the origin, \textit{i.e. }it has index zero. It should be noted that while global in time solutions \textit{may} exist under this condition (as \textit{(2)} states), it is also possible to have $T^\star < +\infty$, as we prove in Section \ref{sec:limiteqandblowup}. Numerical experiments are also provided in Section \ref{sec:numerics}.

\medskip

This results show that the heuristic analysis above can be justified, at least on some time interval $\tau \in (0, T^\star)$. After that time, blow-up of solution of \eqref{ieq:AngleVariable} may occur (and in fact does occur for some initial data, as seen below), and our analysis breaks down. However, note that \eqref{ieq:BlowUpL4} does not in fact prove that $R(T^\star) = 0$. Cancellation of the radius requires the stronger condition $\| \partial_x \theta \|_{L^2([0, T^\star) ; L^2)} = + \infty$. What happens exactly at the time of blow-up is still an open question, and we plan to investigate it in a further paper. In any case, the asymptotic analysis $\varepsilon \rightarrow 0^+$ can be conducted on the time interval $[0, T^\star)$, and we get the following result.

\begin{teor}[See Theorem \ref{teor:main}]\label{it:Convergence}
    Consider $(b_\varepsilon)_{\varepsilon > 0}$ the solutions of the primitive equations \eqref{eq:magneticrelaxation} associated to some regular enough initial datum $b_0$ such that there exists a constant $c_0 > 0$ with $|b_0(x)| > c_0$ for all $x \in \T$. Then the following assertions hold:
    \begin{enumerate}
        \item The solution $\widetilde{b}(t,x)$ of the perfectly conducting system \eqref{ieq:Hyperbolic} associated to the initial datum $b_0$ converges as $t \rightarrow \infty$ to a function $\widetilde{b_\infty} = \lim_{t \rightarrow \infty} \widetilde{b}(t,x)$.

        \item Let $T^\star > 0$ be the maximal time of existence of solutions of \eqref{ieq:LimitDynamics} for the initial datum $R_0e^{i \theta_0(x)} = \widetilde{b_\infty}$, and let $b(\tau, x) = R(\tau)e^{i \theta(\tau, x)}$ be the associated solution. Then the sequence $(b_\varepsilon(\tau, x))_{\varepsilon > 0}$ converges to $b(\tau, x)$ as $\varepsilon \rightarrow 0^+$, where $\tau = \varepsilon t$.
    \end{enumerate}
\end{teor}

\subsection{Organization of the Paper and Notation}

We devote Section \ref{sec:epsilonzero} to study the case of a perfectly conducting fluid, \textit{i.e.} the case of $\varepsilon=0$. This will serve as a base case to study the first time scale of the full problem \eqref{eq:magneticrelaxation}. We show global well posedness of classical solutions and convergence to a particular case of force-free field, the case of $b$ with constant modulus. 

In Section \ref{sec:epsilonfix} we show global well posedness for the system \eqref{eq:magneticrelaxation} with a fixed, but non-zero, resistivity $\varepsilon$. This is done by a classical approximation scheme, based on energy estimates of the equation. We also show that this equation exhibits instantaneous regularization. This will prove very useful, as we will be able to work with classical techniques when describing the qualitative behaviour when $\varepsilon\rightarrow 0$.

In Section \ref{sec:Munchhausenestimates} we study the qualitative behaviour of the system. It is here where we show the existence of two time scales. More precisely, we show that for times of order $t=\log(\varepsilon^{-1}))$ (the ``fast'' time scale) the solution $b$ of \eqref{eq:magneticrelaxation} behaves like its zero resistivity counterpart, and that for times up to order $1/\varepsilon$ (the ``slow'' time scale) the magnetic field $b$ can be approximated by the solution of a limiting PDE as $\varepsilon \rightarrow 0^+$. 

Finally, Section \ref{sec:limiteqandblowup} is devoted to the study of the limit dynamics \eqref{ieq:AngleVariable}. The results concerning the well-posedness of the equations, the finite time blow-up of certain solutions, and the numerical experiments are enclosed therein.

\subsubsection{Notation, Conventions and Definitions}
We now introduce some of the notation we will use throughout the paper.

\begin{itemize}
    \item As indicated before, we will work with a periodic domain, namely, $\T=\R/\Z$. For functions defined on $\T$, we will equivalently refer to them as functions $f:\R\rightarrow \T$ which are periodic with period $1$. Often, specially in Section \ref{sec:limiteqandblowup}, it will be useful to treat with functions $f:\T\rightarrow \R$ as functions $f:(-1/2,1/2)\rightarrow \R$ with periodic boundary conditions. 

    \item Given a finite dimensional vector $u\in \R^d$, we denote by $|u|$ its euclidean norm, \textit{i.e.} 
    $$
        |u|:= \left(\sum_{k=1}^du_k^2\right)^{1/2}.
    $$

    \item For any $s \in \R$, we denote $H^s(\T)$ and $\dot{H(\T)}^s$ the Sobolev spaces associated to the norm and semi-norm
    \begin{equation*}
        \| f \|^2_{H^s} = \sum_{k \in \Z} (1 + k^2)^s |\what{f}(k)|^2 \qquad \text{and} \qquad \| f \|^2_{\dot{H}^s} = \sum_{k \in \Z} k^{2s} |\what{f}(k)|^2.
    \end{equation*}
    Very often, we consider functions $f(t,x)$ of time and space as functions of $t$ with values in some Sobolev space $f \in L^p([0, T) ; H^s(\T))$. We will occasionally use the shorthand $L^p_T(H^s)$ for such spaces. Unless otherwise specified, the norms $\| \, \cdot \, \|_{H^s} = \| \, \cdot \, \|_{H^s(\T)}$ and $\| \, \cdot \, \|_{L^q} = \| \, \cdot \, \|_{L^q(\T)}$ refer to the space variable only.
\end{itemize}

Regarding the functional spaces, besides the classical Sobolev and Lebesgue spaces, we will also make use of the so called parabolic Hölder spaces, particularly suitable to derive regularity estimates for parabolic equations. 

\begin{defi}
    Let $\alpha\in (0,1)$ and $T>0$. We say that a function $f:[0,T]\times \T\rightarrow \R^d$ belongs to the space $C^{\alpha/2,\alpha}([0,T]\times \T)$ if is in $C^0([0,T]\times \T)$ and the seminorm 

    \begin{equation}\label{Hoelder}[f]_{\alpha/2,\alpha}:=\underset{(t,x),(t',x')\in [0,T]\times \T}{\sup_{(t,x)\neq (t',x')}}\frac{|f(t,x)-f(t',x')|}{(|t-t'|^{1/2}+|x-x'|)^{\alpha}}
    \end{equation}
    is finite. 
\end{defi}

We can then define the parabolic Hölder space.

\begin{defi}
    Let $\alpha\in (0,1)$, and $T>0$. We define the parabolic Hölder space $C^{\alpha/2,\alpha}([0,T]\times \T)$ as the set of functions $f\in C^0([0,T]\times \T)$ such that the semi-norm \eqref{Hoelder} is bounded. We then define the norm 

    $$\|f\|_{C^{\alpha/2,\alpha}}:=\|f\|_{C^{0}([0,T]\times \T)}+[f]_{\alpha/2,\alpha}.$$
\end{defi}
\begin{obs}
    It is straightforward to see that $C^{\alpha/2,\alpha}$ is a Banach space. 
\end{obs}
\section{The Case of a Perfectly Conducting Fluid}\label{sec:epsilonzero}

We start by examining the behavior of solutions when the fluid is perfectly conducting, that is, when the resistivity parameter $\varepsilon$ is zero. In that case, we find that the damping produced by the viscosity forces the magnetic field to converge to a state with constant modulus $|b|$ as $t \rightarrow \infty.$ 

\medskip

In the case of a perfectly conducting fluid, $\varepsilon = 0$, equations \eqref{eq:magneticrelaxation} can be written as
\begin{equation}\label{eq:hyperbolicSystem}
    \begin{cases}
        \partial_t b + \partial_x (ub) = 0, \\
        \partial_x^2 u = \partial_x \left( \frac{1}{2}|b|^2 \right).
    \end{cases}
\end{equation}
These equations describe to first stage of the evolution as $\varepsilon \ll 1$. The main theorem of this section reads as follows:

\begin{teor}\label{teor:hyperbolicmain}
    Let $b_0\in C^1(\T)$. Then, there is a unique solution $\tilde{b}\in C^1([0,\infty)\times \T)$ for the system \eqref{eq:hyperbolicSystem} with initial data $b_0$. The solution $\tilde{b}(t,\cdot)$ tends in  $L^2(\T)$, as $t\rightarrow \infty$, to a constant $R$. Assume further than $|b_0(x)|\geq c_0>0$ for every $x\in\T$. Then, the solution $\tilde{b}(t,\cdot)$ converges uniformly, as $t\rightarrow \infty$, to a function $S(b_0)\in C^0(\T)$ with a constant modulus.
\end{teor}

The proof of global well posedness of \eqref{eq:hyperbolicSystem} is in Proposition \ref{prop:wellposednesshyp}. In \ref{prop:L2convergencehyperbolic} we show that the modulus of $b(t,\cdot)$ converges in $L^2$ to a constant and, in the case of $|b_0(x)|^2\geq c_0$, we show that $b(t,\cdot)$ converges uniformly to a function of constant modulus in \ref{prop:decayunondiff}.
\begin{prop}\label{prop:wellposednesshyp}
    Let $b_0\in C^1(\T)$. Then, there exists a unique global in time solution $b\in C^1(\R_+\times \T)$ of \eqref{eq:hyperbolicSystem}.
\end{prop}
\begin{proof}
    In order to prove local existence and uniqueness, we use a classical argument by means of Banach fixed point theorem. Note that, for a given map $f\in C^1([0,T]\times \T;\R^2)$, the equation 

    \begin{equation}\label{eq:linearizedhyperbolic}
    \left\{\begin{array}{ll}
    \partial_t g+u(\partial_x g)+(\partial_x u) g=0, & (t,x)\in [0,T]\times \T\\
    \partial_x u=\frac{1}{2}\left(|f|^2-\|f\|^2_{L^2}\right)\\
    g(0,x)=b_0
    \end{array}\right.
    \end{equation}
    has a unique solution in $C^1([0,T]\times \T)$. This is an application of the method of characteristics. Now, for a given $f\in C^1([0,T]\times \T;\R^2)$, we denote by $F[f]=:g$ the solution of the problem \eqref{eq:linearizedhyperbolic}. It is then immediate that a solution of \eqref{eq:hyperbolicSystem} is a fixed point of $F$. We can prove by means of standard arguments that, for any given $C_0>0$, if $T=T(\|b_0\|_{C^1},C_0)>0$ is small enough,  $F$ admits a unique fixed point in the space $(X_{C_0},\|\cdot\|_{C^0})$, with 

    \begin{equation*}
    \begin{split}
    X_{C_0}:=\{f\in C^0([0,T]\times \T)\cap L^\infty([0,T];\text{Lip}(\T))\,&: \|f\|_{C^0([0,T]\times \T)}+\|f\|_{L^\infty([0,T];\text{Lip}(\T)}\leq (C_0+1)\|b_0\|_{C^1}\\
    &\text{ and }f(0,\cdot)=b_0(\cdot)\}.
    \end{split}
    \end{equation*}

    Note that the space $X_{C_0}$ is a closed subspace of $C^0([0,T]\times \T)$. Moreover, since the initial data is $C^1(\T)$, it follows that the solution is $C^1([0,T]\times \T)$. Then, the choice of $X_{C_0}$ implies that at the maximum time $T^\star$ of existence of the solution $b$ 
    it holds that

    $$\limsup_{t\rightarrow (T^{\star})^-}
    \left(\|b(t,\cdot)\|_{L^\infty(\T)}+\|b(t,\cdot)\|_{\text{Lip}(\T)}\right)=\infty.$$
    
    Therefore, if we show that both $\|b(t,\cdot)\|_{L^\infty}$ and $\|b(t,\cdot)\|_{C^1}$ remain bounded in $[0,T]$, then the solution can be prolonged beyond $T$. The proof can be adapted to show that there exists also a unique solution in $[-T,0]$ for \eqref{eq:hyperbolicSystem}. We now prove that the solution is global in time.

    Note that functions with constant modulus independent of time are solutions of \eqref{eq:hyperbolicSystem}. Therefore, by uniqueness, it is clear that if $b(t,\cdot)$ has modulus equal to constant for some $t>0$, then it is independent of time and global existence is clear. We then assume in the sequel that $b(t,\cdot)$ does never have constant modulus for any $t\geq 0$. 
    
    In order to prove that the solution is global, we begin proving that the $L^\infty$ norm of $b$ does not blow up. To that end, we prove that the solution $b$ satisfies the estimate
    
    $$\|b\|_{L^\infty([0,T]\times \T)}\leq \|b_0\|_{L^\infty}$$

        This follows from a maximum principle argument. Indeed, at the point $(t_\star,x_\star)\in [0,T]\times \T$ where $|b|^2$ reaches its maximum, it holds that $\partial_x (|b|^2)(t_\star,x_\star)=0$. Therefore, if $t_\star>0$,

        $$0\leq \partial_t |b|^2=-\frac{1}{2}\left(|b|^2-\int_\T |b|^2\right)|b|^2<0\quad \text{at }(t_\star,x_\star),$$
         which leads to a contradiction. We can also prove an estimate for the derivative of $b$, namely 
         
         \begin{equation}\label{eq:boundderivativeb} 
         \|\partial_x b(t,\cdot)\|_{L^\infty(\T)}^2\leq e^{\|b_0\|_{L^\infty}^2}\| b_0\|^2_{L^\infty}\exp\left(T\|b_0\|^2_{L^\infty}\right).
         \end{equation}

        To do so, we use the representation formula that follows from the method of characteristics. Note that both $u$ and $\partial_x u$ are $C^1([0,T]\times \T)$ functions, so we can write the solution $b$ as $b=Z\circ \Psi^{-1}(t,x)$, with $\Psi:=(t,\Phi_t(\xi))$ and $\Phi_t$ is the one parameter family of diffeomorphisms given by 
        
        \begin{equation} \label{eq:defiphi} 
        \left\{\begin{array}{ll}
            \frac{\partial}{\partial t}\Phi_t (\xi)=u(t,\Phi_t(\xi)) & t>0 \\
            \Phi_0(\xi)=\xi & \xi\in \T 
        \end{array}\right.,
        \end{equation} 
        and $Z$ is the solution of the ODE
        $$\left\{\begin{array}{ll}
            \partial_t Z=-(\partial_x u\circ \Phi_t (\xi))Z & t>0 \\
             Z(0,\xi)=b_0(\xi)& \xi\in \T  
        \end{array}\right..$$

    Note that, since $\partial_x u=\frac{1}{2}\left(|b|^2-\int|b|^2\right)$, it holds that   

    $$\partial_t Z(\xi,t)=-\frac{1}{2}\left(|Z(\xi,t)|^2-\int|b|^2\right)Z.$$

    As a result, we can take the derivative with respect the parameter $\xi$ and multiply times $\partial_\xi Z$, leading to 

    \begin{equation} \label{eq:estimatederivative}
    \frac{1}{2}\partial_t |\partial_\xi Z(t,\xi)|^2=-(Z\cdot\partial_\xi Z)^2-\frac{1}{2}\left(|Z(\xi,t)|^2-\int|b|^2\right)|\partial_\xi Z|^2.
    \end{equation}

    Therefore, 

    $$|\partial_\xi Z(t,\xi)|\leq |\partial_\xi b_0(\xi)|\exp\left(T\|b_0\|^2_{L^\infty}\right).$$
   Now, $b$ equals $Z(t,\xi(t,x))$, with $\xi$ defined via
   
   \begin{equation} \label{eq:eqxi}
   \Phi_t(\xi(t,x))=x.
   \end{equation} 
   
    Therefore, $\partial_x \xi$ equals 

    $$\frac{\partial \xi}{\partial x}(t,x)=\left(\frac{\partial \Phi_t}{\partial \xi}(\xi(t,x))\right)^{-1}=\left(\exp\left(-\int_0^t\partial_x u(t,x )\ds\right)\right).$$

    Since $|\partial_x u(t,x)|\leq 2\|b\|_{L^\infty}\leq 2\|b_0\|_{L^\infty}$, and since $\partial_x b(t,x)=\partial_\xi Z(t,\xi(t,x))\partial_x \xi(t,x)$, the bound \eqref{eq:boundderivativeb} follows.
    We then conclude that the norm $\|b(t,\cdot)\|_{C^1}$ does not blow up in finite time. Therefore, the solution is global. 
\end{proof}

	We now study the long time behaviour of the solutions of \eqref{eq:hyperbolicSystem}. It turns out that  solutions for \eqref{eq:hyperbolicSystem} converges to a function  $b$ with $|b|$ constant. This is the content of the following proposition,

    \begin{prop}\label{prop:L2convergencehyperbolic}
    Consider an initial datum $b_0\in C^1(\T)$ that is not identically equal to zero. Then, the solution $b$ of \eqref{eq:hyperbolicSystem} satisfies that 
    
    \begin{equation} \label{eq:L2convergencetozero}
    \|\partial_x u(t,\cdot)\|_{L^2}\leq Ce^{-(t\|b_0\|^2_{L^1})/2},
    \end{equation}
    and $|b(t,\cdot)|\longrightarrow \|b_0\|_{L^1(\T)}$ in $L^2(\T)$ as $t\rightarrow\infty$.
\end{prop}
\begin{proof}
    The proof of the statement is based on the conservation of the $L^1$ norm for the magnetic field $b$. Since our solution is $C^1(\R_+\times \T)$, it is Lipschitz, so it is absolutely continuous, and its a.e. derivative satisfies 

    $$\partial_t (|b|)+\partial_x (u|b|)=0 \quad \text{a.e. } (t,x)\in \R_+\times \T.$$
    Therefore, it holds that 

    \begin{equation} \label{eq:massconservation}
    \frac{d}{dt}\int_\T |b(t,x)|\dx=0 \quad \text{a.e.}\, t\geq 0.
    \end{equation}

    As a result, the $L^1$ norm of $|b(t,\cdot)|$ is conserved.

    On the other hand, due to Hölder's inequality (recall that the torus is assumed to have measure one), $\|b(t,\cdot)\|_{L^1}\leq \|b(t,\cdot)\|_{L^2}.$ We now derive an energy estimate for $b$. First, we test \eqref{eq:magneticrelaxation} against $b$ and integrate with respect to $x$, so we obtain 
    
    \begin{equation}\label{eq:energyestimateb} 
    \frac{1}{2}\frac{d}{dt}\int_\T |b(t,x)|^2\dx+\int_\T \psi^2(t,x)\dx=0.
    \end{equation}

    We can use this to derive an equation for $\psi:=\partial_x u$ that reads  

    $$\partial_t\psi +u\partial_x \psi+2\psi |b|^2=\int_\T \psi^2 \dx.$$

    After testing against $\psi$ and integrating by parts, we obtain the following energy estimate for $\psi$:

    \begin{equation} \label{eq:energyestimatepsi} 
    \frac{1}{2}\frac{d}{dt}\int_\T \psi^2(t,x)\dx+\frac{1}{4}\left(\int_\T |b(t,x)|^2\dx\right)\left(\int_\T \psi^2(t,x)\dx\right)+\frac{3}{4}\int_\T|b(t,x)|^2\psi^2(t,x)\dx=0.
    \end{equation}

    Therefore, we obtain the inequality 

    $$\frac{1}{2}\frac{d}{dt}\int_\T \psi^2(t,x)\dx+\frac{1}{4}\|b_0\|_{L^1}^2\left(\int_\T \psi^2(t,x)\dx\right)\leq 0,$$
    and by means of Gronwall's lemma,

    $$\int_\T \psi^2(t,x)\dx\leq \|\psi_0\|^2_{L^2}\exp\left(-\frac{\|b_0\|^2_{L^1}}{2}t\right),$$
    so \eqref{eq:L2convergencetozero} holds. Now, we can prove convergence of $b$. To do so, note that \eqref{eq:energyestimateb} implies that $\|b(t,\cdot)\|_{L^2}^2$ is decreasing. Therefore, $\|b(t,\cdot)\|_{L^2}^2$ has a limit $\ell\geq 0$ as $t\rightarrow \infty$.  $|b(t,\cdot)|^2$ tends to $\ell$ in $L^1$ as $t\rightarrow \infty$. Indeed, $|b(t,x)|^2=2\psi+\|b(t,\cdot)\|_{L^2}^2$. The first summand tends to $0$ in $L^2$ (and, thus, in $L^1$), and the second tends to $\ell$ by construction. 

    An immediate application of the Theorem of Dominated Convergence implies that, actually, $|b(t,\cdot)|-\sqrt{\ell}$ tends to $0$ in $L^2$ as $t\rightarrow \infty$. Consequently, and since the $L^1$ norm is conserved, it holds that the $L^2$-limit of $|b(t,\cdot)|$ equals $\|b_0\|_{L^1}$. 
\end{proof}
	In the previous proposition we obtained a convergence theorem that holds with respect to the $L^2$ norm. If we further assume that $b_0$ does not vanish anywhere in $\T$, we can obtain convergence in stronger norms.
	\begin{prop}\label{prop:decayunondiff}
		Let $b$ be a $C^1([0,T]\times \T)$ solution of \eqref{eq:hyperbolicSystem}. Assume that the initial value $b_0(x)$ satisfies that 
        
        $$|b_0(x)|^2\geq c_0>0\quad \forall\,x\in \T$$
        for some $c_0>0$. Then, $|b(t,x)|^2\geq c_0$ for every $(t,x)\in \R_+\times \T$, and there exists a constant $C>0$ satisfying that  
		
		\begin{equation} \label{eq:expdecayu}
        \|\partial_x u(t,\cdot)\|_{L^\infty}\leq Ce^{-c_0t}.
        \end{equation}
	\end{prop}
	\begin{proof} 
		
		We begin showing that $|b(t,x)|^2\geq c_0$ for every $(t,x) \in \R_+\times \T$. We write an equation for $|b(t,\cdot)|^2$, that can be obtained simply by multiplying the equation times $b$, leading to 
		
		\begin{equation}\label{eq:bnodifusive}
			\frac{1}{2}\partial_t (|b|^2)+\frac{1}{2}u\partial_x (|b|^2)+\partial_x u |b|^2=0\quad \text{with}\quad \partial_x u=\frac{1}{2}\left(|b|^2-\int_{\T}|b|^2\,\dx\right).
		\end{equation}
		
		As noted out in the proof of Theorem \ref{prop:wellposednesshyp}, the function $|b|^2$ constant is trivially a solution for this equation. The proposition holds for this case, so we may assume that $|b|^2$ is not identically constant. 
        
        Now, for the case when $|b_0|^2$ is not identically constant, we can argue by contradiction. If $|b(t,x)|^2$ is smaller than $c_0$ for some $t>0$, then  the minimum of $|b|$ in $[0,t]\times \T$ is attained in some $(t_\star,x_\star)$ with $t_\star>0$. Then, $\partial_t (|b|^2)(t_\star,x_\star)\leq 0$. Note that $x_\star$ is a minimum of the function $|b(t_\star,\cdot)|^2$  too. Therefore, at such $t_\star$, $\partial_x u=\left(|b|^2-\|b\|_{L^2}^2\right)/2$ attains a minimum at $x_\star$. Since $|b|^2$ is not identically constant, $\partial_x u(t_\star,x_\star)<0$. As a result, and since $|b(t_\star,x_\star)|>0$ by definition of $t_\star$,  
		
		$$0>\frac{1}{2}\partial_t (|b|^2)+\frac{1}{2}u\partial_x (|b|^2)+\partial_x u (|b|^2)=0 \quad \text{on }(t_\star,x_\star),$$
	which is a contradiction. As a result, $|b(t,x)|^2\geq c_0$ for every positive time.
		
		We now prove the decay estimate \eqref{eq:expdecayu}. Integrating \eqref{eq:bnodifusive} we can deduce an equation for $\partial_x u$, that reads 
		
		$$\partial_t (\partial_x u)+u\partial_x (\partial_x u)+|b|^2\partial_x u=\int_{\T}(\partial_x u)^2\,\dx\geq 0.$$
		
		Using the maximum principle and the fact that $|b(t,x)|^2\geq c_0$,  we can estimate $\partial_x u$ from below by 
		
		$$\partial_x u\geq -\|\partial_x u(0,\cdot)\|_{L^\infty}e^{-c_0t}.$$
		
		On the other hand, since $\|b\|_{L^\infty}^2\leq \|b_0\|^2_{L^\infty}$, we can bound $\partial_x u$ from above by 
		
		\begin{equation} \label{eq:bound}
        \partial_x u\leq  \|\partial_x u(0,\cdot)\|_{L^\infty}e^{-c_0 t}+\int_0^te^{-c_0(t-s)}\|\partial_x u(s,\cdot)\|_{L^2}\ds.
        \end{equation}
		
		Now, since the integral of $\partial_x u$ equals zero, we find that 

        $$\int_\T (\partial_x u)^+=\int_\T (\partial_x u)^-.$$
        Therefore, and due to the fact that $\|\partial_x u\|_{L^\infty}\leq \|b\|_{L^\infty}^2\leq \|b_0\|^2_{L^\infty}$, we find that 
		
		\begin{equation} \label{eq:decaylinftynormpsi}
			\begin{split} 
			\|\partial_x u(t,\cdot)\|_{L^2}^2&=\int_{\T}|\partial_x u|^2\dx\\
			&\leq \|b_0\|_{L^\infty}^2\int_\T |\partial_x u|\\
			&=    \|b_0\|_{L^\infty}^2\int_\T ((\partial_x u^+)+(\partial_x u)^-)\\
			&\leq 2\|b_0\|_{L^\infty}^2\int_\T (\partial_x u)^-\\
			&\leq 2\|b_0\|_{L^\infty}^2\|\partial_x u(t,\cdot)\|_{L^\infty}e^{-c_0t}.
			\end{split} 
		\end{equation}
	
	As a result, plugging this into \eqref{eq:bound}, we obtain 
	
	$$\|\partial_x u(t,\cdot)\|_{L^\infty}\leq 2\|b_0\|_{L^\infty}^2\|\partial_x u(t,\cdot)\|_{L^\infty}e^{-c_0t}.$$
	
	Thus, the result follows.
	\end{proof}

	This, together with the a priori estimates, leads to the following conclusion
	
	\begin{coro}
		Let $b$ be a $C^1([0,\infty),\T)$ solution of \eqref{eq:hyperbolicSystem} satisfying $|b_0|^2\geq c_0$ for some $c_0>0$. Then,
        
		$$|b(t,\cdot)|^2\longrightarrow \|b_0\|^2_{L^1}$$
		uniformly as $t\longrightarrow\infty$. Furthermore, there is a positive constant $C=C(\|b_0\|_{L^\infty})$ independent of $c_0$ such that 
		
		$$||b(t,\cdot)|^2-\|b_0\|^2_{L^1}|\leq \frac{C}{c_0}e^{-c_0t}.$$
	\end{coro}
	\begin{proof}
		Notice that $\|b\|_{L^2}^2$ converges to $\|b_0\|_
        {L^1}^2$ due to Proposition \ref{prop:L2convergencehyperbolic}. On other hand, $\partial_x u=\frac{1}{2}(|b|^2-\|b\|_{L^2}^2)$ satisfies that 
		
		$$\|\partial_x u(t,\cdot)\|_{L^\infty}\leq Ce^{-c_0t}.$$
		
		As a result, since $|b|^2-\|b_0\|_
        {L^1}^2=|b|^2-\|b\|_{L^2}^2+\|b\|_{L^2}^2-\|b_0\|_
        {L^1}^2$, we obtain that $|b(t)|$ converges uniformly to $\|b_0\|_{L^1}$ as $t\rightarrow \infty$. To obtain the precise rate of convergence, we just need to integrate \eqref{eq:energyestimateb} with respect to time, so 
		$$|\|b(t,\cdot)\|_{L^2}^2-\|b_0\|_
        {L^1}^2|\leq \int_t^\infty \left|\frac{d}{dt}\|b\|^2_{L^2}\right|\ds\leq C\int_t^\infty e^{-c_0s}\ds\leq \frac{C}{c_0}e^{-c_0t},$$
		and the result follows.
	\end{proof}

	Finally, we can prove that the solution $b$ itself converges to a function with constant modulus.
	\begin{prop}\label{prop:binfty}
		Let $b$ be a $C^1([0,\infty),\T)$ solution of \eqref{eq:hyperbolicSystem}. Assume that for every $x\in \T$, $|b_0(x)|^2\geq c_0$ for some $c_0>0$. Then, $b$ converges uniformly as $t\rightarrow \infty$ to a function $b_\infty\in \text{Lip}(\T)$ with $|b_\infty(x)|$ independent of $x$. Furthermore, 
		
		$$\|b(t,\cdot)-b_\infty\|_{L^\infty}\leq Ce^{-{c_0}t}.$$
	\end{prop}
	\begin{proof}
		We may use that, since $b$ is $C^1$, it can be represented using the method of characteristics. Indeed, due to the regularity of the solution $b$, $u\in C^0([0,\infty),\text{Lip}(\T))$. Furthermore, due to Proposition \ref{prop:decayunondiff} there is some $C>0$ such that the Lipschitz norm of $u$ is bounded from above by   $Ce^{-c_0t}$. Now, recall that $b$ can be obtained uniquely as $b=Z(\xi(t,x),t)$, where $\xi(t,x)$ is defined via \eqref{eq:eqxi}
	and $Z$ satisfies 
	
	$$\left\{\begin{array}{ll}
	\partial_tZ(t,\xi)=-Z(t,\xi)\partial_xu(t,\Phi_t(\xi)) & t>0,\, \xi\in \T\\
		Z(\xi,0)=b_0(\xi)&
	\end{array}\right..$$

    Now, we know that $\|b(t,x)\|_{L^\infty}\leq \|b_0\|_{L^\infty}.$ As a result, for $0\leq s<t$, 
	
	$$|\Phi_t (\xi)-\Phi_s(\xi)|\leq \int_s^t|u(\Phi_\tau(\xi),\tau)|d\tau\leq \frac{C}{c_0}\left(e^{-c_0s}-e^{-c_0t}\right),$$
	and 
	\begin{equation} \label{eq:convcharacteristics}
    |Z(\xi,t)-Z(\xi,s)|\leq C\frac{\|b_0\|_{L^\infty}}{c_0}\left(e^{-c_0s}-e^{-c_0t}\right).
    \end{equation}
	
	We can also derive an estimate for the functions $\xi(t,x)$. Indeed, differentiating \eqref{eq:defiphi} with respect to $\xi$, and by means of the chain rule, we find that $\xi$ satisfies the ODE 

    $$\left\{\begin{array}{ll}
        \frac{\partial \xi}{\partial t}(t,x)=-\left(\frac{\partial \Phi_t}{\partial \xi}(\xi(t,x))\right)^{-1}u(t,\Phi_t(\xi(t,x))). &  (t,x)\in \R_+\times \T \\
        \xi(0,x)=x &  
    \end{array}\right.$$
    Since 
    $$\frac{\partial\Phi_t}{\partial\xi}=\exp\left(\int_0^t\partial_x u(t,\Phi_t(\xi))\right),$$
    we find that 

    \begin{equation}\label{eq:convinverse} 
    \left|\xi(t,x)-\xi(s,x)\right|\leq C(e^{-c_0s}-e^{-c_0t}).
    \end{equation}
    Then, both estimates \eqref{eq:convcharacteristics} and \eqref{eq:convinverse} imply that there exists a limit function $b_\infty\in C^0(\T)$ such that $b(t,\cdot)\longrightarrow b_\infty(\cdot)$ uniformly as $t\rightarrow \infty$.

    Finally, in order to show that the limit function is Lipschitz we can use the decay of $\partial_x u$ to prove that, out of \eqref{eq:estimatederivative}, the derivative of $b(t,\cdot)$ remains bounded. Recall the formula \eqref{eq:estimatederivative} for the derivative of $\partial_\xi Z$:

    \begin{equation*} 
    \frac{1}{2}\partial_t |\partial_\xi Z(t,\xi)|^2=-(Z\cdot\partial_\xi Z)^2-\frac{1}{2}\left(|Z(\xi,t)|^2-\int|b|^2\right)|\partial_\xi Z|^2\leq -\frac{1}{2}\partial_x u(t,\Phi_t(\xi))\cdot |\partial_\xi Z(t,\xi)|^2.
    \end{equation*}

    Due to Proposition \ref{prop:decayunondiff}, and by means of Grönwall's lemma, we deduce that 

    $$|\partial_\xi Z(t,x)|\leq C,$$
    for some constant $C>0$ depending only on $b_0$. Since $b(t,x)=Z(t,\xi(t,x))$, 

    $$\partial_x b(t,x)=\partial_\xi Z(t,\xi(t,x))\cdot \partial_x\xi(t,x).$$

    Due to the chain rule,

    $$\partial_x\xi(t,x)=\left(\partial_\xi \Phi_t(\xi(t,x))\right)^{-1}=\exp\left(-\int_0^t\partial_x u(s,x)\ds\right).$$

    Again, due to Proposition \ref{prop:decayunondiff}, we conclude that there is a positive constant $C>0$, independent of time, such that $|\partial_x b(t,x)|\leq C$. As a result, the limit function $b_\infty$ lies in $\text{Lip}(\T)$.
    
	\end{proof}

    Out of this proposition we infer that for any choice of initial data with a modulus that never vanishes, there is a unique constant function $b_\infty$ so that $b(t,\cdot)\rightarrow b_\infty$ uniformly exponentially fast. In other words, for times of order $\log(\varepsilon^{-1})$, the function $b(t,\cdot)$ approaches $b_\infty$ up to a correction of order $\varepsilon$. This will be used in the sequel to give an initial condition to a limit problem. 

    \begin{defi}\label{defi:nonlinearoperator}
        Let us define the following sets of functions
        
        $$X:=\{f\in C^1(\T;\R^2)\,:\, |f(x)|>0\text{ for every }x\in \T\}$$
        and 

        $$Y:=\{f\in \text{Lip}(\T;\R^2)\, :\, |f(x)|\text{ is constant.}\}.$$

        Then, we define the map $S$ as 

        \begin{equation}\label{eq:defiSb}
        \begin{array}{rcl}
            S: X & \longrightarrow & Y \\
            b_0(\cdot) & \mapsto & S(b_0)(\cdot):=\lim_{t\rightarrow \infty}b(t,\cdot)
       \end{array}
       \end{equation}
        where $b(t,\cdot)$ is the solution of \eqref{eq:hyperbolicSystem} with initial condition $b_0$. 
    \end{defi}
    \begin{obs}
        Note that the limit in \eqref{eq:defiSb} is a well defined as a limit in the uniform topology and the limit belongs to $Y$, due to Proposition \ref{prop:binfty}.
    \end{obs}

\section{Global Well Posedness for positive Resistivity}\label{sec:epsilonfix}

 In order to obtain well posedness results for the equation, we shall study some of the basic energy estimates. These are obtained first formally assuming that the solution exists, and they will serve as a foundation for the rigorous proof of well-posedness. The simplest one is obtained after multiplying the equation times $b$, leading to 

$$\frac{1}{2}\frac{d}{dt}\int_\T |b(t,x)|^2\dx+\int_\T \partial_x (ub)b=-\varepsilon \int_\T |\partial_x b(t,x)|^2.$$

Using that $\partial_x u:=\frac{1}{2}\left(|b|^2-\int_{\T}|b|^2\dx\right)$, as well as integration by parts, we find 

\begin{equation} \label{eq:apriorib}
\frac{1}{2}\frac{d}{dt}\int_\T|b(t,x)|^2+\int_\T (\partial_x u)^2+\varepsilon \int_\T |\partial_x b(t,x)|^2 \dx=0.
\end{equation}

We can perform similar estimates for the derivative. Differentiating \eqref{eq:magneticrelaxation}, we obtain 

$$\partial_t \partial_x b+u\partial_x ^2 b+2\partial_x u\partial_x b+\partial_x ^2 u b=\varepsilon \partial_x^2(\partial_x b).$$

Again, multiplying the equation by $\partial_x b$, we obtain 

\begin{equation}\label{eq:aprioriparb}
\frac{1}{2}\frac{d}{dt}\int_\T |\partial_x b|^2+\frac{3}{2}\int_\T\partial_x u |\partial_x b|^2+\int_\T (b\cdot \partial_x b)^2+\varepsilon \int_\T |\partial_x ^2b|^2=0.
\end{equation}

\begin{teor}\label{teor:wellposednesspar}
   Let $b_0\in H^2(\T)$. Then, for $T>0$, there exists a unique solution $b\in L^2([0,T];H^2)\cap C([0,T];L^2)$ of the magnetic relaxation system \eqref{eq:magneticrelaxation}. Furthermore, the solution is in $C^\infty((0,T)\times \T)$.
\end{teor}
\begin{proof}
    Let us fix $T>0$ and define $I=[0,T]$. We shall proceed in a rather standard way, exploiting inequalities \eqref{eq:apriorib}-\eqref{eq:aprioriparb} by means of an approximation procedure. We use a Galerkin scheme, by finding approximate solutions in finite dimensional vector spaces. Consider $X_n$ the space of $C^\infty(\T)$ functions defined by 

    $$X_n:=\text{span}\{e^{ikx}\,:\, k\in\Z,\,|k|\leq n\}.$$
    Now, we shall study an approximate version of the system. We define the approximate system, \eqref{eq:approxsystem}, as 

    \begin{equation}\label{eq:approxsystem}
        \left\{\begin{split}
            &\partial_t b_n+\mathbb{P}_n\partial_x (u_n b_n)=\varepsilon\partial_x^2b_n\\
            &2\partial_x u_n=|b_n|^2-\|b_n\|_{L^2}^2\\
            &b_n(0)=\mathbb{P}_nb_0
        \end{split}\right.,\tag{$P_{n}$}
    \end{equation}
    where $\mathbb{P}_n:L^2(\T)\longrightarrow X_n$ is the orthogonal projection onto $X_n$ with respect to the $L^2$ inner product. It is easy to see, by means of Picard-Lindelöf's Theorem, that this problem admits a unique solution in $C^1(I, X_n)$. We then check that the sequence $(b_n)_{n=1}^\infty$ converges in an appropriate sense to a solution $b$ of our problem. Again, by multiplying the equation by $b_n$, we obtain the following version of the a priori estimate \eqref{eq:apriorib}

    \begin{equation} \label{eq:aprioriaprox}
    \frac{1}{2}\frac{d}{dt}\int_\T |b_n|^2+\int_\T (\partial_x u_n)^2=-\varepsilon \int_\T|\partial_x b_n|^2.
    \end{equation}

    This means that $\|b_n\|_{L^2}$ is bounded in $L^\infty(I,L^2(\T))$ by a constant independent of $\varepsilon$ or $n$. As a result, up to the choice of a subsequence, $b_n\xrightharpoonup{\star} b$ in $L^\infty(I;L^2(\T))$. Analogously, we also find that $\partial_x b_n$ is also bounded in $L^2(I;L^2(\T))$ by a constant (this time, depending on $\varepsilon$), so (again, up to subsequence) $b_n\rightharpoonup b$ in $L^2(I,\dot{H}^2(\T))$. Regarding the velocity $u$, as it has zero mean, and as $b_n$ is bounded in $L^\infty(I;L^2(\T))$, we deduce immediately that $u_n$ is bounded in $L^\infty(I;H^1(\T))$. To deal with these non-linear terms, we resort to the Aubin-Lions' Lemma \cite{Aubin1963 }. Note that we have a chain of compact embeddings $H^1(\T)\hookrightarrow L^2(\T)\hookrightarrow H^{-1}(\T)$. Furthermore, if $\varphi$ is an arbitrary element in $H^1$, then 

    $$\int_{\T}\partial_t b(t,x) \varphi(t,x)\,\dx=-\varepsilon\int_{\T} \partial_x b(t,x)\partial_x \varphi(t,x)\,\dx-\int_{\T} u_n(t,x) b_n(t,x)\partial_x \varphi(t,x)\,\dx.$$

    Therefore, 

    $$\|\partial_t b\|_{H^{-1}}\leq \varepsilon \|\partial_x b_n\|_{L^2}+\|u_n\|_{L^\infty}\|b_n\|_{L^2}.$$

    Since $\|u_n\|_{L^\infty}\leq \|b_n\|_{L^2}$, and $\|b_n(t,\cdot)\|_{L^2}^2$ is a decreasing function, we conclude that 

    $$\|\partial_t b_n\|_{L^2(I;H^{-1})}\leq C,$$
    with $C$ only depending on $T$ and $\varepsilon$, not on $n$. Due to Aubin-Lions Lemma, up to subsequence, $b_n\rightarrow b$ in $L^2(I;L^2(\T))$. This also implies the convergence $u_n\rightarrow u$ in $L^2(I;H^1(\T))$.  This means that the limit is a weak solution $b\in L^\infty(I;L^2(\T))\cap L^2(I;H^1(\T))$ with derivative $\partial_t u\in L^2(I;H^{-1}(\T))$.
    
    Notice that, since we are dealing with more regular initial data, we can obtain higher regularity. Since $\mathbb{P}_n$ and $\partial_x$ commute, we obtain again the same a priori estimate as in \eqref{eq:aprioriparb}. Since $\|b_n\|_{L^2}^2$ is decreasing, we find that 

    $$\frac{d}{dt}\|\partial_x b_n\|^2_{L^2}\leq \frac{3}{2}\|b_{n,0}\|_{L^2}\|\partial_x b\|_{L^2}^2.$$

    Therefore, for a fixed time $T$, we find that 
    $\partial_xb_n$ is bounded in $L^\infty((0,T);L^2)$ by a constant depending only on $T$. Therefore, since $f\in\{ L^\infty((0,T);L^2(\T))\,:\,\|f\|_{L^\infty((0,T);L^2(\T)}\leq C\}$ is a closed convex subset of $L^2((0,T);L^2)$, by Mazur's Lemma (\textit{c.f. } \cite[Theorem 3.7]{Brezis}), the function $\partial_x b$ lies in $L^\infty((0,T);L^2)$. As we did to obtain \eqref{eq:aprioriaprox}, we can obtain a version of the same estimate as in estimate \eqref{eq:aprioriparb}, and we find that $\partial_x^2 b_n$ is bounded in $L^2((0,T);L^2)$ by a constant only depending on $T$ and $\varepsilon$. We can go further, and see that if we use $\partial_t b_n$ as a test function in \eqref{eq:approxsystem}, we obtain

    $$\|\partial_t b_n\|_{L^2}^2+ \frac{\varepsilon}{2}\frac{d}{dt}\|\partial _xb_n\|^2_{L^2}\leq \|b_0\|_{L^2}^2\|\partial_x b\|_{L^2}\|\partial_t b\|_{L^2}.$$

    By means of Young's inequality, we find that $\partial_t b_n$ is bounded in $L^2(I;L^2(\T))$, and that $\partial_x b_n$ is bounded in $L^\infty(I;L^2(\T))$. In conclusion, we obtain that the function $b$ obtained as the limit of a subsequence of $b_n$ is in $L^2([0,T];H^2(\T))$. Furthermore, it holds that the derivatives $\partial_t b$ lies in $L^2([0,T];L^2(\T))$. Therefore, due to Lions-Magenes Theorem \cite[Theorem II.5.12]{BF2013}, $b\in C([0,T];L^2(\T))$. 

    In order to prove uniqueness, assume that we have two solutions for \eqref{eq:magneticrelaxation}, $b^1$ and $b^2$, in $C(I;L^2(\T))\cap L^2(I;H^2(\T))$. Then, integrating by parts
\small
    $$\frac{1}{2}\frac{d}{dt}\int_{\T}|b^1(t,x)-b^2(t,x)|^2\dx= -\varepsilon \int_{\T}|\partial_x (b^1-b^2)|^2\dx+\int_{\T}u^2(b^1-b^2)(\partial_x b^1-\partial_x b^2)-\int_{\T}(u^1-u^2)b^2\partial_x(b^1-b^2).$$
    \normalsize

    Now, $u^2(t,x)$ vanishes at $x=0$, and $\partial_x u^2=\frac{1}{2}\left(|b|^2-\|b(t,\cdot)\|_{L^2(\T)}\right)$. Therefore, 

    $$\int_{\T}u^2(b^1-b^2)(\partial_x b^1-\partial_x b^2)\leq C\|b^2(t,\cdot)\|_{L^2(\T)}\|b^1-b^2\|_{L^2(\T)}\|\partial_x b^1-\partial_x b^2\|_{L^2(\T)}$$
    Analogously, 

    $$\int_{\T}(u^1-u^2)b^2\partial_x(b^1-b^2).\leq \|b^2\|_{L^2}\|b^1-b^2\|_{L^2}\|b^1+b^2\|_{L^2}\|\partial_xb^1-\partial_xb^2\|_{L^2}.$$
We then conclude, using Young's inequality, that 

    $$\frac{1}{2}\frac{d}{dt}\int_{\T}|b^1(t,x)-b^2(t,x)|^2\dx\leq C\|b^2\|_{L^2}^2\|b^1+b^2\|_{L^2}^2\|b^1-b^2\|_{L^2}^2.$$

    By means of Gronwall's lemma, uniqueness follows. 
    
    In order to prove higher regularity, we can perform higher order estimates for the solutions of \eqref{eq:approxsystem}. Differentiating the equation, we find that 

    $$\partial_t \partial_x^kb_n=\varepsilon\partial_x^{k+2}b_n-\P_n\partial_x^{k+1} \left(u_n b_n\right).$$
    We now test against $\partial_x^kb_n$ and integrate by parts, resulting in 

    $$\frac{1}{2}\frac{d}{dt}\|\partial_x^kb_n\|_{L^2}^2+\varepsilon\|\partial_x^{k+1}b_n\|_{L^2}^2=\int_{\T}\partial_x^{k+1} b_n\cdot \partial_x^{k}(u_nb_n)\dx.$$

    Due to the embedding $W^{1,1}(\T)\hookrightarrow L^\infty$, we conclude that 

    $$\frac{1}{2}\frac{d}{dt}\|\partial_x^kb_n\|_{L^2}^2+C\|\partial_x^{k+1}b_n\|_{L^2}^2\leq \|b_n\|_{H^k}^2\|b_n\|_{H^{k-1}}^2.$$
    In integral form, the inequality above reads 

    \begin{equation}\label{eq:aprioriderivatives}
    \frac{1}{2}\|\partial_x^kb_n(t,\cdot)\|_{L^2(\T)}^2-\frac{1}{2}\|\partial_x^kb_n(t',\cdot)\|_{L^2(\T)}^2+C\int_{t'}^t\|\partial_x ^{k+1}b_n(s,\cdot)ds\leq \int_{t'}^t\|b_n(s,\cdot)\|_{H^k}^2\|b_n(s,\cdot)\|_{L^2(\T)}^{k-1}
    \end{equation}
    We can now prove by induction that $b\in C^\infty((0,T)\times \T)$. To that end, take $t>0$. Since $b\in L^2(I;H^2(\T))$, there is some $t_2<t$ such that $b(t_2,\cdot)\in H^2(\T)$. Now, due to \eqref{eq:aprioriderivatives} for $k=3$, and since $\partial_x b_n$ is bounded in $L^2(I;L^2(\T))\cap L^\infty (I;L^2(\T))$, we find that the limit $b$ is bounded in $L^2([t_2,T];H^3(\T))\cap L^\infty([t_2,T];H^2(\T))$. Therefore, there exists some $t_2<t_3<t$ such that $b(t_3,\cdot)\in H^3(\T)$. We can then repeat the argument inductively to find that $b\in L^\infty_{loc}(I;H^k(\T))$ for every $k\in \mathbb{N}$. Finally, to obtain regularity in time, we notice that 

    $$\partial_t b_n=\varepsilon \partial_x^2b_n-\P_n\partial_x (u_nb_n).$$

    Since the right hand side is bounded in $L^{\infty}_{loc}(I;H^k(\T))$ for every $k\in\mathbb{N}$, we deduce that $b\in W^{1,\infty}_
    {loc}(I;H^k(\T))$ for all $k\in \mathbb{N}$. It is then immediate, repeatedly differentiating the equation  above with respect to $t$, that $b\in C^\infty((0,T)\times \T)$. 
\end{proof}

	\section{Convergence to the Limit Problem}\label{sec:Munchhausenestimates}

    In this section we prove the existence of a limit function when we take $\varepsilon$ tending to zero. We will also give a characterization of such function in terms of the PDE it satisfies. One interesting feature of this limit equation is the fact that its solutions can exhibit blow up in finite time, as we will show in Section \ref{sec:limiteqandblowup}. 

    The main Theorem in this section reads as follows: 

    \begin{teor}\label{teor:main}
        Let $b_0\in H^2(\T)$ such that $|b_0|>c_0>0$ and let $b_\varepsilon$ be the global solution of the system \eqref{eq:magneticrelaxation} with initial data $b_0$. Consider the time-rescaled functions $b_\varepsilon (\tau,x):=b_\varepsilon (\varepsilon t,x)$ and denote them by $b_\varepsilon$, in a slight abuse of notation. Then, the sequence $(b_\varepsilon)_{\varepsilon>0}$ converges in $C^1_{loc}((0,T^\star)\times \T)$ to the solution $b$ of the system \eqref{ieq:LimitDynamics} with initial value $S(b_0)$, with the operator $S$ as in Definition \ref{defi:nonlinearoperator}, and $T^\star$ the maximum time of existence of such solution.
    \end{teor}
    written with
    
    In order to prove the result, we need to improve the bounds we obtain from the estimates in \eqref{eq:apriorib} \eqref{eq:aprioriparb}. We do this in Proposition \ref{prop:munchhausen}, where we prove that $b_\varepsilon$ is bounded in $L^\infty((0,\delta)\times \T)$ for some small $\delta>0$. This will provide compactness to prove the existence of a limit object for times up to $\delta$. As we show in Proposition \ref{prop:limiteq}, such limit object is a solution of the system \eqref{ieq:LimitDynamics}. Finally, we show in Prop \ref{prop:main} that this convergence can be extended up to the time of blow up, thus finishing the proof of Theorem \ref{teor:main}
    
    \begin{lemma}\label{lemma:lowerboundb}
		Consider $b$ the solution of the magnetic relaxation problem \eqref{eq:magneticrelaxation} with initial data $b_0\in H^2(\T)$ satisfying that $c_0:=\min\{|b_0(x)|^2\}>0$. Now, assume that there exists a time $T_\varepsilon>0$ and a constant $C_\varepsilon>0$, depending on $\varepsilon$, such that $\|\partial_x b(t,x)\|^2_{L^\infty}\leq C_\varepsilon$ for $t<T_\varepsilon$. Then, for 
		
	\begin{equation}\label{eq:timelapse}
    t\leq \textit{\textsf{T}}:=\min\left\{T_\varepsilon,\frac{c_0}{4C_\varepsilon\varepsilon}\right\},
    \end{equation}
		the magnetic field $b(t,x)$ satisfies
		
		$$|b(t,x)|^2\geq c_0/2.$$
	\end{lemma}
	\begin{proof}
		First of all, we study the equation for $|b(t,x)|^2$. To that end, we only need to multiply the system by $b$. Then, $|b|^2$ satisfies
		
		\begin{equation}\label{eq:b2}
			\partial_t |b|^2+u\partial_x( |b|^2)+2\partial_x u|b|^2=\varepsilon \partial_x^2(|b|^2)-2\varepsilon|\partial_x b|^2.
		\end{equation}
		
		Take $\alpha>0$ and define the function $f$ by
		
		$$f(t)=c_0-(\alpha+2\varepsilon C_\varepsilon)t.$$
		Now,   $|b|^2-f$ satisfies the PDE 
		
		\begin{equation}\label{eq:subsol} 
			\partial_t (|b|^2-f)+u\partial_x (|b|^2-f)+2\partial_x u|b|^2=\varepsilon \partial_x ^2(|b|^2-f)+2\varepsilon(C_\varepsilon-|\partial_x b|^2)+\alpha> \partial_x ^2(|b|^2-f).
		\end{equation}
		
		Now, we prove that that $|b|^2-f\geq 0$ for every positive time $t\leq\textit{\textsf{T}}$. We argue by contradiction. Therefore, since $|b|^2\geq f$ at time $t=0$, there is  $(t_\star,x_\star)$ with $\textit{\textsf{T}}\geq t_\star>0$ such that 
        
        $$|b(t_\star,x_\star)|^2=\text{argmin}\{|b(t,x)|^2\,:\, t\leq t_\star\}<f(t_\star).$$ 
        Now, at such point, it holds that 
		
		\begin{itemize}
			\item $\partial_t (|b|^2-f)\leq 0$.
			\item $\partial_x(|b|^2-f)=0.$
			\item $\partial_x^2(|b|^2-f)\geq 0.$
			\item If $(t_\star,x_\star)$ is a minimum, then $x_\star$ is a minimum for $|b(t_\star,\cdot)|^2$, so it also holds that $\partial_x u\leq 0$ at $(t_\star,x_\star)$.
		\end{itemize}
	
		Furthermore, since we are at $t\leq \textit{\textsf{T}}\leq T_\varepsilon$, it holds that $|\partial_x b(t_\star,x_\star)|^2\leq C_\varepsilon$. As a result, due to \eqref{eq:subsol}, we obtain $0>0$, which yields a  contradiction. 
		
		Thus, $|b|^2+(\alpha+2\varepsilon C)t-c_0\geq 0$. Taking $\alpha\rightarrow 0$, we find that $|b|^2\geq -2\varepsilon C_\varepsilon t +c_0$ for $t<\textit{\textsf{T}}$. As the right hand side is a decreasing function, we find that for every $t\leq \frac{c_0}{4C_\varepsilon\varepsilon}$, it holds that $|b|^2\geq c_0/2$.
	\end{proof}
	\begin{lemma}\label{lemma:decaypsi}
		Consider $b$ the solution of \eqref{eq:magneticrelaxation}  with initial $b_0\in H^2(\T)$ satisfying $c_0:=\min\{|b_0(x)|^2\}>0$. Assume further that there exists a time $T_\varepsilon>0$ and a constant $C_\varepsilon>0$, depending on $\varepsilon$, such that $\|\partial_x b\|_{L^\infty}\leq C_\varepsilon$ for $t<T_\varepsilon$ and that $|b(t,x)|\geq c_0$ for $t\leq T_\varepsilon$. 
		
		Then, for $t\leq T_\varepsilon$,  $\psi:=\partial_x u$ satisfies
		
		\begin{equation} \label{eq:psidecay}
			-\|\psi_0\|_{L^\infty}e^{-\frac{c_0}{2}t}-\frac{2\varepsilon C_\varepsilon}{c_0}\leq \psi\leq \|\psi_0\|_{L^\infty}e^{-\frac{c_0}{2}t}+\int_0^te^{-(t-s)\frac{c_0}{2}}\left(2\|\psi(s,\cdot)\|_{L^2}^2+\varepsilon \|\partial_x b(s,\cdot)\|_{L^2}^2\right)\ds.
		\end{equation}
	\end{lemma}
	
	\begin{proof}
		As in Lemma \ref{lemma:lowerboundb}, this follows from maximum principle arguments. We will not delve into the full proof, as it follows the same ideas as in Lemma  \ref{lemma:lowerboundb}. Once we have the equation \eqref{eq:b2} for $|b|^2$, we can obtain an equation for $\psi$. Indeed, integrating \eqref{eq:b2} with respect to the space variable, we find that 
		
		\begin{equation}\label{eq:apriori}
			\frac{d}{dt}\int_\T |b|^2+\int_\T \psi^2+\varepsilon\int_\T |\partial_x b|^2=0.
		\end{equation}
		We now substract this to \eqref{eq:b2}, leading to 
		
		$$\partial_t \psi+u\partial_x \psi+\psi|b|^2-\int_\T \psi^2=\varepsilon \partial_x^2 \psi+\varepsilon \left(\int_\T |\partial_x b|^2-|\partial_x b|^2\right).$$
		
		Since $|b|^2$ is bounded from below by $c_0/2$ in $t\leq T_\varepsilon$, we expect a fast decay for $\psi$. We define now the functions $f$ and $g$, as the solutions of the following ODEs
		
		$$\left\{\begin{array}{l}
			f'+\frac{c_0}{2}f=2\|\psi\|^2_{L^2}+\varepsilon \|\partial_x b\|_{L^2}^2    \\
			f(0)=\|\psi_0\|_{L^\infty}  
		\end{array}\right.\quad \text{and}\quad \left\{\begin{array}{l}
			g'+\frac{c_0}{2}g=-\varepsilon C_\varepsilon   \\
			g(0)=-\|\psi_0\|_{L^\infty}  
		\end{array}\right.$$
		This solutions can be computed explicitly, and they are given by
		
		\begin{equation*}
			\begin{split}
				f(t)&=\|\psi_0\|_{L^\infty}e^{-\frac{c_0}{2}t}+\int_0^te^{-(t-s)c_0/2}\left(2\|\psi(s,\cdot)\|_{L^2}^2+\varepsilon \|\partial_x b(s,\cdot)\|_{L^2}^2\right)\ds\\
				g(t)&=-\|\psi_0\|_{L^\infty}e^{-\frac{c_0}{2}t}-\frac{2\varepsilon C_\varepsilon}{c_0}\left(1-e^{-\frac{c_0}{2}t}\right).
			\end{split}
		\end{equation*}
		To conclude the statement, we only need to use the same arguments based on the maximum principle as those in Lemma \ref{lemma:lowerboundb}, which leads to the inequalities
		
		$$g(t)\leq \psi(t,x)\leq f(t),$$
		for every $t\leq T_\varepsilon$. 
	\end{proof}
	\begin{obs}
	 Note that Lemma \ref{lemma:decaypsi} above implies that $\|\psi\|_{L^\infty}$ remains bounded in $[0,\textit{\textsf{T}}\,]$, since due to the estimates \eqref{eq:apriorib}-\eqref{eq:aprioriparb},
     
     $$\int_0^t\left(\|\psi(s,\cdot)\|_{L^2}^2+\varepsilon \|\partial_x b(s,\cdot)\|_{L^2}^2\right) \ds\leq \|b_0\|_{L^2}^2.$$
	\end{obs}
	Note that the lower bound on $\psi$ implies an exponential decay to $O(\varepsilon)$ of $\psi^-$. We can use this to prove that $\|\psi\|_{L^\infty}$ also decays exponentially, by means of the same kind of argument as the one used to derive \eqref{eq:decaylinftynormpsi}.

    In these lemmas, we have assumed that $\partial_x b$ satisfies an $L^\infty$ bound. In the following one, we will assume the opposite, \textit{i.e}. that as long as the conclusions of Lemmas \ref{lemma:lowerboundb} and  \ref{lemma:decaypsi} hold, and prove that this implies an $L^\infty$ bound for $\partial_x b$.
	
	\begin{lemma}\label{lemma:boundpartialb}
		Assume that $b$ is a solution of the magnetic relaxation system \eqref{eq:magneticrelaxation} . Assume that there exists a time $T_\varepsilon$ satisfying  that $|b(t,x)|^2\geq c_0/2$ for every $t\in [0,T_\varepsilon]$, $x\in \T$ and 
		
		$$\psi(t,x)\geq -\|\psi_0\|_{L^\infty}e^{-\frac{c_0}{2}t}-\frac{2C_\varepsilon\varepsilon}{c_0}$$
		for $t\leq T_\varepsilon$. Then, $\partial_x b$ satisfies 
		
		\begin{equation}\label{eq:boundpartialb}
        \|\partial_x b(t,\cdot)\|_{L^\infty}^2\leq \|\partial_x b_0(\cdot)\|_{L^\infty}^2\exp\left(4\varepsilon t \frac{2C}{c_0}+\frac{8}{c_0}\|\psi_0\|_{L^\infty}\right) \quad \text{for }t\leq T_\varepsilon.
        \end{equation}
	\end{lemma}
	\begin{proof}
		Again, we derive an equation for $|\partial_x b|^2$. Then, we find a suitable subsolution to such equation to estimate $|\partial_x b|^2$.
        
        Differentiating the equation for the magnetic relaxation, we infer 
		
		$$\partial_t(\partial_x b)+u\partial_x (\partial_x b)+2\partial_x u\,\partial_x b+\partial_x^2 u \, b=\varepsilon \partial_x^2 (\partial_x b).$$
		We may now multiply by $\partial_x b$ to deduce that 
		
		\begin{equation}\label{eq:abspartialb}
			\partial_t (|\partial_x b|^2)+u\partial_x (|\partial_x b|^2)+4\psi |\partial_x b|^2+2(b\cdot \partial_x b)^2=\varepsilon \partial_x ^2|\partial_x b|^2-2\varepsilon |\partial_x^2 b|^2.
		\end{equation}
		
		Note that $(b\cdot \partial_x b)^2$ is non-negative, so we can bound  $|\partial_x b(t,x)|^2$ from above by the solution of the ODE 
		
		$$\left\{\begin{array}{ll}
			f'(t)-4\left(\|\psi_0\|_{L^\infty}e^{-\frac{c_0}{2}t}+\frac{2C\varepsilon}{c_0}\right)f=0 & 0<t<T_\varepsilon \\
			f(0)=\|\partial_x b_0\|_{L^\infty}^2
		\end{array}\right..$$
		
		The result then follows applying the maximum principle, as $f$ is a supersolution of \eqref{eq:abspartialb} that equals the right hand side of \eqref{eq:boundpartialb}
	\end{proof}
	
	We can now improve the estimates in Lemmas \ref{lemma:lowerboundb}-\ref{lemma:boundpartialb}. More precisely, we can show that there is decay for $\psi$, boundedness of $\partial_x b$ and a lower bound for $|b|(t,x)|^2$ for times of order $1/\varepsilon$. 
	
	\begin{prop}\label{prop:munchhausen}
		Let $b$ be a  solution of the magnetic relaxation system \eqref{eq:magneticrelaxation}  with $|b(0,x)|^2\geq c_0>0$, for any $x\in \T$. Then, there exist $\delta>0$, $C_1$ and $C_2$ positive constants, depending only on $c_0,\, \|b_0\|_{L^\infty},$ and $ \|\partial_x b_0\|_{L^\infty}$  such that, for every $\varepsilon\geq 0$ and $t\leq \delta/\varepsilon$, the following estimates hold true
		\begin{itemize}
			\item $|b(t,x)|^2\geq c_0/2$
			\item $\psi(t,x)\geq -C_2(\|\psi_0\|_{L^\infty}e^{-\frac{c_0}{2}t}+\varepsilon).$
			\item $\|\partial_x b\|_{L^\infty}^2\leq C_1$.
		\end{itemize}
	\end{prop}
	\begin{proof}
		Consider the set 
		
		$$\mathcal{A}^\varepsilon_\gamma:=\left\{t>0\,:\, \min_{x\in \T}|b(t,x)|^2\geq c_0/2 \,\text{ and }\, \|\partial_xb(t)\|_{L^\infty}^2\leq (5\|\partial_x b_0\|^2_{L^\infty}+1)\exp\left(\gamma\right)\right\},$$
		where $\gamma$ is a positive real number to be fixed later. Note that, due to continuity of both $b$ and $\partial_x b$, it  holds that $\mathcal{A}^\varepsilon_\gamma$ is non-empty and closed. Now, let $t_\star^{\gamma,\varepsilon}$ be the largest positive time such that $[0,t_\star^{\varepsilon,\gamma}]\subset \mathcal{A}^\varepsilon_\gamma$, \textit{i.e.}
        
        $$t_\star^{\gamma,\varepsilon}:=\sup\{t>0\,:\, [0,t]\subset \mathcal{A}^\varepsilon_\gamma\}.$$
        
        Note that $t_\star^{\varepsilon,\gamma}$ depends on $\varepsilon$. Due to Lemma \ref{lemma:decaypsi}, for $t\leq t_\star^{\gamma,\varepsilon}$, we can obtain the bound 
		
		\begin{equation}\label{eq:decaypsiMunchhausen}
			-\|\psi_0\|_{L^\infty}e^{-\frac{c_0}{2}t}-\frac{2\varepsilon }{c_0}(5\|\partial_x b_0\|^2_{L^\infty}+1)\exp\left(\gamma\right)\leq \psi.
		\end{equation}

		Now, we use Lemma \ref{lemma:boundpartialb} to conclude that, if $t<t_\star^{\gamma,\varepsilon}$, 
		
		$$\|\partial_x b(t,\cdot)\|^2_{L^\infty}\leq \|\partial_x b_0\|_{L^\infty}^2\exp\left(\frac{8\varepsilon t_\star^{\gamma,\varepsilon}}{c_0}(5\|\partial_x b_0\|_{L^\infty}+1)\exp(\gamma)+\frac{8}{c_0}\|\psi_0\|_{L^\infty}\right).$$
		Now, we use this to prove that $t_\star^{\gamma,\varepsilon}$ is actually of order $1/\varepsilon$, which concludes the proof. Choose $\gamma=2+\frac{16}{c_0}\|\partial_x u_0\|_{L^\infty}$.  Let $t_0:=\delta/\varepsilon$, with $\delta$ given by
		
		$$\delta:=\frac{c_0}{8(5\|\partial_xb_0\|_{L^\infty}+1)\exp(\gamma)}.$$ 
		
		We show now that $t_\star^{\gamma,\varepsilon}\geq t_0$ $\varepsilon>0$. We argue by contradiction. Assume that $t_0>t_\star^{\varepsilon,\gamma}$ for some $\varepsilon_0$. Since $|b|\geq c_0/2$ for $t\leq t_\star^{\varepsilon,\gamma}< t_0$, and the inequality \eqref{eq:decaypsiMunchhausen} holds for $t\leq t_\star^{\gamma,\varepsilon}$, then Lemma \ref{lemma:boundpartialb} implies the following bound for the derivative of $b$ for $t\leq  t_\star^{\varepsilon,\gamma}<\delta/\varepsilon$.
		
		$$\|\partial_x b(t,\cdot)\|^2_{L^\infty}< \|\partial_x b_0\|_{L^\infty}^2\exp\left(\frac{8\delta}{c_0}(5\|\partial_x b_0\|^2_{L^\infty}+1)\exp\left(\gamma\right)+\frac{8}{c_0}\|\psi_0\|_{L^\infty}\right)\leq (5\|\partial_x b_0\|_{L^\infty}+1)\exp(\gamma).$$
	
		As a result, due to continuity, there is some $\alpha>0$ such that for $t\in (0, t_\star^{\varepsilon,\gamma}+\alpha)$, the derivative of the magnetic field satisfies $\|\partial_x b(t)\|_{L^\infty}^2\leq (5\|\partial_x b_0\|_{L^\infty}+1)\exp(\gamma)$. Now, due to Lemma \ref{lemma:lowerboundb}, $|b(t,x)|\geq c_0/2$ for times 
		
		$$t\leq \min\left\{ t_\star^{\varepsilon,\gamma}+\alpha,\frac{\delta}{\varepsilon} \right\}= t_\star^{\varepsilon,\gamma}+\alpha$$
		for $\alpha$ small enough, since $ t_\star^{\varepsilon,\gamma}<\delta/\varepsilon$. This implies that, then, $(0, t_\star^{\varepsilon,\gamma}+\alpha)\subset \mathcal{A}^\varepsilon_\gamma$. But this contradicts the definition of $t_\star^{\varepsilon,\gamma}$. 
	\end{proof}
	
	We can use these statements to find suitable estimates for higher derivatives of both $u$ and $b$ that hold in weaker norms. This will be useful when finding compactness arguments that will lead to convergence of $b$ to weak solutions of a certain limit equation. 
	
	\begin{prop}\label{prop:decaysecondderivativeu}
		Let $b$ and $\delta$ be as in Proposition \ref{prop:munchhausen}. Then, for $\varepsilon<\varepsilon_0$ with $\varepsilon_0$ small enough, there is some $C_3$ such that, for every $t\leq \delta/\varepsilon$, the inequality 
		
		$$\|\partial_x ^2 u(t,\cdot)\|_{L^2}^2\leq C_3(e^{-\frac{c_0}{2}t}+\varepsilon).$$
	\end{prop}
	\begin{proof}
	We  compute the  equation that $\partial_x u$ satisfies, that  equals 
	
	$$\partial_t (\partial_x u)+u\partial_x^2u+|b|^2\partial_x u-\int_\T(\partial_x u)^2=\varepsilon \partial_x^2(\partial_x u)+\varepsilon \int_\T(\partial_x b)^2-\varepsilon (\partial_x b)^2.$$
	
	We now compute an $x$ derivative and test the resulting equation against $\partial_x^2u$. This leads to 
	
	\begin{equation*} 
		\begin{split} 
			\frac{1}{2}\frac{d}{dt}\int_{\T}|\partial_x^2u|^2+\int_{\T} \left(|b|^2+\frac{5}{2}\partial_x u\right)|\partial_x^2u|^2&=-\varepsilon \int_{\T}(\partial_x^3u)^2-\varepsilon \int_\T \partial_x (|\partial_x b|^2)\partial_x^2u\\
			&=-\varepsilon\int_{\T}(\partial_x^3 u)^2+\varepsilon \int_{\T}|\partial_x b|^2\partial_x^3 u\\
			&\leq \frac{\varepsilon}{2}\|\partial_x b(t,\cdot)\|_{L^\infty}^4.
		\end{split}
	\end{equation*}

	The result then follows after using the bounds for $|b|^2$, $\partial_x u$ and $\|\partial_x b(t,\cdot)\|_{L^\infty}$ from Proposition \ref{prop:munchhausen}.
	\end{proof}

We can prove a similar result for the $L^2(\T)$ norm of the second derivative of $b$. 

\begin{prop}\label{prop:boundseconderivb}
	Let $b$ and $\delta$ be as in Proposition \ref{prop:munchhausen}. Then, for $\varepsilon <\varepsilon_0$ for $\varepsilon_0$ small enough, there is some $C_4$ such that, for every $t<\delta/\varepsilon$, $\|\partial_x^2b(t,\cdot)\|_{L^2}\leq C_4.$
\end{prop}
\begin{proof}
	We first derive an equation for the second derivative of $b$, by differentiating \eqref{eq:magneticrelaxation} twice, leading to 
	
	$$\partial_t (\partial_x^2b)+u\partial_x (\partial_x^2b)+3\partial_x u\,\partial_x^2b+3\partial_x^2u\,\partial_x b+\partial_x^3\, b=\varepsilon \partial_x^2(\partial_x ^2b).$$
	
	Now, we test against $\partial_x^2b$, leading to
	
	\begin{equation*}
		\frac{1}{2}\frac{d}{dt}\int_{\T}|\partial_x^2b|^2+\frac{5}{2}\int_{\T}\partial_x u|\partial_x^2b|^2+\frac{3}{2}\int_{\T}\partial_x^2u\partial_x (|\partial_x b|^2)+\int_{\T}\partial_x^3u (b\cdot \partial_x^2b)=-\varepsilon \int_{\T}(\partial_x^3b)^2.
	\end{equation*}

Using that $\partial_x^3u=|\partial_x b|^2+b\cdot\partial_x^2b$, we find that 

	\begin{equation*}
	\frac{1}{2}\frac{d}{dt}\int_{\T}|\partial_x^2b|^2+\frac{5}{2}\int_{\T}\partial_x u|\partial_x^2b|^2+\frac{5}{2}\int_{\T}\partial_x^2u\partial_x (|\partial_x b|^2)+\int_{\T}\partial_x^3u (b\cdot \partial_x^2b)=-\varepsilon \int_{\T}(\partial_x^3b)^2-\int_{\T}(\partial_x^3u)^2.
\end{equation*}

using Proposition \ref{prop:munchhausen} and Proposition \ref{prop:decaysecondderivativeu}, we conclude the result. 
\end{proof}

\subsection{Convergence to a Limit Object}
In this subsection we will prove that the solutions $b_\varepsilon$ (we include now the subscript to make explicit that the solution depends on $\varepsilon$) converge to a solution of the PDE system \ref{ieq:AngleVariable}. Due to the estimates that we computed in Propositions \eqref{prop:munchhausen}-\eqref{prop:boundseconderivb}, the velocity $u$ becomes of order $\varepsilon$ in the norm $L^\infty$ exponentially fast. Therefore, if we take a rescale of the time $\tau:=\varepsilon t$, it is natural to rescale the velocity $U_\varepsilon:=\frac{1}{\varepsilon}u_\varepsilon$, so that $U_\varepsilon$ becomes of order $1$ for times of order $\log(\varepsilon^{-1})$. We will show that the rescaled velocity $U_\varepsilon$ satisfies

$$U_\varepsilon\sim \int_\T (\partial_x \theta_\varepsilon)^2\,\dx-(\partial_x \theta_\varepsilon)^2 \quad \text{as }\varepsilon\rightarrow 0,$$
where $\theta$ is the phase of $b$ when written as $b_\varepsilon=R_\varepsilon e^{i\theta_\varepsilon}.$ We begin by stating a simple compactness result that will allow us to work with the non-linear terms of the equation. From now on, we take the new rescaled time variable $\tau:=\varepsilon t$, and all the functions will be assumed to depend on this time variable. Even though it is possible to prove compactness in $L^2$, the estimates we derived in the previous section are particularly useful to obtain compactness in spaces of continuous functions. Therefore, we include a lemma that will give us boundedness in Hölder spaces and, therefore, compactness in $C^0$. 

\begin{lemma}\label{lema:lemaregularity}
    Let $f\in L^2([0,T]\times \T)$. Then, the solution of the Heat Equation 

    $$\left\{\begin{array}{ll}
       \partial_t u-\partial_x^2 u=f  & t>0,\,x\in \T \\
        u(0,x)=0 & x\in \T
    \end{array}\right.$$
    given by the heat semigroup is in $C^{1/4,1/2}([0,T]\times\T)$. Furthermore, there is a constant $C>0$ independent of $f$ such that 

    $$[u]_{1/4,1/2}\leq C\|f\|_{L^2}$$
\end{lemma}
\begin{proof}
    This is an adaptation of a theorem that can be found in \cite[Exercise 8.8.9]{Krylov1996}. An outline of the proof is included in the indicated reference, but we have not found a detailed proof anywhere in the literature, that is why we include it here. We begin proving the space regularity. To that end, we examine the regularity of $u$. Take $x,y \in \T$ arbitrary. the difference $u(t,x)-u(t,y)$ can be written as 

    \begin{align}  
    \left|u(t,x)-u(t,y)\right|&=\left|\int_0^t\sum_{k\in \Z}e^{-(2\pi k)^2(t-s)}e^{2\pi i kx}\left(1-e^{2\pi ik(y-x)}\right)\widehat{f}(s,k)\ds\right|\nonumber\\
    &\leq \int_0^t\sum_{k\in \Z\setminus \{0\}}e^{-(2\pi k)^2(t-s)}|\sin(\pi k(y-x))||\widehat{f}(s,k)|\ds\nonumber\\
    &\label{eq:integraltermembedding} \leq \|f\|_{L^2([0,T]\times \T)}\left(\int_{-\infty}^t\sum_{k\in \Z\setminus\{0\}}e^{-2(2\pi k)^2(t-s)}\sin^2(\pi k(x-y)) \ds\right)^{1/2}.\\
    &=\|f\|_{L^2(([0,T]\times \T)}\left(\int_{-\infty}^{t-(x-y)^2} [...]\ds +\int_{t-(x-y)^{2}}^t[...]\ds\right)^{1/2}
    \end{align}

    We examine now the last parenthesis in \eqref{eq:integraltermembedding}. This can be divided into the integral from $-\infty$ up to $t-(y-x)^2$, and the integral form $t-(x-y)^2$ up to $t$. We begin with the last one.

    \begin{equation*} 
    \begin{split} 
    \int_{t-(x-y)^2}^t\sum_{k\in \Z\setminus\{0\}}e^{-2(2\pi k)^2(t-s)}\sin^2(\pi k(x-y)) \ds&=2\int_{t-(x-y)^2}^t\sum_{k=1}^\infty e^{-2(2\pi k)^2(t-s)}\sin^2(\pi k(x-y))\ds\\
    &=2\int_{t-(x-y)^2}^t\sum_{k=1}^\infty e^{-2(2\pi k)^2(t-s)}\sin^2(\pi k(x-y))\int_{k-1}^k\dz \ds\\
    &\leq 2 \int_{t-(x-y)^2}^t\sum_{k=1}^\infty \int_{k-1}^ke^{-2(2\pi z)^2(t-s)}\dz\ds\\
    &\leq C\int_{t-(x-y)^2}^t (t-s)^{-1/2}\ds\\
    &\leq C|x-y|,
    \end{split}
    \end{equation*}
    where $C$ is a numerical constant independent of $x,y$ or $t$. On the other hand, the integral between $0$ and $t-(x-y)^2$ can be estimated by

    \begin{equation*} 
    \begin{split} 
    \int_{-\infty}^{t-(x-y)^2}\sum_{k\in \Z\setminus\{0\}}e^{-2(2\pi k)^2(t-s)}\sin^2(\pi k(x-y)) \ds&=2\int^{t-(x-y)^2}_{-\infty}\sum_{k=1}^\infty e^{-2(2\pi k)^2(t-s)}\sin^2(\pi k(x-y))\ds\\
    &\leq C \sum_{k=1}^\infty \sin^2(\pi k(x-y))\int_{-\infty}^{t-(x-y)^2}e^{-2(2\pi k)^2(t-s)} \ds\\
    &\leq C\sum_{k=1}^\infty \frac{e^{-2(2\pi k)^2 (x-y)^2}}{k^2}\sin^2(\pi k(y-x))\\
    & \leq C|x-y|^2\sum_{k=1}^\infty e^{-2(2\pi k)^2(x-y)^2}\\
    &\leq C|x-y|.
    \end{split}
    \end{equation*}

    Therefore, for every $t>0$, it holds that 

    $$|u(t,x)-u(t,y)|\leq C\|f\|_{L^2([0,t]\times \T)}|x-y|^{1/2}.$$

    We now prove the time regularity up to the boundary of the function $u$. This follows by examining $b^h(t',x)-b^h(t,x)$ for $t'>t\geq 0$, and $x\in \T$ arbitrary. This difference reads 

    \begin{equation*}
        \begin{split}
            |u(t',x)-u(t,x)|&\leq \left|\int_{t}^{t'}\sum_{k\in \Z}e^{-(2\pi k)^2(t'-s)}e^{2\pi ki x}\widehat{f}(k,s)\ds\right|\\
            &+\left|\int_0^{t}\sum_{k\in \Z}\left(e^{-(2\pi k)^2(t'-s)}-e^{-(2\pi k)^2(t-s)}\right)e^{2\pi i kx}\widehat{f}(s,k)\ds\right|=:(I)+(II).
        \end{split}
    \end{equation*}
    In order to bound $(I)$, we use H\"older's inequality again, so that 

    \begin{equation} 
    \begin{split} 
    |(I)|&\leq \|f\|_{L^2([0,T]\times \T)}\left(\int_t^{t'}\sum_{k\in\Z} e^{-2(2\pi k)^2(t'-s)}\ds\right)^{1/2}\\
    &\leq C\|f\|_{L^2([0,T]\times \T)}\left(\int_t^{t'}\int_{\R} e^{-2(2\pi z)^2(t'-s)}\dz\ds\right)^{1/2}\\
    &\leq C\|f\|_{L^2([0,T]\times \T)}\left(\int_t^{t'} (t'-s)^{-1/2}\ds\right)^{1/2}\\
    &\leq C\|f\|_{L^2([0,T]\times \T)} |t'-t|^{1/4}.
    \end{split}
    \end{equation}

    For $(II)$, we can use the Fundamental Theorem of Calculus followed by Hölder's inequality again, leading to 

    \begin{equation*} 
    \begin{split} 
    |(II)|&\leq \int_0^t\sum_{k\in \Z\setminus\{0\}}|\widehat{f}(s,k)|\int_{t}^{t'}(2\pi k)^2e^{-(2\pi k)^2(\sigma-s)}d\sigma \ds\\
    &\leq C\|f\|_{L^2([0,T]\times\T)}\int_{t}^{t'}\left(\int_0^t\sum_{k\in\Z\setminus\{0\}}k^4e^{-(2\pi k)^2(\sigma-s)}\ds\right)^{1/2}d\sigma\\
    &\leq C\|f\|_{L^2([0,T]\times\T)}\int_{t}^{t'}\left(\int_0^t\int_\R z^4e^{-(2\pi z)^2(\sigma-s)}\dz\ds\right)^{1/2}d\sigma\\
    &\leq C\|f\|_{L^2([0,T]\times\T)}\int_{t}^{t'}\left(\int_0^t(\sigma-s)^{-5/2}\ds\right)^{1/2}d\sigma\\
    &\leq C\|f\|_{L^2([0,T]\times\T)}\int_{t}^{t'}(\sigma-t)^{3/4}d\sigma\\
    &\leq C\|f\|_{L^2([0,T]\times\T)}|t'-t|^{1/4}.
    \end{split}
    \end{equation*}
\end{proof}

\begin{lemma}\label{lemma:unifconv}
    The sequence of functions $b_\varepsilon$ converge, up to subsequence, to a function b in $C^0_{loc}((0,\delta];C^1(\T))$. This converges takes place in the $C^0_{loc}((0,\delta];C^1(\T))$ topology. 
\end{lemma}
\begin{proof}
    This follows as an application of Arzelà-Ascoli Theorem. 
    Recall that the equation for  $b_\varepsilon$ reads 

    \begin{equation}\label{eq:eqtonta}
    \partial_\tau b_\varepsilon-\partial_x^2b_\varepsilon=-\partial_x (U_\varepsilon b_\varepsilon).
    \end{equation}

    Due to Proposition \ref{prop:munchhausen}, we deduce that for $a>0$, there is a constant $C$ such that both $\partial_x b_\varepsilon$ and $\partial_x U_\varepsilon$ are bounded in $L^\infty([0,a]\times\T)$. As a result, $b_\varepsilon(t,x)$ is bounded in $L^\infty([a,\delta],C^{1/2}(\T))$ and we can use Lemma \ref{lema:lemaregularity} to prove that, actually, $b_\varepsilon$ is bounded in $C^{1/4,1/2}([a,\delta]\times\T)$. Furthermore, Proposition \ref{prop:boundseconderivb} implies also a bound on the $L^\infty([0,\delta];C^{1/2}(\T))$ norm of $\partial_x b(t)$ in $[a,\delta]$. By differentiating \eqref{eq:eqtonta} and using the $L^2$ bounds from Propositions \ref{prop:decaysecondderivativeu} and \ref{prop:boundseconderivb}, we obtain a bound on $C^{1/4,1/2}([a,\delta]\times \T)$ for $\partial_x b$ as well. Since $a>0$ is arbitrary (even though the bounds we have obtained depend on it), we can use Arzelà-Ascoli to prove that, after taking a diagonal subsequence, $b_\varepsilon$ converges in $C^0_{loc}((0,\delta);C^1(\T))$ to a function $b\in C^1_{loc}((0,\delta);C^1(\T))$. 
\end{proof}

The limit function $b$ is a priori not well defined in $t=0$. This is reasonable, as at time zero $U_\varepsilon$ will be generically of order $1/\varepsilon$, so the time derivative become very large in a (shrinking) neighborhood of $t=0$ as time $\varepsilon\rightarrow 0$. Therefore, no compactness in time in a uniform norm is to be expected at small times. However, by means of the hyperbolic system, \eqref{eq:hyperbolicSystem}, we will show in Proposition \ref{prop:initialvalue} that we can prove the existence of $\lim_{t\rightarrow 0}b(t,\cdot)$ if we take the limit in $L^2$.

\begin{obs}
    The limit function $b$ lies in $L^\infty((0,\delta)\times \T)$. 
\end{obs}
\begin{proof}
    This follows after applying the maximum principle, since the function constant equal to $\|b_0\|_{L^\infty}$ is a subsolution of \eqref{eq:magneticrelaxation}. Therefore, $|b_\varepsilon(t,x)|^2\leq \|b_0\|_{L^\infty}^2$ for any positive time. 
\end{proof}

We now want to use this result to find a limiting equation for the object $b$. Before stating the theorem, we shall make use of a technical lemma that will prove useful in the proof of the convergence 

\begin{lemma}\label{lemma:approxidentity}
	Let $f\in C_{loc}^1((0,\delta]\times \T)\cap L^\infty((0,\delta)\times \T)$. Assume that $g:[0,\delta]\longrightarrow \R$ is a continuous function bounded from below by $c_0$. Then, 
	
	$$f^2(\tau)=g(\tau)\lim_{\varepsilon\rightarrow 0}\frac{1}{\varepsilon}\int_0^\tau\exp\left(-\frac{1}{\varepsilon}\int_\sigma^\tau g(\lambda)d\lambda\right)S(\tau-\sigma)f^2(\sigma)d\sigma \quad \text{in } L^1((0,\delta)\times \T).$$
\end{lemma}
\begin{proof}
	Note that the integral  
	
	$$\frac{1}{\varepsilon}\int_0^\tau \exp\left(-\frac{1}{\varepsilon}\int_\sigma^\tau g(\lambda)d\lambda\right)g(\sigma)d\sigma=\frac{1}{\varepsilon}\int_0^{\mu(\tau)}e^{-\frac{\mu(\tau)-\mu}{\varepsilon}}d\mu=1-e^{-\mu(\tau)/\varepsilon}$$
	with 
	
	$$\mu(\tau)=\int_0^\tau g(\lambda)d\lambda.$$
			
	As a result,
		
	\begin{equation*} 
		\begin{split} 
		f^2(x,\tau)&=e^{-\mu(\tau)/\varepsilon}f^2(\tau,x)+\frac{1}{\varepsilon}\int_0^\tau \exp\left(-\frac{1}{\varepsilon}\int_\sigma^\tau g(\lambda)d\lambda\right)g(\sigma)d\sigma f^2(\tau,x)\\
		\end{split}
	\end{equation*}

	Then, we can write  
	\small
	\begin{equation*}
		\begin{split}
			g(\tau)\frac{1}{\varepsilon}\int_0^\tau &\exp\left(-\frac{1}{\varepsilon}\int_\sigma^\tau g(\lambda)d\lambda\right)S(\tau-\sigma)f^2(\sigma)d\sigma-f^2(\tau)\\
			&=e^{-\mu(\tau)/\varepsilon}f^2(\tau,x)+\frac{1}{\varepsilon}\int_0^\tau \exp\left(-\frac{1}{\varepsilon}\int_\sigma^\tau g(\lambda)d\lambda\right)\int_{\T}k_{\tau-\sigma}(y)(g(\tau)f^2(\sigma,x-y)-g(\sigma)f^2(\tau,x))\dy {\rm d}\sigma\\
			&=e^{-\mu(\tau)/\varepsilon}f^2(\tau,x)+f^2(\tau,x)\frac{1}{\varepsilon}\int_0^\tau \exp\left(-\frac{1}{\varepsilon}\int_\sigma^\tau g(\lambda)d\lambda\right)(g(\tau)-g(\sigma))d\sigma\\
			&\phantom{asdfffffffffff}+g(\tau)\frac{1}{\varepsilon}\int_0^\tau \exp\left(-\frac{1}{\varepsilon}\int_\sigma^\tau g(\lambda)d\lambda\right)(S(\tau-\sigma)(f^2(\sigma)-f^2(\tau)) d\sigma\\
			&\phantom{asdfffffffffff}+g(\tau)\frac{1}{\varepsilon}\int_0^\tau \exp\left(-\frac{1}{\varepsilon}\int_\sigma^\tau g(\lambda)d\lambda\right)\left(S(\tau-\sigma)-Id\right)f^2(\tau) d\sigma\\
            &\phantom{asdfffffffffff}=(I)+(II)+(III)+(IV).
		\end{split}
	\end{equation*}
	\normalsize
	We now analyze each term separately. $(I)$ tends to zero in $L^1$. Indeed, $e^{-\mu(\tau)/\varepsilon}$, is bounded by one, and it converges to zero almost everywhere in $[0,\delta)$. As a result, by the Dominated Convergence Theorem, we find that 
	
	$$\int_{0}^\delta \int_{\T}e^{-\mu(\tau)/\varepsilon}f^2(\tau,x)\,\dx {\rm d}\tau=\|f\|^2_{L^\infty((0,\delta)\times \T)}\int_0^\delta e^{-\mu(\tau)/\varepsilon}{\rm d}\tau\xrightarrow{\varepsilon\rightarrow 0} 0.$$
	
	On the other hand, to bound $(II)$, we observe that
	
	\begin{equation}\label{eq:tendstozero}
		\begin{split} 
			\int_0^\delta\int_\T |f^2(\tau,x)|&\left|\frac{1}{\varepsilon}\int_0^\tau \exp\left(-\frac{1}{\varepsilon}\int_\sigma^\tau g(\lambda)d\lambda\right)(g(\tau)-g(\sigma))d\sigma\right|\,d\tau\\
			&\leq \|f\|_{L^\infty((0,\delta)\times \T)}\frac{1}{\varepsilon}\int_0^\delta\int_0^\tau \exp\left(-\frac{c_0}{\varepsilon}(\tau-\sigma)\right)|g(\tau)-g(\sigma)|d\sigma d\tau
		\end{split}
	\end{equation}

        Take $n\in \N$. Since $g$ is uniformly continuous, there exists some $h>0$ such that if $|\tau-\sigma|\leq h$, then $|g(\tau)-g(\sigma)|\leq \frac{1}{n}$. As a result, 

        $$\left|\int_0^\tau \exp\left(-\frac{c_0}{\varepsilon}(\tau-\sigma)\right)|g(\tau)-g(\sigma)|d\sigma\right|\leq C\left(\frac{1}{n}+\|g\|_{L^\infty}e^{-\frac{c_0}{\varepsilon}h}\right).$$
        As a result, we find that for every $\tau>0$, 

        $$\left|\int_0^\tau \exp\left(-\frac{c_0}{\varepsilon}(\tau-\sigma)\right)|g(\tau)-g(\sigma)|d\sigma\right|\longrightarrow 0$$
        as $\varepsilon\rightarrow 0$. By means of the Dominated Convergence Theorem, we conclude that the $L^1$ norm of $(I)$ tends to zero.

	In order to bound $(III)$, we can use that $f\in L^\infty((0,\delta);L^\infty)$, so  
	
	\begin{equation}\label{eq:tendstozero2} 
		\begin{split} 
			\Bigg\|g(\tau)\frac{1}{\varepsilon}\int_0^\tau&\exp\left(-\frac{1}{\varepsilon}\int_\sigma^\tau g(\lambda)d\lambda\right)(S(\tau-\sigma)(f^2(\sigma)-f^2(\tau)) d\sigma\Bigg\|_{L^1([0,\tau]\times \T)}\\
			&\leq \|g\|_{L^\infty}\frac{1}{\varepsilon}\int_0^\delta\int_0^\tau \exp\left(-\frac{1}{\varepsilon}\int_\sigma^\tau g(\lambda)d\lambda\right)\|f^2(\sigma)-f^2(\tau)\|_{L^\infty} d\sigma d\tau\\
			&\leq 2\|g\|_{L^\infty}\frac{1}{\varepsilon}\int_0^\delta\int_0^\tau \exp\left(-\frac{1}{\varepsilon}\int_\sigma^\tau g(\lambda)d\lambda\right)\|f^2(\sigma)-f^2(\tau)\|_{L^\infty} d\sigma d\tau\\
            &\leq2\|g\|_{L^\infty}\frac{1}{\varepsilon}\int_0^\delta\int_0^\tau \exp\left(-\frac{c_0}{\varepsilon}\sigma \right)\|f^2(\tau-\sigma)-f^2(\tau)\|_{L^\infty} d\sigma d\tau
		\end{split}
	\end{equation} 

	Now, since $\tau\mapsto f(\tau)$ is continuous in $(0,\delta)\times \T$ we can repeat a similar argument as in \eqref{eq:tendstozero} to prove that

    $$\frac{1}{\varepsilon}\int_0^\tau \exp\left(-\frac{c_0}{\varepsilon}\sigma \right)\|f^2(\tau-\sigma)-f^2(\tau)\|_{L^\infty} d\sigma\longrightarrow 0\quad \text{for all }\tau>0.$$

    Then, due to the Dominated Convergence Theorem, $(II)$ converges to zero in $L^1(I;L^1(\T))$. Finally, for $(IV)$,
	\begin{equation} \label{eq:finalboundasdf}
		\begin{split} 
	\Bigg\|g(\tau)\frac{1}{\varepsilon}\int_0^\tau &\exp\left(-\frac{1}{\varepsilon}\int_\sigma^\tau g(\lambda)d\lambda\right)\left(S(\tau-\sigma)-Id\right)f^2(\tau) d\sigma\Bigg\|_{L^1((0,\delta)\times \T)}\\
	&\leq |g(\tau)|\frac{1}{\varepsilon}\int_0^\delta\int_0^\tau \exp\left(-\frac{c_0}{\varepsilon}(\tau-\sigma)\right)\|\left(S(\tau-\sigma)-Id\right)f^2(\tau)\|_{L^2} d\sigma d\tau. 
		\end{split}
	\end{equation} 

	Notice that this function tends to zero as $\varepsilon\rightarrow 0$ for almost every $\tau\in (0,\delta)$. Indeed, since $f\in L^\infty((0,\tau);\T),$ the function $f^2(\tau)$ lies in $L^2$ for almost every $\tau$. As a result, due to the continuity properties of the heat semigroup, $(S(\tau-\sigma)-I)f^2$ tends to zero in $L^2$ as $\sigma\rightarrow \tau$ for almost every $\tau$. Furthermore,  \eqref{eq:finalboundasdf} is bounded. Thus, repeating the same arguments as  in \eqref{eq:tendstozero2}, it follows that \eqref{eq:finalboundasdf} tends to zero for almost every $\tau$. As a result, if we now integrate with respect to $\tau$, due to the Dominated Convergence Theorem, we obtain  
	
	\begin{equation} 
			\Bigg\|g(\tau)\frac{1}{\varepsilon}\int_0^\tau \exp\left(-\frac{1}{\varepsilon}\int_\sigma^\tau g(\lambda)d\lambda\right)\left(S(\tau-\sigma)-Id\right)f^2(\tau) d\sigma\Bigg\|_{L^1((0,\delta);L^1)}\longrightarrow 0. 
	\end{equation} 
	\normalsize
\end{proof}

We are now in a position to prove the convergence theorem for $b_\varepsilon$. 
\begin{prop}\label{prop:limiteq}
	The limit function $b=\lim_{\varepsilon\rightarrow 0}b_\varepsilon$ is a weak solution for the system of PDEs
	
	\begin{empheq}[left=\empheqlbrace]{alignat=3}
		\label{eq:limitsystem}	&\partial_t b+\partial_x(Ub)=\partial_x^2b & &\tau\in (0,\delta).\\
		\label{eq:limitsystem2}	&\partial_x U=\frac{1}{|b|^2}\left(\int_\T |\partial_x b|^2-|\partial_x b|^2\right) &\phantom{asfd}& \tau \in (0,\delta).
	\end{empheq}

More precisely, for every $\varphi\in C_c^\infty((0,\delta)\times \T)$, the function $b$ satisfies 

\begin{equation} \label{eq:weaksol}
\int_0^\delta \int_\T \left\{ b\partial_t \varphi+Ub\cdot \partial_x \varphi+ \partial_x b\cdot \partial_x \varphi\right\}\,\dx \dt=0
\end{equation} 

\end{prop}
\begin{obs}
    As pointed out before, the function $b(\tau,\cdot)$ also has constant modulus, \textit{i.e.} $|b(\tau,x)|$ is a function only depending on time. Therefore, we can write $b=R(\tau)e^{i\theta(\tau,x)}$, with both $R$ and $\theta$ real valued, so \eqref{eq:limitsystem}-\eqref{eq:limitsystem2} can be written as 

    \begin{align}
        \partial_t R&=-R\int (\partial_x \theta)^2\dx.\\
        \partial_t \theta+U\partial_x \theta&=\partial_x^2\theta.\\
        \partial_x U&=\int_{\T} (\partial_x\theta)^2\,\dx-(\partial_x \theta)^2.
    \end{align}
\end{obs}
\begin{proof}[Proof of Theorem \ref{prop:limiteq}]
	We know that $b_\varepsilon\longrightarrow b$ strongly in $C_{loc}^1((0,\delta]\times \T)$ as $\varepsilon\rightarrow0$. Therefore, it is enough to show that 
	
	$$\partial_x U_\varepsilon\longrightarrow \frac{1}{|b|^2}\left(\int|\partial_x b|^2-|\partial_x b|^2\right)$$
	in $L^1_{loc}(I;L^1(\T))$.  We find an equation for $\partial_xU_\varepsilon=:\Psi_\varepsilon$:
	
	$$\partial_\tau \Psi_\varepsilon+U_\varepsilon\partial_x\Psi_\varepsilon+\frac{R^2}{\varepsilon}\Psi_\varepsilon=\partial_x^2\Psi_\varepsilon+\frac{1}{\varepsilon}\left(\int_\T |\partial_x b_\varepsilon|^2-|\partial_xb_\varepsilon|^2\right)+\frac{R^2-R^2_\varepsilon}{\varepsilon}\Psi_\varepsilon+\int_\T \Psi_\varepsilon^2\,\dx,$$
	where $R^2_\varepsilon=|b_\varepsilon|^2$. 
	
	Using Duhamel's formula, we find that for any $a,\tau>0$, 
	
	\begin{equation}\label{eq:convergenceu}
		\begin{split}
			\Psi_\varepsilon(\tau+a,x)&=\exp\left(-\frac{1}{\varepsilon}\int_{0}^\tau\|R(\lambda+a)\|_{L^2}^2d\lambda\right)S(\tau+a)\Psi_\varepsilon(a)\\
			&+\frac{1}{\varepsilon}\int_{0}^\tau \exp\left(-\frac{1}{\varepsilon}\int_{\sigma}^\tau\|R(\lambda+a)\|_{L^2}^2d\lambda\right)S(\tau-\sigma)\left[\int_\T|\partial_x b_\varepsilon(\sigma+a,\cdot)|^2-|\partial_x b_\varepsilon(\sigma+a,\cdot)|^2\right]d\sigma\\
			&-\int_{0}^\tau \exp\left(-\frac{1}{\varepsilon}\int_{\sigma}^\tau\|R(\lambda+\alpha)\|_{L^2}^2d\lambda\right)S(\tau-\sigma)\left(U_\varepsilon(\sigma+a,\cdot)\partial_x\Psi_\varepsilon(\sigma+a,\cdot)\right)d\sigma\\
            &+\int_{0}^\tau \exp\left(-\frac{1}{\varepsilon}\int_{\sigma}^\tau\|R(\lambda+\alpha)\|_{L^2}^2d\lambda\right)S(\tau-\sigma)\left(\frac{R(\sigma+a)-R_\varepsilon(\sigma+a)}{\varepsilon}\Psi_\varepsilon(\sigma+a,\cdot)\right)d\sigma\\
			&=(I)+(II)+(III)+(IV).
		\end{split}
	\end{equation}

We now claim that $\Psi_\varepsilon(\tau+a,\cdot)$ tends in $L^1((0,\delta-a)\times \T)$ to 

$$\frac{1}{|b|^2}\left(\int_\T |\partial_x b(\tau+a,y)|^2\dy-|\partial_x b(\tau+a,x)|^2\right).$$

To that end, we estimate the three terms from \eqref{eq:convergenceu}. We begin with $(I)$. Note that,

$$\|(I)\|_{L^1((0,\delta-a);L^1(\T))}\leq \frac{C}{\varepsilon}\int_{a}^\delta e^{-\frac{c_0}{\varepsilon}\tau}d\tau\longrightarrow 0 \quad \text{as }\varepsilon\rightarrow 0$$

The estimate for $(III)$ can be derived in a similar way. We just need to use that 

$$\|S(\tau-\sigma)U\partial_x \Psi_\varepsilon\|_{L^1(\T)}\leq \|S(\tau-\sigma)U\partial_x \Psi_\varepsilon\|_{L^\infty(\T)}\leq \|U\partial_x \Psi_\varepsilon\|_{L^\infty(\T)}\leq \frac{C}{\varepsilon^{2}}\left(e^{-\frac{c_0}{4\varepsilon}\tau}+\varepsilon\right)^{2}.$$

This leads immediately to 

$$\|(III)\|_{L^1((0,\delta-a);L^1(\T))}\longrightarrow 0\quad \text{as }\varepsilon\rightarrow 0.$$

We now study then the term $(II)$. $\partial_x b_\varepsilon(\sigma+a) \longrightarrow \partial_x b(\sigma+a)$ as $\varepsilon\rightarrow 0$ in $C^0((0,\delta-a)\times\T)$ due to Lemma \ref{lemma:unifconv}. Now, due to Lemma \ref{lemma:approxidentity}, we find that 

\begin{equation*} 
\begin{split} 
\lim_{\varepsilon\rightarrow 0}\frac{1}{\varepsilon}\int_{0}^\tau \exp\left(-\frac{1}{\varepsilon}\int_{\sigma}^\tau R(\lambda+a)^2d\lambda\right)&S(\tau-\sigma)\left[\int_\T|\partial_x b(\sigma+a,\cdot)|^2-|\partial_x b(\sigma+a,\cdot)|^2\right]d\sigma\\
&=\frac{1}{R(\tau+a)}\left(\int|\partial_x b(\tau+a,\cdot)|^2-|\partial_x b(\tau+a,\cdot)|^2\right),
\end{split}
\end{equation*}
where the limit is taken in $L^1((0,\delta-a);L^1(\T))$.  As a result, we can write 

\begin{equation*}
	\begin{split}
		&\lim_{\varepsilon\rightarrow 0}\frac{1}{\varepsilon}\int_{0}^\tau \exp\left(-\frac{1}{\varepsilon}\int_{\sigma}^\tau R(\lambda+a)^2d\lambda\right)S(\tau-\sigma)\left[\int_\T|\partial_x b_{\varepsilon}(\tau+a,y)\dy|^2-|\partial_x b_{\varepsilon}(\tau+a,\cdot)|^2\right]d\sigma\\
        &\phantom{asdfasdfasdfasdf}-\frac{1}{R(\tau)}\left(\int_\T|\partial_x b(\tau+a,y)|^2\dy-|\partial_x b(\tau,x)|^2\right)\\
		=&\lim_{\varepsilon\rightarrow 0}\frac{1}{\varepsilon}\int_{0}^\tau \exp\left(-\frac{1}{\varepsilon}\int_{\sigma}^\tau R(\lambda+a)^2d\lambda\right)S(\tau-\sigma)\left[\int_\T\left(|\partial_x b_{\varepsilon}(\tau+a,y)|^2\dy-|\partial_x b(\tau+a,\cdot)|^2\right)\right.\\
        &\phantom{\int asdfasdfasdfasdf}\left.-\left(\int_\T|\partial_x b_{\varepsilon}(\sigma+a,y)|^2\dy-|\partial_x b(\sigma+a,\cdot)|^2\right)\right]d\sigma
	\end{split}
\end{equation*}

The $L^1(I;L^1(\T))$ norm of the term above is bounded by 

\begin{equation}\label{eq:annoying}
	\begin{split}
		\frac{1}{\varepsilon}\int_0^\delta\int_0^\tau e^{-\frac{c_0}{2\varepsilon}(\tau-\sigma)}\|\partial_x b_\varepsilon-\partial_x b\|_{L^2(\T)}d\sigma d\tau.
	\end{split}
\end{equation}

Again,

$$\frac{1}{\varepsilon}\int_0^\tau e^{-\frac{c_0}{2\varepsilon}(\tau-\sigma)}\|\partial_x b_\varepsilon(\sigma+a,\cdot)-\partial_x b(\sigma+a,\cdot)\|_{L^2(\T)}d\sigma \longrightarrow 0$$
as $\varepsilon\rightarrow 0$ for almost every $\tau>0$. Therefore, due to the Dominated Convergence Theorem, we conclude that \eqref{eq:annoying} tends to zero. 

Now, using that $\|R_\varepsilon\|_{L^2}\longrightarrow \|R\|_{L^2}$, we infer that $$\partial_x U_\varepsilon\longrightarrow \frac{1}{|b|^2}\left(\int_\T |\partial_x b|^2-|\partial_x b|^2\right)$$ in $L^1_{loc}(I;L^1(\T))$. The result then follows.
\end{proof}

We prove now that we can make sense of $b(0,\cdot)$ as $\lim_{t\rightarrow 0}b(t,\cdot)$, where the limit is taken with respect to the $L^2$ norm. We also prove that such limit equals $S(b_0)$, where such function is given by Definition \ref{defi:nonlinearoperator}. We establish first a series of auxiliary lemmas.

\begin{lemma}\label{lemma:initialvalue}
	Consider   $\tilde{b}$  the solution of \eqref{eq:hyperbolicSystem}. Then, if the initial data $b_0:=b_\varepsilon(0,\cdot)=\tilde{b}(0,\cdot)\in H^2(\T)$ satisfies $|b_0|\geq c_0/2$, then there exists a constant $C>0$ such that for $t\leq \delta/\varepsilon$ (\textit{c.f. Proposition \ref{prop:munchhausen}}) it holds that 
	
	$$\|b_\varepsilon(t,\cdot)-\tilde{b}(t,\cdot)\|_{L^2}\leq C\varepsilon t.$$
\end{lemma}
\begin{proof}
	Subtracting the equations of both $u_\varepsilon$ and $\tilde{u}$ we obtain 
	
	\begin{equation*}
		\begin{split}
			&\partial_t(\partial_x u_\varepsilon)+u_\varepsilon\partial_x ^2u_\varepsilon+|b_\varepsilon|^2\partial_x u_\varepsilon-\int_\T(\partial_x u_\varepsilon)^2=\varepsilon \partial_x ^2(\partial_x u_\varepsilon)\,+\varepsilon\int_\T|\partial_x b_\varepsilon|^2-\varepsilon|\partial_x b_\varepsilon|^2\\
			&\partial_t(\partial_x \tilde{u})\phantom{_\varepsilon}+\tilde{u}\partial_x ^2\tilde{u}\phantom{_\varepsilon}\phantom{_\varepsilon}+|\tilde{b}|^2\partial_x \tilde{u}\phantom{_\varepsilon}\phantom{_\varepsilon}-\int_\T(\partial_x \tilde{u})^2\phantom{_\varepsilon}=0
		\end{split}
	\end{equation*}

As a result, we obtain an estimate for the $L^2$ norm of $\partial_x u_\varepsilon-\partial_x \tilde{u}$ as

\begin{equation*}
	\begin{split}
		\frac{1}{2}\frac{d}{dt}\int_\T|\partial_x u_\varepsilon-\partial_x \tilde{u}|^2&-\frac{1}{2}\int_\T\partial_x \tilde{u}|\partial_x u_\varepsilon-\partial_x \tilde{u}|^2+\int_\T|\tilde{b}|^2|\partial_x u_\varepsilon-\partial_x \tilde{u}|^2\\
		&+\int_\T(u_\varepsilon-\tilde{u})(\partial_x u_\varepsilon-\partial_x \tilde{u})\partial_x^2u_\varepsilon+\int_\T(|b_\varepsilon|^2-|\tilde{b}|^2)\partial_x u_\varepsilon(\partial_x u_\varepsilon-\partial_x \tilde{u})\\
		&=-\varepsilon\int_\T\partial_x^2 u_\varepsilon(\partial_x^2\tilde{u}-\partial_x^2 u_\varepsilon)-\varepsilon\int_\T|\partial_x b|^2(\partial_x \tilde{u}-\partial_x u_\varepsilon)
	\end{split}
\end{equation*}

Now, since $H^1(\T)$ embedds into $ L^\infty(\T)$ in dimension 1,  we have

$$\left|\int_\T(u_\varepsilon-\tilde{u})(\partial_x u_\varepsilon-\partial_x \tilde{u})\partial_x^2u_\varepsilon\right|\leq \|\partial_x \tilde{u}(t,\cdot)-\partial_x u_\varepsilon(t,\cdot)\|_{L^2}^2\|\partial_x^2 u_\varepsilon(t,\cdot)\|_{L^2}.$$

On the other hand, 

$$\int(|b_\varepsilon|^2-|\tilde{b}|^2)\partial_x u_\varepsilon(\partial_x u_\varepsilon-\partial_x \tilde{u})\leq \|\partial_x u_\varepsilon\|_{L^\infty}\|\partial_x u_\varepsilon-\partial_x \tilde{u}\|_{L^2}^2+(\|b_\varepsilon\|_{L^2}^2-\|\tilde{b}\|_{L^2}^2)\|\partial_x u_\varepsilon\|_{L^2}\|\partial_x u_\varepsilon-\partial_x \tilde{u}\|_{L^2}.$$

Now, using the estimates derived in  Theorem \ref{prop:munchhausen} for solutions of the system \eqref{eq:magneticrelaxation}; as well as the estimates obtained in Proposition \ref{prop:decayunondiff} for \eqref{eq:hyperbolicSystem}, and since $\log(\varepsilon^{-1})=o(\varepsilon^{-1})$ as $\varepsilon\rightarrow 0$, there exists $\varepsilon_0$ such that, for $\varepsilon<\varepsilon_0$, 

$$\frac{1}{2}\frac{d}{dt}\|\partial_x u_\varepsilon-\partial_x \tilde{u}\|_{L^2}+\left(c_0-Ce^{-c_0t}\right)\|\partial_x u_\varepsilon-\partial_x \tilde{u}\|_{L^2}\leq C\left(\varepsilon+\|\tilde{b}-b_{\varepsilon}\|_{L^2}\right)\left(e^{-\frac{c_0}{2}t}+\varepsilon\right).$$

By means of Gronwall's inequality, and since $\partial_x u_\varepsilon=\partial_x \tilde{u}$ at $t=0$, we find that 
\small
\begin{equation} 
	\begin{split} 
		\|\partial_x u_\varepsilon(t,\cdot)-\partial_x \tilde{u}(t,\cdot)\|_{L^2}&\leq C\int_0^t\exp\left(-c_0(t-s)+e^{\frac{C}{c_0}e^{-c_0s}}-e^{\frac{C}{c_0}e^{-c_0t}}\right)\left(e^{-\frac{c_0}{2}s}+\varepsilon\right)\left(\varepsilon+\|\tilde{b}-b_{\varepsilon}\|_{L^2}\right)\ds\\
		&\leq C\int_0^te^{-c_0(t-s)}\left(e^{-\frac{c_0}{2}s}+\varepsilon\right)\left(\varepsilon+\|\tilde{b}-b_{\varepsilon}\|_{L^2}\right)\ds\\
		&\leq C\varepsilon \left(e^{-\frac{c_0}{2}t}+\varepsilon\right)+C\varepsilon+e^{-\frac{c_0}{2}t}\int_0^te^{-\frac{c_0}{2}(t-s)}  \|\tilde{b}(s,\cdot)-b_{\varepsilon}(s,\cdot)\|_{L^2} \ds
	\end{split}
\end{equation}
\normalsize
We now study the difference between $b_\varepsilon$ and $\tilde{b}$ for large values of $b$. It is easy to check that the difference $b_\varepsilon-\tilde{b}$ satisfies 

$$\frac{1}{2}\frac{d}{dt}\int_\T|\tilde{b}-b_\varepsilon|^2+\int_\T \partial_x(ub_\varepsilon)(b_\varepsilon-\tilde{b})-\int_\T\partial_x(\tilde{u}\tilde{b})(b_\varepsilon-\tilde{b})=\varepsilon \int_\T\partial_x^2 b_\varepsilon(b_\varepsilon-\tilde{b}).$$

Making use of this estimate for the difference of the velocities, we find that 

$$\frac{1}{2}\|\tilde{b}(t,\cdot)-b_\varepsilon(t,\cdot)\|_{L^2}\leq C\left(\int_0^te^{-\frac{c_0}{2}t}\|b_\varepsilon(s,\cdot)-\tilde{b}(s,\cdot)\|_{L^2}\ds+\varepsilon t\right)$$

Finally, a direct application of the integral version of Gronwall's inequality leads to the result. 
\end{proof}

\begin{prop}\label{prop:initialvalue}
	The limit function $b$ has the property that $\lim_{t\rightarrow 0^+}b(t,\cdot)$ exists in $L^2$, and it equals $S(b_0)$, with $S$ given by Definition \ref{defi:nonlinearoperator}.
\end{prop}
\begin{proof}
	Take $\gamma>0$ and $\omega>0$. Now, using the bounds for $\partial_x b_\varepsilon$ and $\partial_x u_\varepsilon$ from Proposition \eqref{prop:munchhausen}, we can multiply \eqref{eq:magneticrelaxation} times $\partial_\tau b$ to infer

    $$\|\partial_\tau b(\tau,x)\|_{L^2}^2-\frac{d}{d\tau}\int_\T |\partial_x b_\varepsilon(\tau,x)|^2\dx\leq \frac{C}{\varepsilon^2}\left(e^{-\frac{c_0}{4\varepsilon}\tau}+\varepsilon\right)^2.$$

    Furthermore, $-\varepsilon\log(\varepsilon)\rightarrow 0$, for $\varepsilon$ small enough, $-\varepsilon\log(\varepsilon)<\omega$. Therefore, 

    \begin{equation} 
    \begin{split}
    \left\|b_\varepsilon\left(-\frac{4}{c_0}\varepsilon\log(\varepsilon),\cdot\right)-b_\varepsilon(\omega,\cdot)\right\|_{L^2}&\leq \int_{-\frac{4}{c_0}\varepsilon\log(\varepsilon)}^\omega \|\partial_\tau b_\varepsilon(s,\cdot)\|_{L^2}^2ds\\
    &\leq C(\varepsilon +\omega)
    \end{split}
    \end{equation}

    On the other hand, we showed in Lemma \ref{lemma:initialvalue} that 

    $$\left\|b_\varepsilon \left(-\frac{4}{c_0}\varepsilon\log(\varepsilon),\cdot\right)-\tilde{b}\left(-\frac{4}{c_0}\varepsilon\log(\varepsilon),\cdot\right)\right\|_{L^2}\leq -C\varepsilon \log(\varepsilon),$$
    and in Proposition \ref{prop:binfty} we proved that 

    $$\|\tilde{b}(\tau)(\cdot)-S(b_0)(\cdot)\|_{L^\infty}\leq Ce^{-\frac{c_0}{\varepsilon}\tau}.$$

    As a result, we obtain that 

    $$\|b(\omega,\cdot)-S(b_0)(\cdot)\|_{L^2}\leq \|b_\varepsilon(\omega,\cdot)-b(\omega,\cdot)\|_{L^2}+C|\varepsilon+\omega|^{1/2}+C\varepsilon\log(\varepsilon)+C\varepsilon.$$

    Using that $b_\varepsilon(\omega,\cdot)$ tends to $b(\omega,\cdot)$ uniformly as $\varepsilon$ tends to zero, the proposition follows.
\end{proof}
\begin{obs}
    It is inmediate that, since $b\in L^\infty(I;H^1(\T))\cap C(I;L^2(\T))$, $b$ actually belongs to $C(I;H^s(\T))$ for every $s\in [0,1)$ due to interpolation.
\end{obs}

So far, we have proved that the sequence of time-rescaled functions $b_\varepsilon(\tau,x)$ converges to a solution $b$ of \eqref{eq:limitsystem}-\eqref{eq:limitsystem2} with initial condition $S(b_0)$. This implies that the system \eqref{eq:limitsystem}-\eqref{eq:limitsystem2} with initial condition $S(b_0)$ admits a solution in $C((0,\delta);L^2)\cap L^2((0,\delta);H^2)$. 

Note that, so far, we have only obtained that there is a subsequence $\varepsilon_n\rightarrow 0$ such that $b_{\varepsilon_n}\longrightarrow b$ in $C_{loc}^0((0,\delta]\times C^1(\T)).$ However, we do not know if we can have two different subsequences $\varepsilon_n$ and $\varepsilon'_k$ so that $b_{\varepsilon_n}$ and $b_{\varepsilon'_k}$ converge to two different solutions $b$ and $b'$. As we will see in Theorem \ref{t:existencelimit}, the limit problem is well posed, so this possibility is ruled out. Therefore, there is a unique limit function $b\in C([0,\delta);L^2(\T))\cap C^0_{loc}((0,\delta]\times \T)$ with $\partial_x b\in L^\infty((0,\delta)\times \T)$.  

Note, further, that the solution $b$ might be extended for times beyond $\delta$. The question is then whether we can still use the function $b_\varepsilon (\tau,x)$ to approximate $b(\tau,x)$ for times up to the maximal existence time of \eqref{eq:limitsystem}-\eqref{eq:limitsystem2}. This is the main theorem of this section, where it is shown that the solution of the limit system \eqref{eq:limitsystem}-\eqref{eq:limitsystem2} is a good approximation for $b_\varepsilon (\tau,x)$ as long as a solution for \eqref{eq:limitsystem}-\eqref{eq:limitsystem2} exists. The precise statement is the following: 

\begin{prop}\label{prop:main}
    Consider $b_0\in H^2(\T)$ with $|b_0|^2\geq c_0>0$. Now, assume that there is a unique weak solution $b\in C([0,T^\star);L^2(\T))\cap L^4(I;H^4(\T))$ for the equations \eqref{eq:limitsystem}-\eqref{eq:limitsystem2} with initial condition $S(b_0)$, and let $T^\star$ be its maximal time of existence.  Let $\overline{T}<T^\star$ and assume further that there exist $d_0,\,K>0$ such that 

    \begin{equation}\label{eq:condicionesconvergencia}
        |b(\tau,x)|\geq d_0\quad\text{and}\quad  \|\partial_x b\|_{C^0([0,T]\times \T)}\leq K.
    \end{equation}

    Then, the sequence $(b_\varepsilon)_{\varepsilon>0}$ converges, as $\varepsilon\rightarrow 0$, in $C^{0}_{loc}((0,\overline{T}]\times C^1(\T))$ to $b$. 
\end{prop}
\begin{obs}
    Here we have proven that accumulation points $b$ of the sequence $(b_\varepsilon)_{\varepsilon > 0}$ solve the limit system \eqref{eq:limitsystem}-\eqref{eq:limitsystem2}. However, we will see in Theorem \ref{t:existencelimit} below (whose proof is independent) that the initial data problem can only have one (weak) solution in the class $L^4_t(H^1_x) \cap C^0_t(H^{-1}_x)$ if the initial datum is $\dot{H}^{1/2}$. Notice that any accumulation point $b$ belongs to the space $L^\infty_t(H^1_x)\cap C^0_t(L^2_x)$, with initial value $S(b_0) \in W^{1, \infty} \subset \dot{H}^{1/2}$. The initial data conditions can be written in the sense of distributions, as the solutions are continuous in the time variables. Consequently, there can only one such accumulation point, and we deduce that the sequence $(b_\varepsilon)_{\varepsilon > 0}$ actually converges without extraction.
\end{obs}
\begin{proof}[Proof of Proposition \ref{prop:main}]
    We will make use of a similar argument as the one we used in Proposition \ref{prop:munchhausen}. Notice that, since the problem \eqref{eq:limitsystem}-\eqref{eq:limitsystem2} admits a unique solution, $b_{\varepsilon}\rightarrow b$ as $\varepsilon\rightarrow 0$, and no resorting to a subsequence is needed.
    
    Let $0<a<\delta\leq \overline{T}$ fixed, but arbitrary, with $\delta$ as the one in Proposition \ref{prop:munchhausen}. Assume that there is some $\varepsilon_0>0$ and a constant $A>0$ such that 

    \begin{equation}\label{eq:condicionsforconvergence}
    |b_\varepsilon(\tau,x)|\geq d_0/4,\quad \|\partial_x b_\varepsilon(\tau,\cdot)\|_{L^\infty}\leq K+1 \quad \|\partial_x u_\varepsilon(\tau,\cdot)\|_{L^\infty}\leq A\left(e^{-\frac{d_0}{4\varepsilon}\tau}+\varepsilon\right)
    \end{equation}
    for $\tau\in [0,\overline{T}]$ and for every $\varepsilon<\varepsilon_0$. Then, we can repeat the arguments in Proposition \ref{prop:limiteq} to show that $b$ converges in $C^{0}([a,\overline{T}];C^1(\T))$ to the solution $b$ of \eqref{eq:limitsystem}-\eqref{eq:limitsystem2}. As a result, we only need to show that the estimates in \eqref{eq:condicionsforconvergence} hold. 

    We begin noticing that, due to Proposition \ref{prop:munchhausen}, if $B_0$ is an $H^2(\T;\R^2)$ function satisfying  $|B_0|\geq d_0/2$ for every $x\in \T$, then there exists some $\delta^\star$ and constants $\overline{K}$, $\overline{A}$ such that for $\tau\in [0,\delta^\star]$, the solution $B$ of \eqref{eq:magneticrelaxation} with initial data $B_0$ satisfies 

    $$|B(\tau,x)|\geq d_0/4,\quad \|\Psi(\tau,\cdot)\|_{L^\infty}\leq \overline{A}\left(\|\Psi_0\|_{L^\infty}e^{-\frac{d_0}{4\varepsilon}\tau}+\varepsilon\right)\text{ and } \|\partial_x B(\tau,\cdot)\|_{L^\infty}\leq \overline{K},$$
    with $\Psi:=\left(|B|^2-\|B\|_{L^2}^2\right)/2$, and $\overline{K}$, $\overline{A}$ and $\delta^\star$ depending only on $\|B_0\|_{C^1(\T)}$ and $d_0$. Furthermore, $\delta^\star$ can be taken as a function $\delta^\star=\delta^\star\left(\frac{1}{d_0},\|B_0\|_{L^\infty},\|\partial_xB_0\|_{L^\infty},\|\Psi_0\|_{L^\infty}\right)$, non-increasing on each of its arguments.

    Now we define, for every $\varepsilon>0$, the set 
    \small
    \begin{equation}\label{eq:Bepsilon}
    \mathcal{B}_\varepsilon:=\left\{\tau\in [0,T^\star]\,: |b_\varepsilon(\tau,x)|\geq d_0/4 ,\,\|\partial_x b_\varepsilon(\tau,\cdot)\|_{L^\infty}\leq K+1 \text{ and }\|\partial_x u_\varepsilon(\tau,\cdot)\|_{L^\infty}\leq (A+1)\left(e^{-\frac{d_0}{4\varepsilon}\tau}+\varepsilon\right)\right\},
    \end{equation}
    \normalsize
    and take $\Gamma$ to be

    $$\Gamma:=\liminf_{\varepsilon\rightarrow 0}\left(\sup\left\{\sigma \in [0,T^\star]\,:\, [0,\sigma]\subset \mathcal{B}_\varepsilon\right\}\right).$$
    It is then enough to show that $\Gamma> \overline{T}$. To prove it, we argue by contradiction. It is clear that, due to Proposition \ref{prop:munchhausen}, $\Gamma>0$.

    Now, take $\delta^\star$ to be the one corresponding to an initial data satisfying the bounds in \eqref{eq:Bepsilon}. By definition of $\liminf$, there is some $\varepsilon_0$ such that $[0,\Gamma-\delta^\star/2]\subset \mathcal{B}_\varepsilon$ for every $\varepsilon<\varepsilon_0$. Due to Proposition \ref{prop:limiteq}, $b_\varepsilon(\tau,x)$ converges uniformly to $b(\tau,x)$ in $[a,\Gamma-\delta^\star/2]$. Since $\Gamma-\delta^\star/2< \overline{T}$, there is $\varepsilon_0'$ small enough so that 
    
    $$|b_\varepsilon(\Gamma-\delta^\star/2,x)|\geq d_0/2 \quad \text{and}\quad \text{and}\quad\|\partial_x b_\varepsilon(\Gamma-\delta^\star/2,\cdot)\|_{L^\infty}\leq K+1.$$

    Now, we can apply Proposition \ref{prop:munchhausen} taking $b_\varepsilon(\Gamma-\delta^\star/2,x)$ as an initial data. Thus, if we denote $\tau_0:=\Gamma-\delta^\star/2$,  there is some $\varepsilon_0'$ so that for every $\varepsilon<\varepsilon_0'$ and $\tau\in [\tau_0,\tau_0+\delta^\star],$
    
    $$|b_\varepsilon(\tau,x)|\geq d_0/4,\quad \|\psi_\varepsilon(\tau,\cdot)\|_{L^\infty}\leq \overline{A}\left(\|\psi_\varepsilon(\tau_0,\cdot)\|_{L^\infty}e^{-\frac{d_0}{4\varepsilon}(\tau-\tau_0)}+\varepsilon\right)\quad\text{ and } \quad \|\partial_x b(\tau,\cdot)\|_{L^\infty}\leq \overline{K}.$$
    
    Since $\|\psi(\tau_0,\cdot)\|\leq C\varepsilon$, $b_\varepsilon$ satisfies the estimates in \eqref{eq:condicionsforconvergence}, and the function $b_\varepsilon(\tau,x)$ converges in $C^{0}_{loc}((0,\Gamma+\delta^\star/2)\times \T)$ to $b$. Thus, for $\varepsilon$ small enough, $b_\varepsilon$ satisfies \eqref{eq:condicionsforconvergence}. This contradicts the definition of $\Gamma$. 
    
\end{proof}

\section{Local Well-Posedness of the Limit System and Finite Time Blow-Up}\label{sec:limiteqandblowup}
In this paragraph, we study the well-posedness theory for the limit system \eqref{eq:limitsystem}-\eqref{eq:limitsystem2}, which we recast in radius-angle variable, \textit{i.e.} if $b = R e^{i \theta}$ then we get the equations
\begin{equation}\label{eq:AngleVariablePDE}
    \begin{cases}
        \partial_t \theta + u \partial_x \theta = \partial_x^2 \theta \\
        \partial_x u = - (\partial_x \theta)^2 + \int (\partial_x \theta)^2,
    \end{cases}
    \qquad \text{and} \qquad R'(t) = - R(t) \int (\partial_x \theta)^2.
\end{equation}
The existence of solutions on a short time interval is a consequence of the discussion above, since we have proven that the magnetic field $b_\varepsilon (\tau, x)$ converges to a function of the form $b(\tau, x) = R(\tau)e^{i \theta(\tau, x)}$ which solves this system. However, many questions remain, such as the uniqueness of solutions, or whether there are global solutions for some initial data and whether some solutions blow up in finite time.


\medskip

As we have noted above, the angular variable has one important feature the angle $\theta(\tau, x)$ is not necessarily a periodic function in $x \in \R$, but is equal to a periodic function up to the addition of a linear function $2 \pi Nx$, where $\theta(1) - \theta(0) = 2\pi N \in 2 \pi \Z$ related to the number $N$ of turns the magnetic field $b = R e^{i \theta}$ makes around the complex origin. Let us show that this number $N$ is independent of time. We have
\begin{equation*}
    N'(t) = \frac{1}{2 \pi} \frac{\rm d}{\dt} \big( \theta(t, 1) - \theta(t, 0) \big) = \frac{1}{2 \pi} \frac{\rm d}{\dt} \int_0^1 \partial_x \theta \dx = \frac{1}{2 \pi} \int_0^1 \partial_t \partial_x \theta \dx.
\end{equation*}
By using the equation that $\theta$ solves and expressing $\partial_t \theta$ as a function of $u$ and $\theta$, we obtain
\begin{equation*}
    N'(t) = \frac{1}{2 \pi} \int_0^1 \Big( \partial_x^3 \theta - \partial_x (u \partial_x \theta) \Big) \dx
\end{equation*}
Now, the integrand above is the derivative of $\partial_x^2 \theta - u \partial_x \theta$, which is a periodic function. The integral is therefore zero, and we see that $N$ is constant throughout the evolution. In particular, we can express the angular variable as
\begin{equation*}
    \theta(t,x) = \zeta(t,x) + 2 \pi Nx,
\end{equation*}
where $\zeta : \R \times \T \longrightarrow \R$ is a periodic function of the space variable. Figure \ref{fig:MagneticTopology} illustrates this phenomenon in terms of the geometry of the magnetic field.

\begin{figure}[h!]
    \centering
    \includegraphics[width=0.5\linewidth]{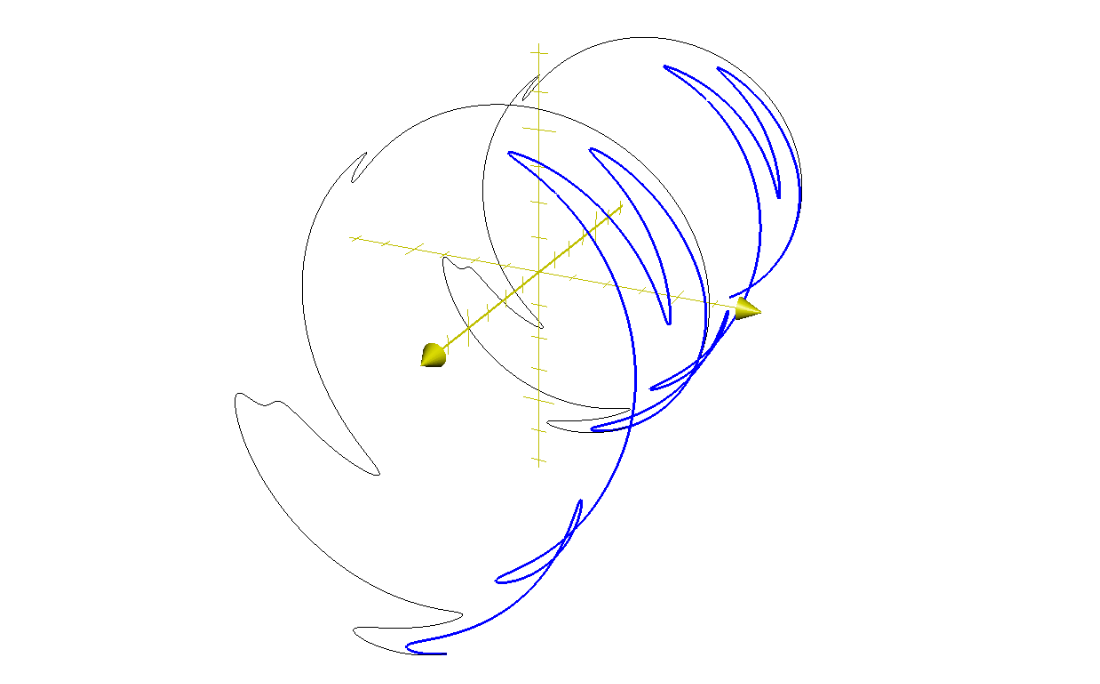}
    \caption{An illustration of two different values of the number of turns $N \in \Z$. As in Figure \ref{fig:Relaxation}, the magnetic field is represented as a parametrized curve $x \mapsto (0, b_1(x), b_2(x))$. The thick blue line represents a situation where $N = 0$ and the magnetic field is ``topologically trivial'' with respect to the axis $(1, 0, 0)$, and the thin black line represents the converse situation, where the magnetic field circles around the axis $N \neq 0$ (here $N = 1$).}
    \label{fig:MagneticTopology}
\end{figure}

Now, we make an important remark on how to determine the velocity field $u$ from the angular variables $\theta$. In the rest of the paper, we required the cancellation condition $u(t,0) = 0$ in order to uniquely write the velocity field as a function of the magnetic field. However, as we wish to work in a low regularity setting and with Fourier analysis, we see that this condition is much less convenient that cancellation of the mean value $\int u$. Therefore, we will prove existence and uniqueness of solutions using the condition
\begin{equation*}
    \forall t, \qquad \int_\T u(t,x) \dx = 0,
\end{equation*}
and then use a change of reference frames \eqref{ieq:Galileo} to go back to a solution with a velocity that cancels at $x = 0$. See Proposition \ref{p:Galileo}. This is the purpose of Proposition \ref{p:Galileo} below. We can then recast system \eqref{eq:AngleVariablePDE} as a PDE system involving $\zeta$ only,
\begin{equation}\label{eq:AnglePDERenormalized1}
    \partial_t \zeta - 8\pi^2 N^2 \left( \zeta - \int \zeta \right) - \partial_x^2 \zeta = 2 \pi N \left( \zeta - \int \zeta \right) \partial_x \zeta - 2 \pi N v - v \partial_x \zeta
\end{equation}
where the function $v$ is defined by the condition at $x = 0$,  $v(t,0) = 0$, and the identity
\begin{equation}\label{eq:AnglePDERenormalized2}
    \partial_x v = - (\partial_x \zeta)^2 + \int (\partial_x v)^2
\end{equation}

\begin{obs}\label{r:AngleAdditionInvariance}
    We observe that the equations \eqref{eq:AnglePDERenormalized1} are invariant under the addition of a real constant to the angle, that is,  for any fixed $\Theta \in \R$ and any solution $\theta$, the quantity $\vartheta(t, x) := \theta(t,x) + \Theta$ is also a solution. This is of course due to the invariance under rotations of the magnetic field $b$ in problem \eqref{eq:magneticrelaxation}. In practice, it means that any condition on the initial data can be achieved up the the addition of a constant (see \textit{e.g.} Remark \ref{r:HomogeneousSobolev} below).
\end{obs}

\subsection{Local Well-Posedness of the Limit System}

Before stating our result concerning the existence and uniqueness of solutions to \eqref{eq:AnglePDERenormalized1} and \eqref{eq:AnglePDERenormalized2}, let us comment on what we should expect based on the general form of the equation. First of all, keeping on the left-hand side only the linear terms in $\zeta$, we can rewrite equation \eqref{eq:AnglePDERenormalized1} as
\begin{equation*}
    \partial_t \zeta - 8 \pi^2 N^2 \left( \zeta - \int \zeta \dx \right) - \partial_x^2 \zeta = F,
\end{equation*}
where the quantity $F$ contains all non-linear terms, which are of order two and three in $\zeta$ and its derivatives. Two things may occur. Firstly, we could have $N = 0$, and the left-hand side of the above is a heat equation with a non-linear perturbation, and we can hope to prove local well-posedness of the equation, and even global-well-posedness for small enough initial data by fixed-point methods. Secondly, we could have $|N| \geq 1$ and the linear part of the equation has some unstable modes. In that case, the spectrum of the operator above $\{ 0 \} \cup \{ 4 \pi^2 (k^2 - 2N^2), \; k \neq 0 \}$ has a negative part. In particular, we cannot expect small initial data to remain small, and this compromises our effort to construct global solutions. Therefore, if $N\neq 0$, we will only derive local in time results.
\begin{teor}\label{t:existencelimit}
    Consider $N \in \Z$ an initial datum $\zeta_0 \in H^{1/2}(\T)$. Then the following assertions hold: 
    \begin{enumerate}
        \item There exists a $T > 0$ depending on that initial datum such that problem \eqref{eq:AnglePDERenormalized1} - \eqref{eq:AnglePDERenormalized2} has a unique solution in the space $L^4([0,T) ; H^1(\T))$. Such a solution also lies in $C^{0, \frac{1}{4}}([0, T) ; H^{-1}(\T))$.

        \item Any solution $\zeta \in L^4([0,T) ; H^1(\T))$ is smooth with respect to the time and space variables $\zeta \in C^\infty ((0,T) \times \T)$.

        \item If $N = 0$, there exists a absolute (numerical) constant $\eta_0 > 0$ such that if $\| \zeta_0 \|_{H^{1/2}} \leq \eta_0$ then the solution above is global.
    \end{enumerate}
\end{teor}

\begin{obs}\label{r:HomogeneousSobolev}
    From Remark \ref{r:AngleAdditionInvariance}, we see that the smallness condition $\| \zeta_0 \|_{H^{1/2}}$ need only be achieved up the the addition of a constant to the initial data. In other words, we only need $\inf_{\Theta \in \R} \| \zeta_0 - \Theta \|_{H^{1/2}} \leq \eta_0$ to hold for solutions to be global, \textit{i.e.} we only need the homogeneous Sobolev condition $\| \zeta_0 \|_{\dot{H}^{1/2}} \leq \eta_0$.
\end{obs}

\begin{proof}
        The proof relies on the parabolic structure of the system in order to implement a fixed point argument, and is divided in three steps. Firstly the obtention of estimates related to the linear part of the equation, secondly the fixed point argument, and thirdly showing the uniqueness of the solutions obtained.

    \medskip

    \textbf{STEP 1: Linear Estimates.} We begin by performing linear estimates so as to bound the solution by norms of $F_2$ and $F_3$. This is the purpose of this first lemma.

    \begin{lemma}\label{l:LinearEstimates}
        Consider a regular function $F$ and $w$ the unique regular solution of the following parabolic problem:
        \begin{equation}\label{eq:LinearParabolicN}
            \begin{cases}
                \partial_t w - 8 \pi^2 N^2 (w - \int w) - \partial_x^2 w = F \qquad \text{for } (t, x) \in \R_+ \times \T \\
                w(0,x) = w_0 \qquad \text{for } x \in \T.
            \end{cases}
        \end{equation}
        Then, it holds that for any time $T > 0$ and $1 \leq q \leq 2 \leq p \leq \infty$ and $s \in \R$, we have
        \begin{equation*}
            \| w \|_{L^p_T(H^{s+2/p})} \leq e^{8 \pi^2 N^2 T} \left( \| w_0 \|_{H^s} + \| F \|_{L^q_T(H^{s-2+2/q})} \right).
        \end{equation*}
    \end{lemma}

    \begin{proof}
    Since the parabolic equation is linear, it is enough to consider separately the cases where $w_0 = 0$ and $F = 0$, and then to sum both estimates. We start by assuming that $F = 0$. In that case, the Fourier coefficients of the solution $w$ of \eqref{eq:LinearParabolicN} read
        \begin{equation*}
            \what{w}(t, k) =
            \begin{cases}
                e^{(8\pi^2 N^2 - 4 \pi^2 k^2)t} \what{w_0}(k) \qquad \text{if } k \neq 0,\\
                \what{w_0}(0) = 0 \qquad \text{if } k = 0.
            \end{cases}
        \end{equation*}
        We then take the $L^p([0, T])$ norm of the coefficients (denoted $\| \, \cdot \, \|_{L^p_T}$, considered with respect to the time variable only), which can be bounded by
        \begin{equation*}
            \| \what{w}(k) \|_{L^p_T} \leq 
            \begin{cases}
                \frac{e^{8\pi^2 N^2 T}}{(4 \pi^2 k^2)^{1/p}} |\what{w_0}(k)| \qquad \text{if } k \neq 0 \\
                |\what{w_0}(k)| \qquad \text{if } k = 0.
            \end{cases}
        \end{equation*}
        We then square the above and sum with respect to $k \in \Z$ and multiply by $(1 + k^2)^{s+2/p}$ to obtain
        \begin{equation*}
            \sum_{k \in \Z} (1 + k^2)^{s+2/p} \| \what{w}(k) \|_{L^p_T}^2 \leq \left( e^{8 \pi^2 N^2 T} \right)^2 \sum_{k \in \Z} (1 + k^2)^s |\what{w_0}(k)|^2 = \left( e^{8 \pi^2 N^2 T} \| w_0 \|_{H^s} \right)^2.
        \end{equation*}
        Now, we remark that thanks to the assumption $p \geq 2$, the Minkowski inequality (Proposition 1.3 in \cite{BCD}) shows that the left-hand side is bounded by below by $\| w \|_{L^p_T(H^{s+2/p})}^2$. This gives the inequality
        \begin{equation}\label{eq:LemmaNInequality1}
            \| w \|_{L^p_T(H^{s+2/p})} \leq e^{8 \pi^2 N^2 T} \| w_0 \|_{H^s},
        \end{equation}
        which corresponds to the statement of the lemma when $F = 0$.

        \medskip

        We can now focus on the case where $w_0 = 0$. Duhamel's formula insures that the Fourier coefficients of the solution $w$ of \eqref{eq:LinearParabolicN} are given by
        \begin{equation*}
            \begin{split}
                \what{w}(t, k) & = \int_0^t e^{(8 \pi^2 N^2 - 4 \pi k^2)(t-s)} \what{F}(s, k) {\rm d}s \\
                & = \int_{- \infty}^{+ \infty} \mathbf{1}_{t-s \geq 0} e^{(8 \pi^2 N^2 - 4 \pi k^2)(t-s)} \mathbf{1}_{s \geq 0} \what{F}(s, k) {\rm d}s
            \end{split}
        \end{equation*}
        when $k \neq 0$ and
        \begin{equation*}
            \what{w}(t, 0) = \int_0^t \what{F}(s, 0) {\rm d} s
        \end{equation*}
        when $k = 0$. Observe that the integral for $k = 0$ is in fact a convolution product with respect to the time variable, so that we can use the Hausdorff-Young convolution inequality to bound its $L^p([0, T])$ norm. Consider $r \in [1, \infty]$ such that $1 + \frac{1}{p} = \frac{1}{q} + \frac{1}{r}$, which will always exist because we have assumed $1 \leq q \leq 2 \leq p \leq \infty$. We have
        \begin{equation*}
            \| \what{w}(k) \|_{L^p_T} \leq \frac{e^{8 \pi^2 N^2 T}}{\big(1 + k^2\big)^{1/r}} \| \what{F} (k) \|_{L^q_T}.
        \end{equation*}
        We perform \textit{mutatis mutandi} the same operations as above, \textit{i.e.} square the inequality above, multiply by $(1 + k^2)^{s + 2/p} = (1 + k^2)^{s + \frac{2}{r} - 2 + \frac{2}{q}}$, and sum over all $k \in \Z$ to obtain
        \begin{equation*}
            \sum_{k \in \Z} (1 + k^2)^{s + 2/p} \| \what{w}(k) \|_{L^p_T}^2 \leq e^{8 \pi^2 N^2 T} \sum_{k \in \Z} (1 + k^2)^{s -2 + \frac{2}{q}} \| \what{F}(k) \|_{L^q_T}^2.
        \end{equation*}
        Finally, we can apply Minkowski's inequality with $1 \leq q \leq 2 \leq p \leq \infty$ and get the desired inequality
        \begin{equation}\label{eq:LemmaNInequality2}
            \| w \|_{L^p_T(H^{s+2/p})} \leq e^{8 \pi^2 N^2 T} \| F \|_{L^q_T(H^{s - 2 + 2/q})}.
        \end{equation}
        The statement of the lemma can be recovered by using the linearity of the equation and summing \eqref{eq:LemmaNInequality1} and \eqref{eq:LemmaNInequality2}.
        
    \end{proof}

    \textbf{STEP 2: Fixed-Point Argument.} Now that we have linear estimates at our disposal, we are ready to implement a fixed-point argument. Before starting, we make a couple of important remarks. First of all, notice that the left-hand side in equation \eqref{eq:AnglePDERenormalized1} is linear with respect to the unknown, while the right-hand side can be decomposed into two terms
    \begin{equation}\label{eq:F2bilinear}
        F_2 := - 2 \pi N v + 4 \pi N \left( \zeta - \int \zeta \right) \partial_x \zeta,
    \end{equation}
    and
    \begin{equation}\label{eq:F3trilinear}
        F_3 := - v \partial_x \zeta,
    \end{equation}
    which are respectively of order two and three in the unknown $\zeta$. Secondly, the function $v$ is defined by the cancellation condition $v(t,0) = 0$ and identity \eqref{eq:AnglePDERenormalized2}. However, for equation \eqref{eq:AnglePDERenormalized1} to make sense, the solution should have regularity $\zeta \in L^p([0, T) ; H^1(\T))$ for some $p \geq 2$. This means that a natural space to look for solutions is of the form $L^p([0, T) ; H^1(\T))$ for some $p \geq 2$ to be fixed later on.

    \medskip

    In order to implement the fixed-point argument, we prove estimates of $F_2$ and $F_3$, considered as functions of $\zeta$. This is the object of the following lemma.

     \begin{lemma}\label{l:multilinearEstimates}
         Consider $p \geq 2$, a function $\zeta \in L^p([0, T) ; H^1(\T))$ and the quantities $F_2$ and $F_3$ given by \eqref{eq:F2bilinear} and \eqref{eq:F3trilinear}. Then the following inequalities hold:
         \begin{equation*}
            \begin{split}
                &\text{For } p \geq 3, \qquad \| F_3 \|_{L^{p/3}_T(L^2)} \leq \| \zeta \|_{L^p_T(H^1)}^3 \\
                &\text{For } p \geq 2, \qquad\| F_2 \|_{L^{p/2}_T(L^2)} \leq C|N| \| \zeta \|_{L^p_T(H^1)}^2,
            \end{split}
         \end{equation*}
         for some (numerical) constant $C > 0$.
     \end{lemma}

     \begin{proof}[Proof (of the lemma)]
         We begin by the estimate on $F_3$, which will be determinant, as it is the higher order non-linearity in the equation. By using the definition  of $v$ (\textit{c.f.}  \eqref{eq:AnglePDERenormalized2}), we see that
         \begin{equation*}
             \| \partial_x v \|_{L^{p/2}_T(L^1)} \leq \| \zeta \|_{L^p_T(H^1)}^2.
         \end{equation*}
         In other words, the velocity field is $L^{p/2}([0, T) ; W^{1, 1}(\T))$. We wish to make use of the embedding $W^{1, 1}(\T) \subset L^\infty(\T)$. Recall that $v$ satisfies the cancellation property $v(t,0) = 0$. The formula
         \begin{equation*}
             v(x) = \int_{0}^x \partial_y v(y) \dy
         \end{equation*}
         then implies that the inequality $\| v \|_{L^\infty} \leq \| \partial_x v \|_{L^1}$ holds. By using this, we deduce that
         \begin{equation*}
             \| F_3 \|_{L^{p/3}_T(L^2)} = \| v \partial_x \zeta \|_{L^{p/3}_T(L^2)} \leq \| v \|_{L^{p/2}_T (L^\infty)} \| \zeta \|_{L^p_T(L^2)} \leq \| \zeta \|_{L^p_T(H^1)}^3,
         \end{equation*}
         which is what we wanted to prove. We can now look at the second estimate. On the one hand, we have already proved that $\| v \|_{L^{p/2}_T(L^\infty)} \leq \| \zeta \|_{L^p_T(H^1)}^2$, so we only need to look at the second summand in $F_2$. Note that both the function $g = \zeta - \int \zeta$  has zero average and an integrable derivative, so that we can fix a $x_\star \in \T$ such that $g(x_\star) = 0$. The formula 
         \begin{equation*}
             g(x) = g(x_\star) + \int_{x_\star}^x \partial_x \zeta (y) \dy
         \end{equation*}
        means that the inequality $\| g \|_{L^\infty} \leq \| \partial_x \zeta \|_{L^1}$ holds, and hence
         \begin{equation*}
             \begin{split}
                 \| F_2 \|_{L^{p/2}(L^2)} & \lesssim |N| \| v \|_{L^{p/2}_T(L^\infty)} + |N| \| \partial_x \zeta \|_{L^p_T(L^2)} \left\| \zeta - \int \zeta \right\|_{L^p_T(L^\infty)} \\
                 & \lesssim |N| \| \zeta \|_{L^{p/2}(H^1)}^2.
             \end{split}
         \end{equation*}
         This ends the proof of the lemma.
     \end{proof}

    We are now ready to implement the fixed point argument. Consider a $s >0$ to be fixed later, an initial datum $\zeta_0 \in H^s(\T)$, and define the map $\Psi(\zeta) = w$ which associates a function $\zeta$ to the unique solution $w$ of 
    \begin{equation*}
        \begin{cases}
            \partial_t w - 8 \pi^2 N^2 \left( w - \int w \right) - \partial_x^2 w = F_2 + F_3 \\
            w(0) = \zeta_0,
        \end{cases}
    \end{equation*}
    where $F_2$ and $F_3$ are given as functions of $\zeta$ by \eqref{eq:F2bilinear} and \eqref{eq:F3trilinear}. For convenience of notation, we note $S_N(t)$ the semi-group associated with the linear parabolic equation \eqref{eq:LinearParabolicN}. Then, making use of the estimates in Lemma \ref{l:multilinearEstimates}, we may use Lemma \ref{l:LinearEstimates} to obtain the inequality
    \begin{equation*}
        \| \Psi(\zeta) \|_{L^p_T(H^1)} \leq e^{8 \pi^2 N^2 T} \left( \| S_N(t) w_0 \|_{L^p_T(H^1)} + \| F_2 \|_{L^{p/3}_T(L^2)} + \| F_3 \|_{L^{p/3}_T(L^2)} \right)
    \end{equation*}
    provided that the condition
    \begin{equation*}
        s + \frac{2}{p} = 1 \qquad \text{and} \qquad s - 2 + \frac{6}{p} = 0
    \end{equation*}
    hold for some $2 \leq p \leq 6$ (thus insuring that $p/3 \leq 2 \leq p$). This forces the choice $p = 4$ and $s = 1/2$. We then invoke Lemma \ref{l:multilinearEstimates} to write
    \begin{equation}\label{eq:FixedPointInequality}
        \begin{split}
            \| \Psi(\zeta) \|_{L^4_T(H^1)} & \leq e^{8 \pi^2 N^2 T} \Big( \| S_N (t) \zeta_0 \|_{L^4_T(H^1)} + \| F_2 \|_{L^{4/3}_T(L^2)} + \| F_3 \|_{L^{4/3}_T(L^2)} \Big) \\
            & \leq e^{8 \pi^2 N^2 T} \Big( \| S_N (t) \zeta_0 \|_{L^4_T(H^1)} + C|N| T^{1/4} \| \zeta \|_{L^4_T(H^1)}^2 + \| \zeta \|_{L^4_T(H^1)}^3 \Big).
        \end{split}
    \end{equation}
    In order to apply the Picard-Banach fixed-point theorem, we have to use this inequality to show that the map $\Psi : L^4([0, T) ; H^1(\T)) \longrightarrow L^4([0, T) ; H^1(\T))$ has a fixed point (at least for some $T > 0$ small enough). Consider a radius $R > 0$. We want to show that the map $\Psi$ is a contraction of the closed ball $B_R := B(0, R) \subset L^4([0, T) ; H^1(\T))$ centred at zero and of radius $R$ onto itself. 

    \medskip

    We start by checking that the ball is 
    invariant under the action of the fixed-point map $\Psi : B_R \longrightarrow B_R$. Consider $\zeta \in B_R$. We have by \eqref{eq:FixedPointInequality},
    \begin{equation*}
        \| \Psi(\zeta) \|_{L^4_T(H^1)} \leq e^{8 \pi^2 N^2 T} \Big( a(T) + C|N|T^{1/4} R^2 + R^3 \Big),
    \end{equation*}
    where we have set $a(T) = \| S(t) \zeta_0 \|_{L^4_T(H^1)}$. The right hand side in the inequality above will be smaller than $R$ under the sufficient condition that
    \begin{equation}\label{eq:ContractionRadius}
        e^{8 \pi^2 N^2 T} a(T) \leq \frac{1}{2}R \qquad \text{and} \qquad C|N|T^{1/4} R + \frac{3}{2} R^2 < \frac{1}{2}e^{-8\pi^2 N^2 T},
    \end{equation}
    Note that the factor $3/2$ appearing in the second inequality is not necessary for proving that $B_R$ is invariant under the action of $\Psi$. However, we include it, as we will see it will also insure that $\Psi$ is a contraction. In any case, there exists a $R > 0$ such that the two inequalities hold if the following stronger condition is satisfied:
    \begin{equation}\label{eq:RadiusRelation}
        a(T) < \frac{1}{6} e^{- 8 \pi^2 N^2 T} \left( -C|N|T^{1/4} + \sqrt{C^2 N^2 T^{1/2} + 3 e^{8 \pi^2 N^2 T}} \right)
    \end{equation}
    holds. Observe that this condition is fulfilled in the two following cases
    \begin{enumerate}
        \item If $T > 0$ is small enough. Indeed, by Lemma \ref{l:LinearEstimates}, the function $S_N(t) \zeta_0$ is belongs to the space $L^4([0, T) ; H^1(\T))$ if $\zeta_0 \in H^{1/2}(\T)$, so that the quantity $\| S_N(t) \zeta_0 \|_{L^4_T(H^1)}$ tends to zero as $T \rightarrow 0^+$. On the other hand, the upper bound of \eqref{eq:RadiusRelation} tends to $\frac{1}{6}\sqrt{3}$ as $T \rightarrow 0^+$.

        \item If $N = 0$ and $\| \zeta_0 \|_{H^{1/2}} \leq \frac{1}{6}\sqrt{3} := \eta_0$. This is because Lemma \ref{l:LinearEstimates} provides in that case $\| S_N(t) \zeta_0 \|_{L^4_T(H^1)} \leq \| \zeta_0 \|_{H^{1/2}}$ while the upper bound of \eqref{eq:RadiusRelation} reduces to $\frac{1}{6}\sqrt{3}$.
    \end{enumerate}
    Under these conditions, we deduce that $\Psi : B_R \longrightarrow B_R$ is a well-defined map. We now have to examine under which conditions this map is a contraction. For this, we use the following lemma, whose proof is similar to that of Lemma \ref{l:multilinearEstimates}.

    \begin{lemma}\label{l:MultilinearStability}
        As above, $B_R$ is the ball of radius $R$ centred at zero in $L^p([0, T) ; H^1(\T))$ for some $p \geq 2$.  Consider $\zeta_1, \zeta_2 \in B_R$, and let $F_2(\zeta_1)$ and $F_3(\zeta_2)$ be the quadratic and cubic quantities obtained by replacing $\zeta$ by $\zeta_1$ and $\zeta_2$ in equations \eqref{eq:F2bilinear} and \eqref{eq:F3trilinear}. Then 
        \begin{equation*}
            \| F_2(\zeta_2) - F_2(\zeta_1) \|_{L^{p/2}_T(L^2)} \leq 2C|N|R \| \zeta_2 - \zeta_1 \|_{L^p_T(H^1)}
        \end{equation*}
        and
        \begin{equation*}
            \| F_3(\zeta_2) - F_3(\zeta_1) \|_{L^{p/3}_T(L^2)} \leq 3 R^2 \| \zeta_2 - \zeta_1 \|_{L^p_T(H^1)},
        \end{equation*}
        where $C > 0$ is the same constant as in Lemma \ref{l:multilinearEstimates}
    \end{lemma}

    Let us use this lemma to find estimates on $w := \Psi(\zeta_2) - \Psi(\zeta_1)$. By definition of $\Psi$, the function solves the PDE problem (using the notation )
    \begin{equation*}
        \begin{cases}
            \partial_t w - 8 \pi^2 N^2 \left( w - \int w \right) - \partial_x^2 w = F_2(\zeta_2) - F_2(\zeta_1) + F_3(\zeta_2) - F_3(\zeta_1)\\
            w(0) = 0,
        \end{cases}
    \end{equation*}
    so that Lemma \ref{l:LinearEstimates} provides the bound
    \begin{equation*}
        \| w \|_{L^4_T(H^1)} \leq e^{8 \pi^2 N^2 T} \Big( \| F_2(\zeta_1) - F_2(\zeta_1) \|_{L^{4/3}_T(L^2)} + \| F_3(\zeta_1) - F_3(\zeta_1) \|_{L^{4/3}_T(L^2)} \Big).
    \end{equation*}
    We may then use Lemma \ref{l:MultilinearStability} with $p = 4$ to bound the differences in the right hand side of the formula above, so that
    \begin{equation}\label{eq:StabilityBound}
        \| w \|_{L^4_T(H^1)} \leq e^{8 \pi^2 N^2 T} \big( 2C|N| T^{1/4} R + 3R^2 \big) \| \zeta_2 - \zeta_1 \|_{L^4_T(H^1)}.
    \end{equation}
    We have shown that $\Psi : B_R \longrightarrow B_R$ is a contraction is the constant in the right-hand side above is smaller than $1$. This is surely the case if the second inequality in \eqref{eq:ContractionRadius} is true, and therefore will hold for some $R > 0$ provided that \eqref{eq:RadiusRelation} is true.

    \medskip

    Let us summarize. If either $T > 0$ is small enough, or if $N = 0$ and $\| \zeta_0 \|_{H^{1/2}} \leq \eta_0$, then there exists a radius $R > 0$ such that the map $\Psi : B_R \longrightarrow B_R$ is a well-defined contraction of the closed ball $B_R \subset L^4([0, T) ; H^1(\T))$. The fixed-point theorem therefore provides the existence of a fixed-point $\Psi(\zeta) = \zeta$ (unique in $B_R$) which is a solution of problem \eqref{eq:AnglePDERenormalized1} - \eqref{eq:AnglePDERenormalized2} for the initial datum $\zeta_0$.
    
    \medskip

    \textbf{STEP 3: Uniqueness of Solutions.} The fixed-point argument shows that the solution is unique among all solutions lying in the ball $B_R \subset L^4([0, T) ; H^1(\T))$, but not among all possible solutions sharing the same initial data. To prove this, we will resort to energy bounds. Let $\zeta_{1, 0}, \zeta_{2, 0} \in H^{1/2}(\T)$ and consider two corresponding solutions $\zeta_1, \zeta_2 \in L^4([0, T) ; H^1(\T))$. Then the difference $w = \zeta_2 - \zeta_1$ solves the equation
    \begin{equation*}
        \partial_t w - 8 \pi^2 N^2 \left( w - \int w \right) - \partial_x^2 w = F_2(\zeta_2) - F_2(\zeta_1) + F_3(\zeta_2) - F_3(\zeta_1),
    \end{equation*}
    so that by multiplying by $w$ and integrating on $[0, T] \times \T$, we obtain
    \begin{equation}\label{eq:StabilityUniqueness}
        \int w^2 \dx + \int_0^T \int (\partial_x w)^2 \dx \dt \leq \| w_0 \|_{L^2}^2 + C|N| \int_0^T \int w^2 \dx \dt + \int_0^T \int w \Big( F_2(\zeta_2) - F_2(\zeta_1) + F_3(\zeta_2) - F_3(\zeta_1) \Big) \dx \dt.
    \end{equation}
    Proceeding as in the proof of Lemma \ref{l:multilinearEstimates}
    , we see that the quadratic and cubic terms in \eqref{eq:StabilityUniqueness} are bounded by
    \begin{equation*}
        \| F_2(\zeta_2) - F_2(\zeta_1) \|_{L^2} + \| F_3(\zeta_2) - F_3(\zeta_1) \|_{L^2} \lesssim \| \zeta_2 - \zeta_1 \|_{H^1} \Big( \| \zeta_1 \|_{H^1} + \| \zeta_1 \|_{H^1}^2 + \| \zeta_2 \|_{H^1} + \| \zeta_2 \|_{H^1}^2 \Big).
    \end{equation*}
    By using this bound in \eqref{eq:StabilityUniqueness} and applying Young's inequality $ab \leq \delta a^2 + \frac{1}{\delta}b^2$, we obtain the estimate
    \begin{equation*}
        \begin{split}
            \| w \|_{L^2}^2 + \int_0^T \int (\partial_x w)^2 \dx \dt & \lesssim \| w_0 \|_{L^2}^2 + \int_0^T \| w \|_{L^2} \| w \|_{H^1} \Big( \| \zeta_1 \|_{H^1} + \| \zeta_1 \|_{H^1}^2 + \| \zeta_2 \|_{H^1} + \| \zeta_2 \|_{H^1}^2 \Big) \dt \\
            & \lesssim \| w_0 \|_{L^2}^2 + \frac{1}{2} \| w \|_{L^2_T(H^1)}^2 \dt + 2 \int_0^T \| w \|_{L^2}^2 \Big( \| \zeta_1 \|_{H^1}^2 + \| \zeta_1 \|_{H^1}^4 + \| \zeta_2 \|_{H^1}^2 + \| \zeta_2 \|_{H^1}^4 \Big) \dt.
        \end{split}
    \end{equation*}
    In the upper bound above, the term $\frac{1}{2} \| w \|_{L^2_T(H^1)}^2$ can be absorbed in the left-hand side. It is then enough to apply Grönwall's lemma to  deduce the 
    estimate
    \begin{equation*}
        \| w \|_{L^2} \leq \| w_0 \|_{L^2} \exp \left( C(N) (1 + T) \big( \| \zeta_1 \|_{L^4_T(H^1)} + \| \zeta_2 \|_{L^4_T(H^1)} \big) \right),
    \end{equation*}
    which gives uniqueness of solutions.

    \medskip

    \textbf{STEP 4: Smoothness of Solutions.} We now prove that the unique solution $\zeta \in L^4([0, T) ; H^1(\T))$ is smooth on $(0,T) \times T$. For this, we will use a recursion argument to improve the regularity estimates.

    \medskip

    Consider any $0 < T_1 < T$. Because $\zeta \in L^4([0, T) ; H^1(\T))$, there exists a $t_1 < T_1/2$ such that $\zeta(t_1) \in H^1(\T)$. We therefore define the interval $I_1 = [t_1, T)$ and first show that $\theta \in L^\infty(I_1 ; H^1(\T)) \cap L^2(I_0 ; H^2(\T))$. By multiplying equation \eqref{eq:AnglePDERenormalized1} by $({\rm Id}- \partial_x^2) \theta$ and integrating, we get the estimate
    \begin{equation*}
        \begin{split}
            \frac{1}{2} \frac{\rm d}{\dt} \| \zeta \|_{H^1}^2 + \| \zeta \|_{H^2}^2 & \leq 8 \pi^2 N^2 \| \zeta \|_{H^1}^2 + \big( \| F_2 \|_{L^2} + \| F_3 \|_{L^2} \big) \| \zeta \|_{\dot{H}^2} \\
            & \lesssim 8 \pi^2 N^2 \| \zeta \|_{H^1}^2 + \frac{1}{\delta} \| F_2 \|_{L^2}^2 + \frac{1}{\delta} \| F_3 \|_{L^2}^2 + \delta\| \zeta \|_{H^2}^2,
        \end{split}
    \end{equation*}
    where the last line is due to Young's inequality $ab \leq \delta a^2 + \frac{1}{\delta}b^2$ and holds for all $\delta > 0$. By choosing $\delta$ sufficiently small, this allows to subtract $\delta \| \zeta \|_{H^2}^2$ in both terms and obtain, after integration in time
    \begin{equation*}
        \| \zeta (t) \|_{H^1}^2 + \int_{t_1}^t \| \zeta \|_{H^2}^2 \lesssim \| \zeta(t_1) \|_{H^1}^2 +  \int_{t_1}^t \big( 8 \pi^2 N^2 \| \zeta \|_{H^1}^2 + \| F_2 \|_{L^2}^2 + \| F_3 \|_{L^2}^2\big).
    \end{equation*}
    Note that in all this paragraph, the implicit constants may depend on $N$, but that is not relevant
    for the matter of regularity. Now, in order to use the fact that $\zeta \in L^4([0, T) ; H^1(\T))$, we decompose the cubic term as $\| F_3 \|_{L^2}^2 = \| F_3 \|_{L^2}^{4/3} \| F_3 \|_{L^2}^{2/3}$, so that, by Lemma \ref{l:multilinearEstimates}, we have
    \begin{equation*}
        \| F_2 \|_{L^2(I_1 ; L^2)} \lesssim \| \zeta \|_{L^4_T(H^1)}^2, \qquad \| F_3(t) \|_{L^2}^{2/3} \leq \| \zeta(t) \|_{H^1}^2 \qquad \text{and} \qquad \| F_3 \|_{L^{4/3}(I_1 ; L^2)} \leq \| \zeta \|_{L^4_T(H^1)}.
    \end{equation*}
    Plugging this in our integral inequality, we see that
    \begin{equation}\label{eq:order0Reg}
        \| \zeta (t) \|_{H^1}^2 + \int_{t_1}^t \| \zeta \|_{H^2}^2 \lesssim \| \zeta(t_1) \|_{H^1}^2 + \| \zeta \|_{L^4_T(H^1)}^2 + \int_{t_1}^t \big( 1 + \| F_3 \|_{L^2}^{4/3} \big) \| \zeta \|_{H^1}^2.
    \end{equation}
    By using Grönwall's inequality in \eqref{eq:order0Reg} above, we derive estimates for $\|\zeta\|_{L^\infty(I_1 ; H^1)}$, and consequently, using these estimates in \eqref{eq:order0Reg}, we also get that $\zeta \in L^2(I_1 ; H^2(\T))$.

    \medskip

    We now use the fact that $\zeta \in L^\infty(I_1 ; H^1(\T)) \cap L^2(I_1 ; H^2(\T))$ as the first step of an induction argument. We can consider $t_1 < t_2 < T_1 \left( \frac{1}{2} + \frac{1}{2^2} \right)$ such that $\zeta(t_2) \in H^2(\T)$. We then set $s = 2$ and multiply equation \eqref{eq:AnglePDERenormalized1} by $({\rm Id} - \partial_x^2)^s\zeta$, so as to obtain
    \begin{equation*}
        \frac{1}{2} \frac{\rm d}{\dt} \| \zeta \|_{H^s}^2 + \| \zeta \|_{H^{s+1}}^2 \leq \big( \| F_2 + F_3 \|_{H^{s-1}} + 8 \pi^2 N^2 \| \zeta \|_{H^{s-1}} \big) \| \zeta \|_{H^{s+1}}.
    \end{equation*}
    Once again, we use Young's inequality $ab \leq \delta a^2 + \frac{1}{\delta}b^2$ to absorb the $\| \zeta \|_{H^{s+1}}$ factor in the left-hand side. This yields
    \begin{equation*}
        \frac{1}{2} \frac{\rm d}{\dt} \| \zeta \|_{H^s}^2 + \| \zeta \|_{H^{s+1}}^2 \lesssim \| \zeta \|_{H^{s-1}}^2 + \| F_2 + F_3 \|_{H^{s-1}}^2.
    \end{equation*}
    Observe that $s -1 = 1 > \frac{1}{2}$, so that the Sobolev space $H^{s-1}(\T)$ is a Banach algebra. This implies that we can bound both the bilinear $F_2$ from \eqref{eq:F2bilinear} by
    \begin{equation*}
        \| F_2 \|_{H^{s-1}} \lesssim \| \zeta \|_{H^{s-1}} \| \zeta \|_{H^s} \lesssim \| \zeta \|_{H^s}^2
    \end{equation*}
    Note that, in the above, the implicit constant depends on $N$. In order to estimate the trilinear term $F_3$, we note that, even though $H^{s-3/2}(\T)$ is not a Banach algebra, the pointwise function product is bounded from $H^{s-3/2}(\T) \times H^{s-3/2}(\T) \longrightarrow H^{s-2}(\T)$ (see for example Theorems 2.82 and 2.85 in \cite{BCD}). Therefore
    \begin{equation*}
        \begin{split}
            \| F_3 \|_{H^{s-1}} & \lesssim \| \zeta \|_{H^s} \| u \|_{H^{s-1}} \lesssim \| \zeta \|_{H^s} \| (\partial_x \zeta)^2 \|_{H^{s-2}} \\
            & \lesssim \| \zeta \|_{H^s} \| \partial_x \zeta \|_{H^{s-3/2}}^2\\
            & \lesssim \| \zeta \|_{H^s} \| \zeta \|_{H^{s - 1/2}}^2.
        \end{split}
    \end{equation*}
    By integrating in time, we get the integral inequality
    \begin{equation*}
        \| \zeta \|_{H^s}^2 + \int_{t_2}^t \| \zeta \|_{H^{s+1}}^2 \lesssim \| \zeta(t_2) \|_{H^s}^2 + \int_{t_2}^t \| \zeta \|_{H^s}^2 \big( 1 + \| \zeta \|_{H^{s - 1/2}}^4 \big).
    \end{equation*}
    Finally, because we have already proven that $\zeta \in L^\infty(I_1 ; H^{s-1}(\T)) \cap L^2(I_1 ; H^s(\T))$, we know that $\zeta \in L^4(I_2 ; H^{s - \frac{1}{2}}(\T))$. This means that we can use Grönwall's lemma in the previous inequality, and deduce that $\zeta \in L^\infty(I_2 ; H^s(\T)) \cap L^2(I_2 ; H^{s+1}(\T))$.

    \medskip

    We can repeat this argument on a set of nested intervals $I_1 \subset I_2 \subset I_3 \cdots \subset [T_1 ; T) := I$, where $I_k  = [t_k, T)$ and $t_{k-1} < t_k < \left( \frac{1}{2} + \frac{1}{2^2} + \cdots \frac{1}{2^k} \right)T_1 < T_1$, and deduce that $\zeta \in L^\infty(I ; H^{s + k}(\T))$ for any $k \geq 0$, which gives smoothness with respect to the space variable. To get smoothness with respect to the time variable, we use the equation:
    \begin{equation*}
        \partial_t \zeta = \partial_x^2 \zeta - F_2 + F_3 + 8 \pi^2 N^2 \zeta \in L^\infty(I ; H^{s + k - 2}(\T))
    \end{equation*}
    so that $\zeta \in W^{1, \infty}_{\rm loc}(I ; H^{s+k - 2}(\T))$ for any $k \geq 0$. We can repeat this argument a finite number of times $\ell$ to show that $\zeta \in W^{\ell, \infty}(I ; H^{s + k - 2\ell}(\T))$ for any $k$ sufficiently large, and hence $\zeta \in C^\infty(I \times \T)$. Since $T_1 > 0$ is arbitrary, we deduce the result, that $\zeta \in C^\infty((0,T) \times \T)$.

    \medskip

    \textbf{STEP 5: Time Continuity.} Finally, we have to make sure that the solution is continuous in time up to $t = 0$ in some topology, so as to make sure the initial data can be written in the sense of distributions. According to Lemma \ref{l:multilinearEstimates} and the computations from Step 2, we have $F_2 + F_3 \in L^{4/3}([0, T) ; L^2(\T))$, so that
    \begin{equation*}
        \begin{split}
            \partial_t \zeta  = \, & \partial_x^2 \zeta + 8\pi^2N^2 \left( \zeta - \int_\T \zeta \dx \right) + F_2 + F_3 \\
            & \in L^{4/3}_{loc}([0, T) ; H^{-1}(\T)).
        \end{split}
    \end{equation*}
    This provides Hölder continuity for the solution with respect to the time variable $\zeta \in C^{0, \frac{1}{4}}([0, T) ; H^{-1}(\T))$, and therefore makes sure that the initial data condition can be written in the sense of distributions.
\end{proof}

As a final remark, we have show that the solution $(\theta, u)$ we have constructed can be used to obtain a solution where the velocity field cancels at $x = 0$ by means of the change of reference frames \eqref{ieq:Galileo}. For this, we assume some little extra regularity on the initial data.

\begin{prop}\label{p:Galileo}
    Consider $s > 1/2$ and let $\theta_0 \in H^s(\T)$ be some initial datum. Then problem \eqref{eq:AngleVariablePDE} has a unique solution $(\theta, u)$ such that
    \begin{equation*}
        \theta \in L^4([0, T) ; H^1) \qquad \text{and} \qquad u(t,0) = 0.
    \end{equation*}   
\end{prop}

\begin{proof}
    Consider $\zeta \in L^4([0, T) ; H^1(\T))$ be the unique solution of \eqref{eq:AnglePDERenormalized1}-\eqref{eq:AnglePDERenormalized2} given by Theorem \ref{t:existencelimit} with $N = \int_0^1 \theta_0 \dx$, and set $\vartheta(t,x) := \zeta(t, x) + 2\pi Nx$. In order to define a suitable change of reference frames, we consider the ODE problem
    \begin{equation}\label{eq:ODEChange}
        g'(t) + w \big( t, g(t) \big) = 0 \qquad \text{and} \qquad g(0) = w(0, 0),
    \end{equation}
    where $w$ is the unique solution of
    \begin{equation}\label{eq:VelocityLL}
        \partial_x w = - (\partial_x \vartheta)^2 + \int_0^1 (\partial_x \vartheta)^2 \dx, \qquad \text{and} \qquad \int_\T w \dx = 0.
    \end{equation}
    We will show that the ODE problem \eqref{eq:ODEChange} has a unique solution. For this, we check that the velocity field $w$ has at least $L^1([0, T) ; W^{1, \infty}(\T))$ regularity. By Theorem \ref{t:existencelimit}, we know that $\zeta \in L^4([0, T) ; H^1(\T))$, and therefore we can apply Lemmas \ref{l:LinearEstimates} and \ref{l:multilinearEstimates} to show that
    \begin{equation*}
        \| \zeta \|_{L^2_T(H^{s + 1})} \lesssim \| \zeta_0 \|_{H^s} + C(N, T) \| \zeta \|_{L^4_T(H^1)}^2 + \| \zeta \|^3_{L^4_T(H^1)}.
    \end{equation*}
    In particular, the solution $\zeta$ is Lipschitz in space $H^{s + 1} \subset W^{1, \infty}$, and the space $H^{s + 1}$ is a Banach algebra. This implies that the velocity $w$ defined in \eqref{eq:VelocityLL} lies in $L^2([0, T) ; W^{1, \infty}(\T))$. The Cauchy-Lipschitz then applies to show that the ODE problem \eqref{eq:ODEChange} has a unique solution. Therefore, by setting
    \begin{equation*}
        \theta(t, x) = \vartheta \big( t, x - g(t) \big) \qquad \text{and} \qquad u(t,x) = w \big( t, x - g(t) \big) + g'(t),
    \end{equation*}
    we see that $(\theta, u)$ is a solution of \eqref{eq:AngleVariablePDE} such that $u(t,0) = 0$ at every time $t \in [0, T)$. Moreover, this is the only such solution, by uniqueness of solutions of \eqref{eq:ODEChange}.
\end{proof}

\subsection{Global Well-Posedness for Initial Data with Small Oscillations}

The paragraph above shows that solutions of \eqref{eq:AnglePDERenormalized1}-\eqref{eq:AnglePDERenormalized2} can be global if $N = 0$, under a smallness condition on the initial datum $\| \theta_0 \|_{H^{1/2}} \leq \eta_0$. In this case, the magnetic field is ``topologically trivial'', in other words, the curve $x \in \T \longmapsto (b_1(x), b_2(x))$ never performs a complete circle around the origin $(0, 0)$. In that case, $\theta = \zeta$ and the problem reduces to study \eqref{eq:AngleVariablePDE}.

\medskip

In this paragraph, we go one step further, and prove that solutions are global if $N = 0$, provided the (not necessarily small) initial datum $\theta_0 \in H^{1/2}$ fulfils a small oscillation condition
\begin{equation*}
    \max_{x \in \T} \theta_0(x) - \min_{x \in \T} \theta_0(x) \leq \kappa_0.
\end{equation*}

\begin{teor}\label{t:smallOscillations}
    Assume that $N = 0$ and let $\theta_0\in H^{s}(\T)$ for $s>1/2$ be an initial datum $\theta_0$ such that we have the following oscillation bound:
    \begin{equation}\label{eq:smallOscillation}
            \max_{x \in \T} \theta_0(x) - \min_{x \in \T} \theta_0(x) \leq \frac{1}{\sqrt{3}} \approx 0.577,
    \end{equation}
    then the unique solution $\theta \in L^4([0, T) ; H^1(\T))$ of \eqref{eq:AngleVariablePDE} is global in time.
\end{teor}

\begin{obs}
    Observe that the smallness condition on $\max \theta_0 - \min \theta_0$ is invariant under the addition of a constant $\Theta_0 \in \R$ to the initial datum, in accordance with Remark \ref{r:AngleAdditionInvariance} (compare also with Remark \ref{r:HomogeneousSobolev}).
\end{obs}

\begin{proof}
    The proof is in two steps. First we show by means of maximum principle that the small oscillation condition holds for any time, provided it does for the initial datum. Then, we show that this enables us to obtain energy estimates on the derivative $\varphi := \partial_x \theta$.

    \medskip

    \textbf{STEP 1: Maximum Principle.} We assume that the initial datum satisfies \eqref{eq:smallOscillation} for some $\kappa_0$, which we will fix later on, and we consider the unique associated (maximally extended) solution $\theta \in C^\infty((0, T) \times \T)$ given by Theorem \ref{t:existencelimit}. From the equation
    \begin{equation*}
        \partial_t \theta + u \partial_x \theta = \partial^2_x \theta,
    \end{equation*}
    observe that the maximum principle implies that the maximum $\max_x \theta(t,x)$ and minimum $\min_x \theta(t,x)$ of the solution are, respectively, non-increasing and non-decreasing functions of time. Therefore, we deduce that the total oscillation of the solution is also a non-increasing function of time and
    \begin{equation*}
        \max_{x \in \T} \theta (t, x) - \min_{x \in \T} \theta (t, x) \leq \max_{x \in \T} \theta_0(x) - \min_{x \in \T} \theta_0(x) \leq \kappa_0,
    \end{equation*}
    for any $0 < t < T$, which is what we were aiming at in this step.

    \medskip

    \textbf{STEP 2: Higher Order Energy Estimates.} Set $\varphi = \partial_x \theta$, and consider $0 < t_1 < T$. Since the solution $\theta \in C^\infty ((0, T) \times T)$ is smooth, the function $\varphi(t_1)$ also is, and the quantity $\varphi$ is a smooth solution to the following PDE problem:
    \begin{equation}\label{eq:FirstDerivativeEquation}
        \begin{cases}
            \partial_t \varphi + \partial_x (u \varphi) = \partial_x^2 \varphi \\
            \partial_x = - \varphi^2 + \int_\T \varphi^2,
        \end{cases}
        \qquad \text{with datum } \varphi_{t = t_1} =  \varphi(t_1).
    \end{equation}
    We may therefore perform an energy estimate on this PDE. Multiplying the top equation by $\varphi$ and integrating in the space variable, we get
    \begin{equation}\label{eq:OscillationEnergy}
        \frac{1}{2} \frac{\rm d}{\dt} \int_\T \varphi^2 + \int_\T (\partial_x \varphi)^2 = - \int_\T \varphi \partial_x(u \varphi).
    \end{equation}
    Let us focus on the right-hand side of this equation. Integrating by parts repeatedly, we see that 
    \begin{equation}
        - \int_\T \varphi \partial_x(u \varphi) = \int_\T u \varphi \partial_x \varphi = \frac{1}{2} \int_\T u \partial_x (\varphi^2) = - \frac{1}{2} \int_\T \varphi^2 \partial_x u.
    \end{equation}
    We then replace $\partial_x u$ by its value as a function of $\varphi$ from the second equation in \eqref{eq:FirstDerivativeEquation}, and obtain the estimate
    \begin{equation}\label{eq:OscillationPreperation}
        \frac{1}{2} \frac{\rm d}{\dt} \int_\T \varphi^2 + \frac{1}{2} \left( \int_\T \varphi^2 \right)^2 + \int_\T (\partial_x \varphi)^2 = \frac{1}{2} \int_\T \varphi^4.
    \end{equation}
    Now, we wish to make the oscillation appear in this equation. To that end, we introduce the mean-free part of the angular variable, 
    \begin{equation*}
        \Gamma := \theta - \int_\T \theta
    \end{equation*}
    and note that $\partial_x \Gamma = \varphi$, so that, integrating by parts one more time, we may continue the computation in \eqref{eq:OscillationPreperation} as
    \begin{equation*}
        \frac{1}{2} \int_\T \varphi^4 = - \frac{3}{2} \int_\T \Gamma \varphi^2 \partial_x \varphi \leq \frac{3}{2} \| \Gamma \|_{L^\infty} \| \varphi \|_{L^4}^2 \| \partial_x \varphi \|_{L^2}.
    \end{equation*}
    We now use Young's inequality $ab \leq \frac{1}{2\lambda}a^2 + \frac{1}{2}\lambda b^2$ to obtain a bound for $\| \varphi \|_{L^4}$. For any $\lambda > 0$,
    \begin{equation*}
        \frac{1}{2} \int_\T \varphi^4 \leq \frac{1}{2 \lambda} \int_\T \varphi^4 + \frac{3}{4}\lambda \| \Gamma \|_{L^\infty}^2 \| \partial_x \varphi \|_{L^2}^2,
    \end{equation*}
    and so, provided $0 < \frac{1}{\lambda} < \frac{1}{2}$,
    \begin{equation}\label{eq:L4varPhiIneq}
        \begin{split}
            \frac{1}{2} \int_\T \varphi^4 & \leq \frac{3 \lambda / 4}{1 - 1/\lambda} \| \Gamma \|_{L^\infty}^2 \| \partial_x \varphi \|_{L^2}^2 \\
            & \leq 3 \| \Gamma \|_{L^\infty}^2 \| \partial_x \varphi \|_{L^2}^2 \qquad \text{with } \lambda = 2.
        \end{split}
    \end{equation}
    Our goal is now to bound the right-hand side of this inequality, and show that if $\kappa_0$ is small enough, then it can be absorbed in the $\| \partial_x \varphi \|_{L^2}^2$ term in the left-hand side of \eqref{eq:OscillationEnergy}. First of all, we note that for any $t_1 < t < T$, then the mean value theorem implies there exists a $x_\star \in \T$ such that $\int \theta(t) = \theta(t,x_\star)$, and so
    \begin{equation*}
        |\Gamma| = |\theta(t,x) - \theta(t,x_\star)| \leq \max_x \theta(t) - \min_x \theta(t) \leq \kappa_0.
    \end{equation*}
    Therefore, by using this in inequality \eqref{eq:L4varPhiIneq}, and plugging it in equation \eqref{eq:OscillationPreperation}, we obtain
    \begin{equation*}
        \frac{1}{2} \frac{\rm d}{\dt} \| \varphi \|_{L^2}^2 + \big( 1 - 3 \kappa_0^2 \big) \| \partial_x \varphi \|_{L^2}^2 \leq 0,
    \end{equation*}
    and deduce that $\varphi \in L^\infty((t_1, \infty) ; L^2)$ provided that $\kappa_0 \leq \frac{1}{\sqrt{3}}$.
\end{proof}

\subsection{Blow Up for Solutions of the Limit System}
	In this section we show that there are non-negative solutions of the equation 
	
	\begin{equation}\label{equation}
		\left\{\begin{array}{ll}
			\partial_t \varphi+\partial_x(u\varphi)=\partial_x^2\varphi & \text{in }(0,T)\times\T\\
			\partial_x u=\|\varphi\|_{L^2}^2-\varphi^2(x)& \text{in }(0,T)\times\T\\
			\varphi(x,0)=\varphi_0& \text{in }\T
		\end{array}\right. 
	\end{equation}
that blow up in finite time, assuming that the initial data satisfies a concentration condition. We begin by stating some lemmas that allow us to make the subsequent construction, and that are nothing  else than an application of a maximum principle and of local well posedness of the equation. 
	
\begin{lemma}\label{positividad}Let $\varphi_0$ be a non-negative smooth function. Then, consider $\varphi$ to be the unique, local-in-time solution of \eqref{equation} with initial value $\varphi_0$. Then, the following holds: 

\begin{itemize}
    \item $\varphi(t,x)\geq 0$ for every $t>0,\,x\in \T$.
    \item If $\varphi_0$ is even, then $\varphi(t,\cdot)$ remains even for every $t>0$.
    \item The total mass remains constant, \textit{i.e.}
    $$\int_\T \varphi(t,x)=\int_\T \varphi_0(x)\dx\quad \text{for all }t>0.$$
    \item If $\varphi_0$ is decreasing in $(0,1/2)$, then it remains decreasing in such interval. 
\end{itemize}
\end{lemma}

We are now in a position to prove the main theorem.

\begin{teor}\label{prop:blowup}
	Let $\varphi_0$ be a $C^2(\T)$, 1-periodic, non-negative, even function (when considered as a function $\varphi_0:\R\rightarrow\R$) which is not identically zero, and let $\varphi$ be the solution of $\eqref{equation}$ with initial value $\varphi_0$. Denote the total mass and the second moment of $\varphi$ by 
	
	\begin{equation} \label{eq:meanvariance}
    M:=\int_{-1/2}^{1/2} \varphi_0(x)\dx \quad \text{and } V(t):=\int_{-1/2}^{1/2}x^2\varphi(t,x)\dx,
    \end{equation}
    respectively. Then, there is a constant $C>0$ such that, if 

    \begin{equation}\label{eq:blowupcondition} 
    M> C\cdot V(0),
    \end{equation}
    $\varphi$ blows up in finite time. 
\end{teor}
\begin{obs}
    Heuristically, the quantity $V(t)$ in \eqref{eq:meanvariance} measures the concentration around $0$. Therefore, the condition \eqref{eq:blowupcondition} implies that if the initial data is very concentrated in comparison with its total mass, then there is a blow up in finite time. 
\end{obs}
\begin{proof}[Proof of Proposition \ref{prop:blowup}]
	We will prove that, if $\varphi$ does not blow up in finite time, then $V(t)$ becomes negative for some finite positive time. This, however, contradicts Lemma \eqref{positividad}, and it finishes the proof. 
	
	To that end, we shall study the function 
	
	$$\int_{-1/2}^{1/2}x^2\varphi(t,x)\,\dx.$$
	
	We take the derivative and use \eqref{equation} to obtain 
	
	\begin{equation}\label{eq:I1I2}
		\begin{split}
			\frac{d}{dt}\int_{-1/2}^{1/2}x^2\varphi(t,x)\dx=\int_{-1/2}^{1/2}x^2\partial_t\varphi(t,x)\dx=\int_{-1/2}^{1/2}x^2\partial_x^2\varphi(t,x)\dx-\int_{-1/2}^{1/2}x^2\partial_x(u\varphi)\dx=I_1+I_2
		\end{split}
	\end{equation}

We now analyse each term in \eqref{eq:I1I2} separately. 

$$I_1=\int_{-1/2}^{1/2}\partial_x(x^2\partial\varphi)\dx-2\int_{-1/2}^{1/2}x\partial_x\varphi(x)\dx=-2\int_{-1/2}^{1/2}x\partial_x\varphi(x)\dx,$$
where we have used that, as $\varphi$ is even and $1$ periodic, its derivative vanishes in $|x|=1/2$. On the other hand, the second term can be estimated as 

\begin{align*} -2\int_{-1/2}^{1/2}x\partial_x\varphi(x)\dx&=-4\int_0^{1/2}x\partial_x\varphi(t,x)\dx\\
	&=-4\int_0^{1/2}\partial_x\left(x\varphi(t,x)\right)\dx+4\int_{0}^{1/2}\varphi(t,x)\dx\\
	&\leq 2\int_{-1/2}^{1/2}\varphi_0(t)\dt,
\end{align*}
where in the first line we used that $x\partial_x \varphi(t,x)$ is even, and in the third line we used that $\varphi$ is non-negative, and that $\varphi$ is even. As a result, if we denote by $M$ to the total mass of $\varphi_0$, we conclude that 

$$I_1\leq 2 M.$$

We now estimate $I_2$ in \eqref{eq:I1I2}. Using that $u$ vanishes on $\pm 1/2$ due to the symmetry of $\varphi$, we can write 

\begin{align*} 
	I_2&=2\int_{-1/2}^{1/2}xu(t,x)\varphi(t,x)\,\dx\\
	&=2\left(\int_{-1/2}^{1/2}x^2\varphi(t,x)\,\dx\right)\left(\int_{-1/2}^{1/2}\varphi^2(t,x)\dx\right)-2\int_{-1/2}^{1/2}x\varphi(t,x)\left(\int_{0}^{x}\varphi^2(t,y)\dy\right)\dx\\
	&=8\left(\int_{0}^{1/2}x^2\varphi(t,x)\,\dx\right)\left(\int_{0}^{1/2}\varphi^2(t,x)\dx\right)-4\int_{0}^{1/2}x\varphi(t,x)\left(\int_{0}^{x}\varphi^2(t,y)\dy\right)\dx\\
	&=I_{21}+I_{22}.
\end{align*}

Note that, then, 

\begin{align} 
	2\left(\int_{0}^{1/2}x^2\varphi(t,x)\,\dx\right)\left(\int_{0}^{1/2}\varphi^2(t,x)\dx\right)&=2\int_{0}^{1/2}x^2\varphi(t,x)\left(\int_{x}^{1/2}\varphi^2(t,y)\dy\right)\dx\nonumber\\
	&\label{primersumando}+2\int_{0}^{1/2}x^2\varphi(t,x)\left(\int_{0}^{x}\varphi^2(t,y)\dy\right)\dx
\end{align}

For the first integral, we will use the following inequality, that relies on the fact that $\varphi$ is decreasing on $[0,1/2]$,

$$\varphi(t,x)\leq \frac{1}{x}\left(\int_0^{1/2}\varphi(t,x)\,\dx\right).$$ 

Therefore,

\begin{align*} 2\int_{0}^{1/2}x^2\varphi(t,x)\left(\int_{x}^{1/2}\varphi^2(t,y)\dy\right)\dx&\leq 2 \int_0^{1/2} x^2 \varphi(t, x) \left\{ \int_x^{1/2} \varphi(t, y) \left( \frac{1}{y} \int_0^{1/2} \varphi(t, z) \dz \right) \dy \right\} \dx
\\ & \leq \left(\int_0^{1/2}\varphi(t,x)\dx\right)^2\left(\int_0^{1/2}x\varphi(t,x)\dx\right)\\
&\leq\left(\int_0^{1/2}\varphi(t,x)\dx\right)^{5/2}\left(\int_0^{1/2}x^2\varphi(t,x)\dx\right)^{1/2}
\end{align*}

On the other hand, we would like to obtain a finer estimate for the second integral in \eqref{primersumando}. Since $x\leq 1/2$, we have that

$$2\int_{0}^{1/2}x^2\varphi(t,x)\left(\int_{0}^{x}\varphi^2(t,y)\dy\right)\dx-\int_{0}^{1/2}x\varphi(t,x)\left(\int_{0}^{x}\varphi^2(t,y)\dy\right)\dx$$
is smaller or equal than zero. The goal then is to give a quantitative estimate that allows us to prove blow-up. We can write the expression above as

$$-2\int_{0}^{1/2}x(1/2-x)\varphi(t,x)\left(\int_{0}^{x}\varphi^2(t,y)\dy\right)\dx.$$

As $\varphi$ is positive and $x\leq 1/2$, the integrand is non-negative. As a result, and due to the minus sign in the beginning of the expression, we can make use of Cauchy-Schwarz to estimate this integral by 

\begin{equation*} 
	\begin{split} 
		2\int_{0}^{1/2}x(1/2-x)\varphi(t,x)\left(\int_{0}^{x}\varphi^2(t,y)\dy\right)\dx&\geq 2\int_0^{1/2}(1/2-x)\varphi(t,x)\left(\int_0^x\varphi(t,y)\dy\right)^2\dx\\
		&\frac{2}{3}\int_0^{1/2}(1/2-x)\frac{\partial}{\partial x}\left(\left(\int_0^x\varphi(t,y)\dy\right)^3\right)\dx
	\end{split}
\end{equation*}

We now integrate by parts and, since the integrand vanishes in $x=0,1/2$, we get no boundary terms. As a result, we can write 

\begin{equation*} 
	\begin{split} 
		-2\int_{0}^{1/2}x(1/2-x)\varphi(t,x)\left(\int_{0}^{x}\varphi^2(t,y)\dy\right)\dx&\leq -\frac{2}{3}\int_0^{1/2}(1/2-x)\frac{\partial}{\partial x}\left(\left(\int_0^x\varphi(t,y)\dy\right)^3\right)\dx\\
		&=-\frac{2}{3}\int_0^{1/2}\left(\int_0^x\varphi(t,y)\dy\right)^3\dx\\
		&\leq -\frac{4}{3}\left(\int_0^{1/2}\left(\int_0^x\varphi(t,y)\right)\dx\right)^3\\
		&=-\frac{4}{3}\left(\int_0^{1/2}\varphi(t,y)\int_x^{1/2}\dx\dy\right)^3\\
		&=-\frac{4}{3}\left(\int_0^{1/2}\varphi(t,y)(1/2-y)\dy\right)^3\\
		&=\frac{4}{3}\left(\int_0^{1/2}\varphi(t,y)(y-1/2)\dy\right)^3.
	\end{split}
\end{equation*}

This quantity is strictly negative. Indeed, we can estimate it in terms of the second moment and the total mass, since 

$$\int_0^{1/2}y\varphi(t,y)\dy\leq \left(\int_0^{1/2} y^2\varphi(t,y)\dy\right)^{1/2}\left(\int_0^{1/2}\varphi(t,y)\dy\right)^{1/2}.$$
Therefore,

$$-2\int_{0}^{1/2}x(1/2-x)\varphi(t,x)\left(\int_{0}^{x}\varphi^2(t,y)\dy\right)\dx\leq \frac{4}{3}\left(\frac{1}{2}V^{1/2}M^{1/2}-\frac{1}{4}M\right)^3,$$
where the second moment $V$ and the mass $M$ are defined as 

$$V:=\int_{-1/2}^{1/2}x^2\varphi(t,x)\dx \quad \text{ and }\quad M:=\int_{-1/2}^{1/2}\varphi(t,x)\dx.$$

As a result, the second moment $V(t)$ satisfies the following differential inequality

\begin{equation} \label{eq:estimatevprime} V'(t)\leq 2M+\frac{1}{4}M^{5/2}V^{1/2}(t)+\frac{4}{3}\left(V^{1/2}(t)M^{1/2}-\frac{1}{4}M\right)^3.
\end{equation}

Note that we have not indicated the dependence of $M$ on time because it is a constant that only depends on the initial value.

Denote now $V_0:=V(0)$. We now show that there exists $\overline{M}>0$ and $\kappa<0$, depending on $V_0$ such that, if $M\geq \overline{M}$, $V'(t)\leq \kappa<0$ for every $t\geq 0$. This implies that $V(t)$ becomes negative in finite time, but this is impossible because $V(t)$ is always non-negative. 
 
Notice that the leading order in the right hand side of \eqref{eq:estimatevprime} is $-M^3$. Therefore, there is an $\overline{M}>0$ depending on $V_0$, such that if $M>\overline{M}$, $V'(0):=\kappa<0$. Consider then  

$$t_\star:=\sup\,\{t\geq 0\;:\, V'(s)\leq \kappa\text{ for every } s\leq t\}.$$

Due to the discussion above, this set is not empty and $t_\star\geq 0$. Now, assume that $t_\star<\infty$. Since $\kappa<0$, there exists some $\alpha>0$ such that $V'(t)\leq 0$ in $[0,t_\star+\alpha]$. Notice that the right hand side in \eqref{eq:estimatevprime} is increasing as a function of $V$, and $V(t)$ is decreasing in $[0,t_\star+\alpha]$. Therefore, 

$$V'(t)\leq V'(0)=\kappa \quad \text{for }t\in [0,t_{\star}+\alpha].$$

However, this contradicts the definition of $t_\star$ as a supremum. Therefore, $t_\star=\infty$. But then, $V'(t)\leq \kappa<0$ for every $t>0$. This implies that $V(t)$ becomes negative in finite time.
\end{proof}

\subsection{Numerical Experiments}\label{sec:numerics}

In this final section, we implement a numerical solver for the limit model, and use it to illustrate the possible blow-up or global existence of smooth solutions. For the reader's convenience, we reproduce the PDE here:
\begin{equation*}
    \begin{cases}
        \partial_t \theta + u \partial_x \theta = \partial_x^2 \theta \\
        u(x) = - \int_0^x (\partial_x \theta)^2 + x \int (\partial_x \theta)^2.
    \end{cases}
\end{equation*}
We restrict our attention to the case $N = 0$, where the magnetic field does not make a full turn around the origin. Therefore, the equations are set on the torus $\T$.

\subsubsection{Numerical Scheme}

For any large integral parameters $M_x, M_t \geq 1$, we introduce the space grid $x_j := 0, \delta x, 2 \delta x, ..., j \delta x, (M_x - 1) \delta x$ and the time discretization $t_n := 0, \delta t, 2 \delta t, ..., n \delta t, ..., (M_t - 1) \delta t$. The numerical approximation of the solution $\theta(t_n, x_j)$ is denoted $\theta_j^n$.

\medskip

We consider the following semi-implicit finite-difference scheme
\begin{equation*}
    \frac{\theta^{n+1}_j - \theta_j^n}{\delta t} + \frac{1}{2} u_j^n \frac{\theta^{n+1}_{j+1} - \theta^{n+1}_{j-1}}{2 \delta x} - \frac{1}{2} \frac{\theta_{j+1}^{n+1} + \theta^{n+1}_{j-1} - 2 \theta^{n+1}_j}{\delta x^2} = \frac{1}{2} u_j^n \frac{\theta^{n}_{j+1} - \theta^{n}_{j-1}}{2 \delta x} + \frac{1}{2} \frac{\theta^n_{j+1} + \theta^n_{j-1} - 2 \theta^n_j}{\delta x^2}
\end{equation*}
\begin{equation*}
    u_j^n = - \frac{1}{2} \sum_{k = 1}^j \left[ \left( \frac{\theta_{j+1}^n - \theta_{j-1}^n}{2\delta x} \right)^2 + \left( \frac{\theta_{j}^n - \theta_{j-2}^n}{2\delta x} \right)^2 \right] + \frac{j}{2N_x} \sum_{k = 1}^{M_x - 1} \left[ \left( \frac{\theta_{j+1}^n - \theta_{j-1}^n}{2\delta x} \right)^2 + \left( \frac{\theta_{j}^n - \theta_{j-2}^n}{2\delta x} \right)^2 \right],
\end{equation*}
with periodic boundary conditions $\theta_{M_x}^n = \theta_0^n$ and $u_{M_x}^n = u_0^n$. Note that the sum $\sum_{1}^0$ is zero by convention. Let us give some explanation about this numerical scheme.
\begin{enumerate}
    \item The velocity $u_j^n$ is computed by trapezoidal integration, explicitly from $\theta^n$. Using $u^{n+1}_j$ would significantly increase the complexity of computations, which is why we use an explicit velocity instead of computing it semi-implicitly, at the cost of a slight loss of precision in time ($O(\delta t)$ instead of $O(\delta t^2)$).

    \item Because we are interested in studying blow-up solutions, the velocity may become very large. To prevent the scheme from becoming unstable, it must be at least partially implicit. In order to limit numerical diffusion while retaining stability independently of the size of $u_j^n$, we resort to a semi-implicit method (Crank-Nicolson).

    \item We could have treated the diffusion fully explicitly under the appropriate CFL condition. However, the scheme, as it stands, is unconditionally stable.   
\end{enumerate}

\subsubsection{Numerical Experiments}

In this paragraph, we perform a few numerical experiments using the scheme defined above. The implementation is written with SciLab.

\medskip

\textbf{Blow-Up Scenario.} In our first experiment, we find evidence of singularity formation, even when the total number of turns of the magnetic field around the origin is zero $N = 0$. We consider the following initial datum (see Figure \ref{fig:BlowUp})
\begin{equation}\label{eq:NumericlInitial}
    \theta_0(x) = \frac{3}{5}\sin(2 \pi x) \cos(6 \pi x) \cos(2 \pi x) - \frac{12}{5} \cos(4 \pi x)\cos(2 \pi x) + 3 \sin (2 \pi x),
\end{equation}
which is $1$-periodic. The choice of this initial datum is simply linked to the numerical evidence of blow-up solutions that stem from $\theta_0$. However, this is not a unique behavior, as many of the initial data we tried lead to blow-up.

Note that this initial datum has nothing special about it, other than the fact that numerical evidence suggests it leads to blow-up. In fact, many of the initial data we tried lead to blow-up.

\medskip

Our simulation is performed with mesh size $\delta x = 2.5 \cdot 10^{-3}$ and $\delta t = 6.25 \cdot 10^{-7}$ up to time $T = 6.619 \cdot 10^{-4}$ (1060 iterations). We observe that the derivative of the solution $\partial_x \theta$ becomes very large as we approach $T$. 

\begin{obs}
    Although our numerical scheme is unconditionally stable (hence no CFL condition is required) the size of discretization $\delta t = 0.1 \cdot \delta x^2$ is simply designed to be consistent with the parabolic scaling of the problem. By taking $\delta t$ too large, we would risk not observing the blow-up which occurs at time $T^\star \approx 6 \cdot 10^{-4} \ll \delta x$. In addition, choosing small time steps $\delta t \ll \delta x$ allows to limit numerical diffusion of the scheme.
\end{obs}

\begin{figure}[h!]
    \centering
    \includegraphics[width=0.9\linewidth]{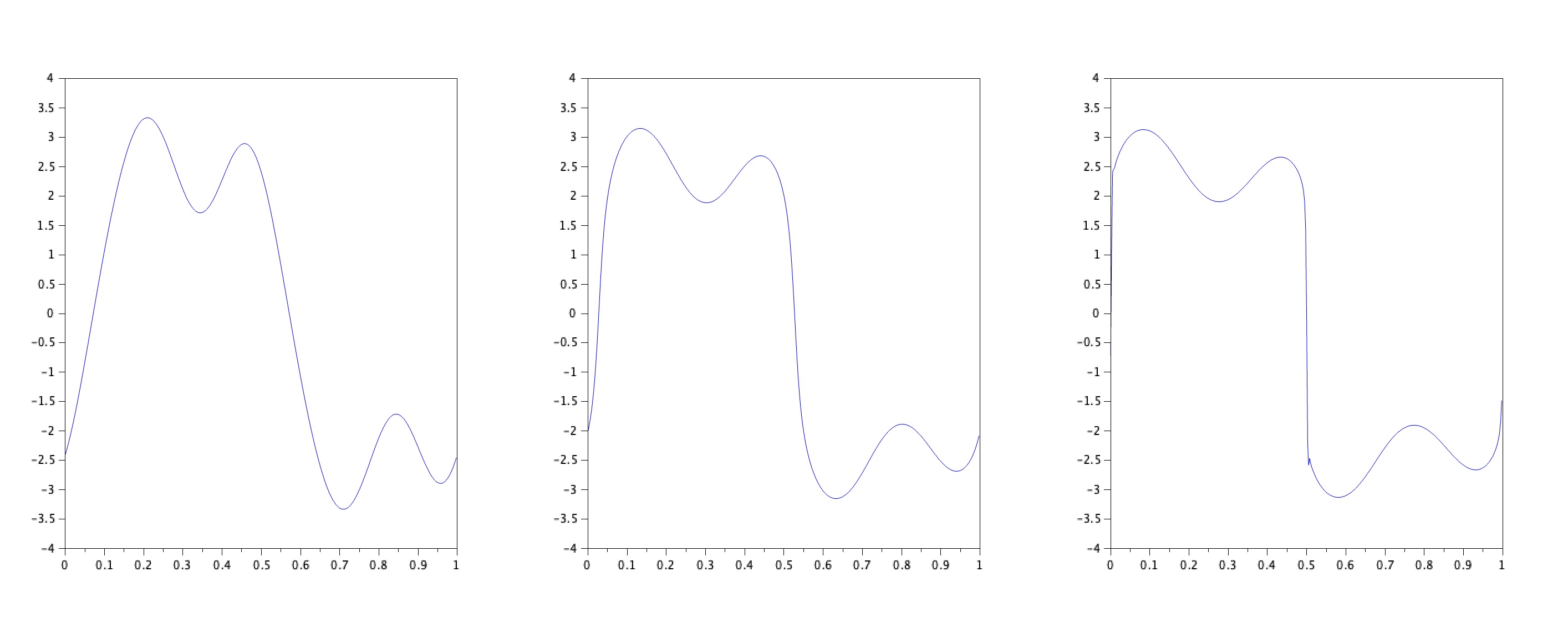}
    \caption{The first panel displays the initial datum $\theta_0(x)$ from \eqref{eq:NumericlInitial}. The second and third panels are the solution evaluated at time $T_1 = 5.625 \cdot 10^{-4}$ and $T_2 = 6.619 \cdot 10^{-4}$ just prior and after singularity formation.}
    \label{fig:BlowUp}
\end{figure}

\medskip

It may seem surprising that the blow-up time $T = 6.619 \cdot 10^{-4}$ is so small, especially given the numerical computations in Moffatt's paper \cite{Moffatt2015}, which observe no qualitative change of behaviour in the solution subsequent to viscosity driven relaxation, up to time $t = 50$ (for $\varepsilon = 0.001$). However, there are two points that should be kept in mind. First of all, the time variable in the PDE \eqref{ieq:AngleVariable} is rescaled $\tau = t / \varepsilon$, so that the blow-up happens, in original variables, at time $T/\varepsilon$. Secondly, this time is actually consistent with the form of the PDE, as a order of magnitude computation will show. Set again $\varphi = \partial_x \theta$, so that
\begin{equation*}
    \partial_\tau \varphi + u \partial_x \varphi - \partial_x^2 \varphi + \varphi \int_\T \varphi^2 = \varphi^3.
\end{equation*}
The blow-up is expected to be driven by the power $\varphi^3$, so that by ignoring all the other terms, we should capture (intuitively) the blow-up rate. We therefore consider the equation $\partial_\tau \varphi = \varphi^3$, whose integral is
\begin{equation*}
    \frac{-1}{2 \varphi^2} = \tau - \frac{1}{2 \varphi_0^2},
\end{equation*}
we see that the blow-up is expected to happen at time $\tau \approx \frac{1}{2 \varphi_0^2}$. In the case of our example \eqref{eq:NumericlInitial}, this is $\| \varphi_0 \|_{L^\infty} \approx 38.57$, so that the order of magnitude of the blow-up time should be $T^\star \approx 10^{-4}$. Therefore, our numerical illustration is not surprising. Likewise, similar order of magnitude computations from Figure 6 (c) of \cite{Moffatt2015} would suggest a value $\| \varphi_0 \|_{L^\infty} \approx 5$ and blow-up time of $\tau \approx 0.02$, or in other words, considering the given value $\varepsilon = 0.001$ given in \cite{Moffatt2015}, $t \approx 20$, which is the rough time scale of the simulations of Moffatt ($t \approx 50$). It is therefore quite possible that the simulations of \cite{Moffatt2015} do not observe any qualitative change in the solution due to blow-up of the limit system, although such a change would likely be observed by pushing the computations further.

\medskip

\textbf{Global Solutions.} As we have proved in Theorem \ref{t:existencelimit}, solutions are smooth for all times provided that the initial datum is small enough $\| \theta_0 \|_{H^{1/2}} \leq \eta_0$. This means that the datum $\lambda \theta_0(x)$ should give rise to a global solution if $\lambda > 0$ is small enough (where $\theta_0$ is given by \eqref{eq:NumericlInitial}). Indeed, by taking $\lambda = 1/3$, we cannot observe blow-up, and the solutions simply decays due to diffusion (see Figure \ref{fig:GlobalInitial}).

\begin{figure}[h!]
    \centering
    \includegraphics[width=0.3\linewidth]{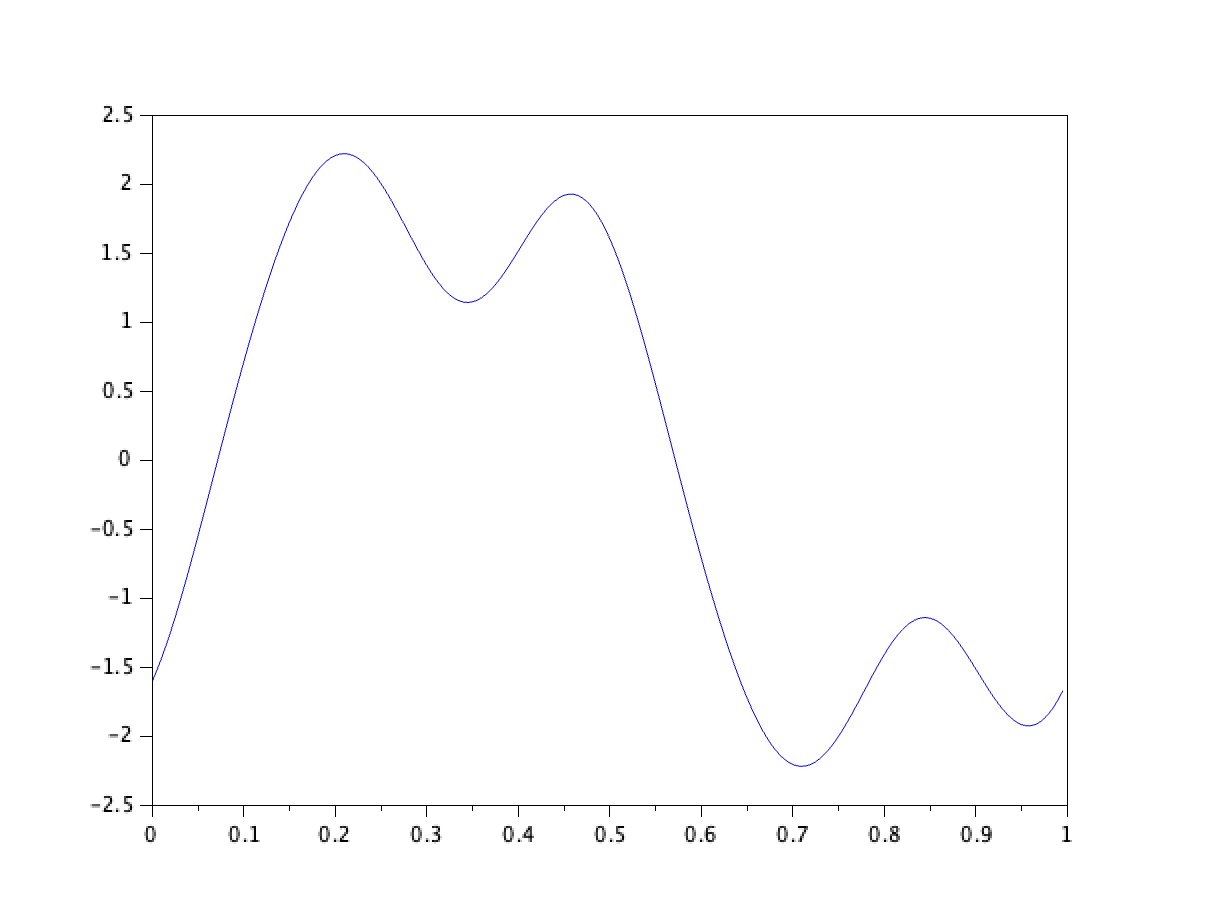}
    \caption{Initial datum $\lambda \theta_0(x)$, which gives rise to a global solution.}
    \label{fig:GlobalInitial}
\end{figure}

The computation was performed with mesh size $\delta x = 5 \cdot 10^{-3}$ and $\delta t = 2.5 \cdot 10^{-6}$ up to time $T = 10^{-2}$ (4000 iterations). See Figure \ref{fig:GlobalT}.

\begin{figure}[h!]
    \centering
    \includegraphics[width=0.9\linewidth]{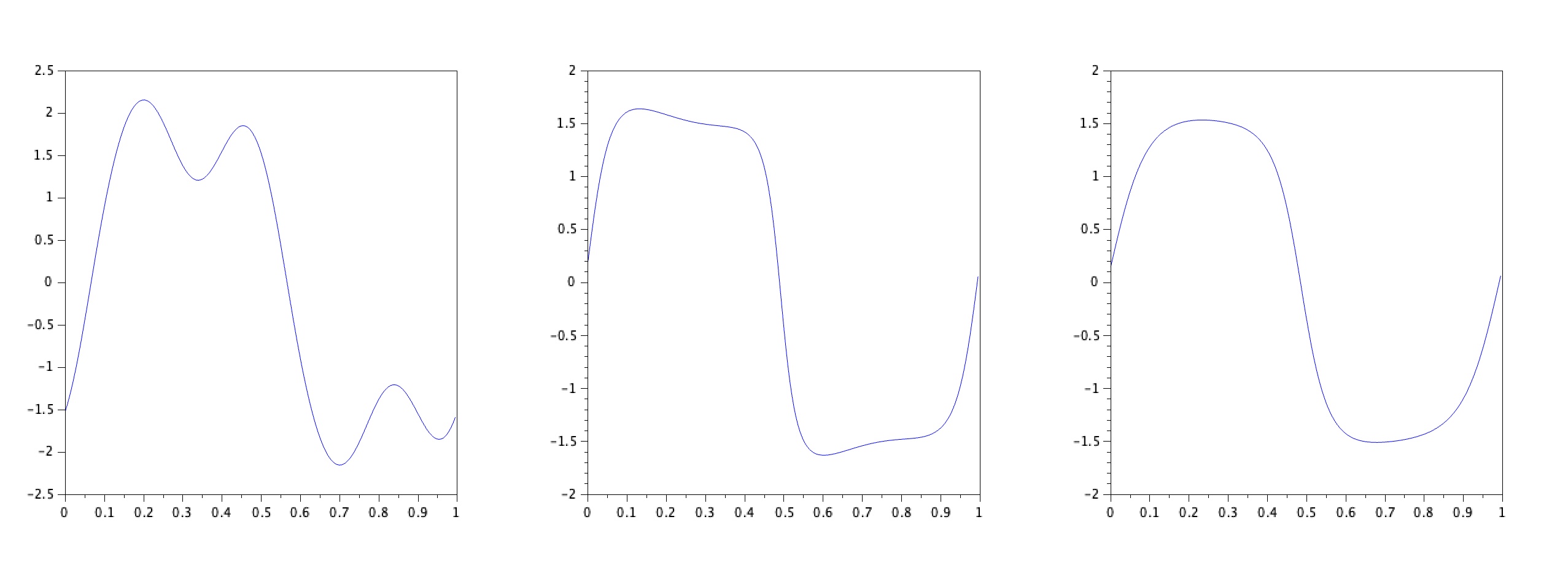}
    \caption{Numerical solution at times $T_1 = 2.5 \cdot 10^{-4}$, $T_2 = 5 \cdot 10^{-3}$ and $T_3 = 10^{-2}$. No formation of singularity is observed.}
    \label{fig:GlobalT}
\end{figure}

\medskip

\textbf{Initial Data with Small Oscillations.} Finally, as stated in Theorem \ref{t:smallOscillations}, solutions are still expected to be global when the quantity
\begin{equation*}
    \max_{x \in \T} \theta_0(x) - \min_{x \in \T} \theta_0(x)
\end{equation*}
is small enough, even if the norm $\| \theta_0 \|_{H^{1/2}}$ is large. We should expect to observe this behaviour in numerical solutions. We therefore consider the initial datum $\lambda \theta_0(x)+ \sin(40 \pi x)$ (see Figure \ref{fig:Oscillations}). The mesh size is $\delta x = 2.5 \cdot 10^{- 3}$ and $\delta t = 6.25 \cdot 10^{-7}$.

\medskip

As predicted by Theorem \ref{t:smallOscillations}, we observe that the term $\sin(40\pi x)$ does not prevent the solution from being global, since the high frequencies are very quickly dissipated (Figure \ref{fig:Oscillations}).

\begin{figure}[h!]
    \centering
    \includegraphics[width=0.9\linewidth]{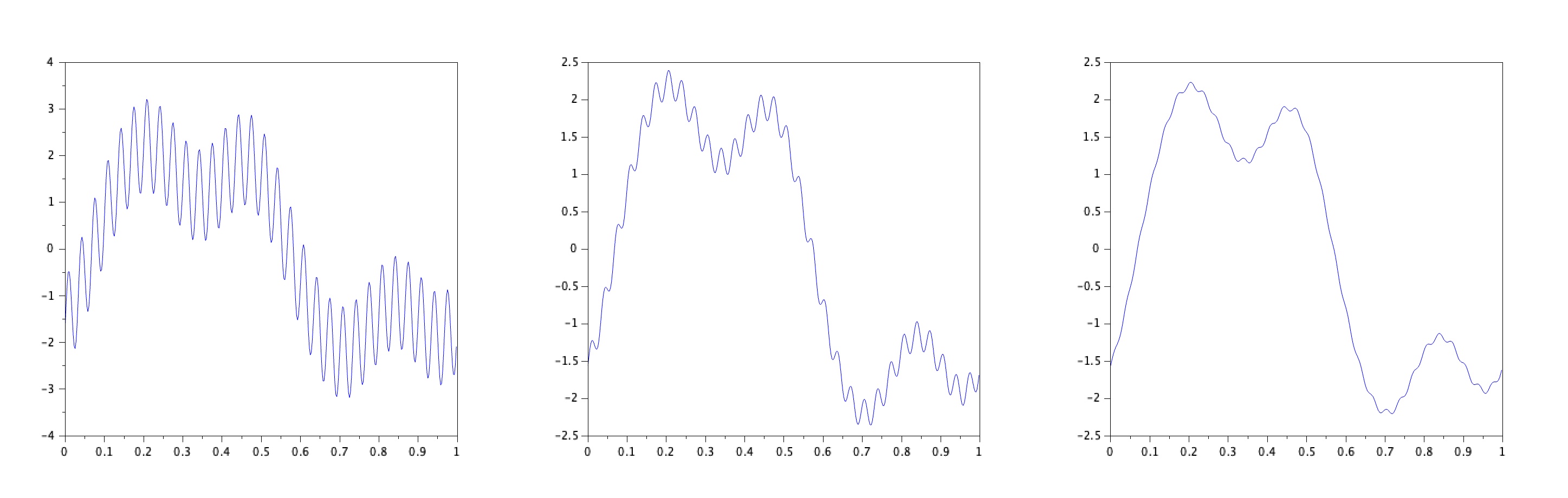}
    \caption{Initial datum with large $H^{1/2}$ norm but small enough $\theta_{\max} - \theta_{\rm min}$ give global solutions. On the left is the initial datum $\lambda \theta_0(x)+ \sin(40 \pi x)$, and then the solution evaluated at times $T_1 = 5 \cdot 10^{-5}$ and $T_2 = 9.375 \cdot 10^{-5}$.}
    \label{fig:Oscillations}
\end{figure}

\section{Further Discussion and Conclusion}

In this paper we prove rigorously some properties of the solutions of \eqref{eq:magneticrelaxation}, some of which are motivated by the  numerical solutions computed in \cite{Moffatt2015}. In particular, we show the existence of a two stage process, corresponding to a fast dynamics for $t$ of order $\log(\varepsilon^{-1})$, dominated by the viscosity, in which the equation behaves roughly like the solutions of \eqref{eq:hyperbolicSystem} and relaxes to a state of $|b|$ constant. Then, the resistivity starts influencing the system, entering a second stage of evolution up to times of order $O(\varepsilon^{-1})$. In this situation, the magnetic field remains in a state of $|b(\tau,\cdot)|$ roughly constant, and the dynamics of the phase is given by the effective equations \eqref{eq:limitsystem}-\eqref{eq:limitsystem2}.

\medskip

The effective equations which we uncover allow to understand the dynamics of solutions on a large time scale $t \sim 1/\varepsilon$, and our findings confirm the numerical observations of Moffatt, in particular equation (4.3) of \cite{Moffatt2015} which is equivalent to the second equation in \eqref{eq:AngleVariablePDE} above.

\medskip

However, we also show that the solutions of the effective equations \eqref{eq:limitsystem}-\eqref{eq:limitsystem2} might blow up in finite time, and this should have an observable effect on the solutions of the primitive equations \eqref{eq:magneticrelaxation}. However, this phenomenon happens on a large time scale $t \sim 1/\varepsilon$, so it will only be noticeable in numerical simulations that run for a very long time. For example, in \cite{Moffatt2015} simulation go up to time $t = 50$ with $\varepsilon = 0.001$, so our arguments from Subsection \ref{sec:numerics} would suggest that qualitative change of behaviour of solutions could occur just outside the range of simulation.

\medskip

Many questions remain unanswered concerning the blow-up of the effective system, and what its significance for solutions of the primitive equations \eqref{eq:magneticrelaxation}, for example giving a characterization of such blow up points, their asymptotic behaviour, as well as the search for an effective model that remains valid at and after the time of blow-up of \eqref{eq:limitsystem}-\eqref{eq:limitsystem2}. This is ongoing work.

\medskip
Furthermore, in view of the complex behaviour we observe in spite of the simplicity of this model, the authors think that it is worth to study more complete systems, in which we take into account some of the quantities we disregarded in the derivation of \eqref{eq:magneticrelaxation}, such as the pressure or the density. This will also be the topic of future research.

\section{Acknowledgments}
All three authors have been funded by the Deutsche Forschungsgemeinschaft (DFG)
through the collaborative research centre ``The mathematics of emerging effects'' (CRC 1060, Project-ID. 211504053)

D. Sánchez-Simón del Pino and J. J. L. Velázquez have been also funded by the Deutsche Forschungsgemeinschaft (DFG, German
Research Foundation) under Germany’s Excellence Strategy - GZ 2047/1, Projekt-ID 390685813.

D. Sánchez-Simón del Pino further acknowledges the support of the Bonn International Graduate School of Mathematics (BIGS) at the Hausdorff Centre for Mathematics.

\newpage



\begin{thebibliography}{xxx}

\bibitem{Aubin1963 }Aubin, Jean-Pierre: 
\textit{Un théorème de compacité} (French).
C. R. Acad. Sci. Paris. Vol. 256., 1963, pp. 5042–5044

\bibitem{BKS} H. Bae, H. Kwon, J. Shin:
\textit{Global solutions to Stokes-Magneto equations with fractional dissipations}.
arXiv:2310.03255 (2023).

\bibitem{BCD} H. Bahouri, J.-Y. Chemin and R. Danchin:
\textit{``Fourier analysis and nonlinear partial differential equations''}.
Grundlehren der Mathematischen Wissenschaften (Fundamental Principles of Mathematical Sciences), Springer, Heidelberg, 2011.

\bibitem{BBV} R. Beekie, T. Buckmaster and V. Vicol:
\textit{Weak Solutions of Ideal MHD Which Do Not Conserve Magnetic Helicity}. 
Ann. PDE 6, 1 (2020).

\bibitem{BFV} R. Beekie, S. Friedlander and V. Vicol:
\textit{On Moffatt's magnetic relaxation equations}.
Comm. Math. Phys. 390 (2022), no. 3, 1311--1339.

\bibitem{Bellan} P. M. Bellan:
\textit{``Fundamentals of Plasma Physics''}.
Cambridge University Press, Cambridge NY, 2006.

\bibitem{BT} D. W. Boutros and E. S. Titi: 
\textit{On the conservation of helicity by weak solutions of the 3D Euler and inviscid MHD equations}.
arXiv:2410.00813.

\bibitem{BF2013} Boyer, F., Fabrie, P.: 
\textit{Mathematical Tools for the Study of the Incompressible Navier-Stokes
Equations and Related Models},
Applied Mathematical Sciences, vol. 183, Springer, New York,
2013.


\bibitem{Brezis} H. Brezis
\textit{Functional Analysis, Sobolev Spaces and Partial Differential Equations}.
Springer

\bibitem{CobbThese} D. Cobb:
\textit{Étude mathématique de fluides en interaction avec un champ magnétique}. PhD dissertation, Université de Lyon 1, 2022.

\bibitem{Cobb2024} D. Cobb:
\textit{Bounded solutions in incompressible hydrodynamics}.
J. Funct. Anal., Vol. 286, 5, 2024, 110290.

\bibitem{Knobel1} M. Dolce N. Knobel and C. Zillinger:
\textit{Large norm inflation of the current in the viscous, non-resistive magnetohydro-dynamics equations.}
\textit{arXiv preprint} \url{https://arxiv.org/pdf/2410.22804}


\bibitem{FL2020} D. Faraco and S. Lindberg:
\textit{Proof of Taylor's Conjecture on Magnetic Helicity Conservation}. 
Commun. Math. Phys. 373, 707--738 (2020).

\bibitem{KK} H. Kim and H. Kwon:
\textit{Global existence and uniqueness of weak solutions of a Stokes-Magneto system with fractional diffusions}.
arXiv:2302.02046 (2023).


\bibitem{Knobel2} N. Knobel and C. Zillinger:
\textit{On the Sobolev stability threshold for the 2D MHD equations with horizontal magnetic dissipation},
Journal of Nonlinear Science 35 (3), 1-36

\bibitem{Knobel3} N. Knobel: 
\textit{Sobolev stability for the 2D MHD equations in the non-resistive limit.}
\textit{arXiv preprint }\url{https://arxiv.org/pdf/2401.12548}

\bibitem{Krylov1996} N. V. Krylov:
\textit{Lectures on Elliptic and Parabolic Equations in H\"older Spaces }.
Graduate Studies in Mathematics, vol. 12, 1996, American Mathematical Society.

\bibitem{Moffatt1969} H. K. Moffatt:
\textit{The degree of knottedness in tangled vortex lines}.
J. Fluid Mech. 35, 1969, pp. 117--129.

\bibitem{Moffatt1985} H. K. Moffatt:
\textit{Magnetostatic equilibria and analogous Euler flows of arbitrarily complex topology. Part 1. Fundamentals}.
J. Fluid Mech., vol. 159 , October 1985 , pp. 359 -- 378.

\bibitem{Moffatt2015} H. K. Moffatt:
\textit{Magnetic relaxation and the Taylor conjecture}.
J. Plasma Phys. vol. 81, 2015, 905810608.

\bibitem{Moffatt2016} H. K. Moffatt:
\textit{Helicity and celestial magnetism}. Proc. R. Soc. A. 472, (2016), 20160183.

\bibitem{Moffatt2021} H. K. Moffatt: 
\textit{Some topological aspects of fluid dynamics}.
J. Fluid Mech. , Vol. 914. (Special JFM volume in celebration of the George K. Batchelor centenary), 10 May 2021 , P1. 

\bibitem{MG} R. Mukherjee and R. Ganesh:
\textit{Numerical relaxation of a 3D MHD Taylor - Woltjer state subject to abrupt expansion}.
arXiv:1811.09803v1.

\bibitem{QLLS} H. Qin, W. Liu, H. Li and J. Squire:
\textit{Woltjer-Taylor state without Taylor's conjecture: plasma relaxation at all wavelengths}.
Phys. Rev. Lett. 109 (23), 2012, 235001.

\bibitem{SCTSDB} C. B. Smiet, S. Candelaresi, A. Thompson, J. Swearngin, J. W. Dalhuisen and D. Bouwmeester: 
\textit{Self-organizing knotted magnetic structures in plasma}. 
Phys. Rev. 115, 095001.

\bibitem{Tan} J. Tan:
\textit{Weak solutions of Moffatt's magnetic relaxation equations}.
arXiv:2311.18407

\bibitem{Taylor1974} J. B. Taylor:
\textit{Relaxation of toroidal plasma generation of reverse magnetic fields}.
Phys. Rev. Lett. 33, 1974, pp. 1139--1141.

\bibitem{Taylor1986} J. B. Taylor:
\textit{Relaxation and magnetic reconnection in plasmas}.
Rev. Mod. Phys. 58, 1986, pp. 741--736.

\bibitem{Woltjer} L. Woltjer:
\textit{A theorem on force-free magnetic fields}.
Proc. Natl Acad. Sci. USA, 44, 1958, pp. 489--491.

\bibitem{YK} K. Yamamoto and A. Kageyama:
\textit{MHD Relaxation with Flow in a Sphere}.
Procedia Computer Science, vol. 80, 2016, pp. 1374--1381.

\bibitem{Yeates} A. R. Yeates
``Magnetic relaxation theory'',
\textit{in} D. MacTaggart and A. Hillier, \textit{Topics in Magnetohydrodynamic Topology, Reconnection and Stability Theory}.
CISM Courses and Lectures, Springer International Publishing.

\bibitem{YRH} A. R. Yeates, A. J. B. Russell and G. Hornig: 
\textit{Evolution of field line helicity in magnetic relaxation}. 
Phys. Plasmas 1 August 2021, 28 (8), 082904.




\end{thebibliography}
\end{document}